\documentclass{amsart}
\usepackage{fancyhdr}
\usepackage{amsaddr}
\usepackage{amssymb,amsfonts}
\usepackage[all,arc]{xy}
\usepackage{enumerate}
\usepackage{mathrsfs}
\usepackage{hyperref}
\usepackage{ulem}
\hypersetup{
pdftitle={A combinatorial one-cocycle in a moduli space of knots from the Vassiliev invariant of order 3},
pdfsubject={Knot theory, Topology},
pdfauthor={Baptiste Gros, Butian Zhang},
pdfkeywords={Gauss diagram, Gauss diagram formula, knot theory, combinatorial one-cocycle}
}
\usepackage{bm}
\usepackage{pgfplots}
\usepackage{graphicx} 
\usepackage{float} 
\usepackage{subcaption} 

\usepackage{tikz,tikz-cd}
\usepackage{geometry}
\usetikzlibrary{decorations,arrows}
\usetikzlibrary{decorations.markings}
\usetikzlibrary{positioning}

\usepackage[T1]{fontenc}         

\usepackage{amsmath}

\usepackage{pict2e}

\newtheorem{thm}{Theorem}[section]

\newtheorem{prop}[thm]{Proposition}
\newtheorem{lem}[thm]{Lemma}

\theoremstyle{definition}
\newtheorem{defn}[thm]{Definition}

\newtheorem{exmp}[thm]{Example}
\newtheorem{exmps}[thm]{Examples}

\theoremstyle{remark}
\newtheorem{rem}[thm]{Remark}

\makeatletter
\let\c@equation\c@thm
\makeatother
\numberwithin{equation}{section}

\title{A combinatorial one-cocycle in a moduli space of knots from the Vassiliev invariant of order 3}

\author{Baptiste Gros}
\address{Département de Mathématiques, ENS de Lyon}
\email{baptiste.gros@ens-lyon.fr }
\author{Butian Zhang}
\address{Institut de Mathématiques de Toulouse, Université Paul Sabatier}
\email{butian.zhang@math.univ-toulouse.fr}

\makeatletter
\g@addto@macro{\endabstract}{\@setabstract}
\newcommand{\authorfootnotes}{\renewcommand\thefootnote{\@fnsymbol\c@footnote}}%
\makeatother

\begin{document}


\maketitle
\begin{abstract}

The theory of Gauss diagrams and Gauss diagram formulas provides convenient ways to compute knot invariants, such as coefficients of the HOMFLYPT polynomial. In \cite{4,5}, the author uses Gauss diagram formulas to find combinatorial 1-cocycles in the moduli space of knots in the solid torus. Evaluated on canonical loops, one can then obtain new, non trivial knot invariants. In those books, the author conjectures that a new formula, based on the Vassiliev invariant $v_3$ also gives a 1-cocycle. We prove that it is in fact true by using the same methods developed by the author in those books.

\end{abstract}

\tableofcontents
\section{Introduction}\label{s1}

The goal of this article is to prove a conjecture (Question 2.1 in \cite{5}) formulated by Thomas Fiedler, stating that a formula based on the Vassiliev invariant $v_3$ defines a 1-cocycle in a moduli space that will be specified later. We start by giving a brief explanation on Gauss diagrams and Gauss diagram formulas, and we give the Gauss diagram formula for $v_3$ in Section \ref{s2}. For more details on this topic, see \cite{2}. We then give a summary of some of the ideas developped by Thomas Fiedler in \cite{4,5}, namely we define the moduli space we work in, we explain which computations need to be made in order to that a formula defines a 1-cocycle, and we define the 1-cocycle that is the topic of this article in Section \ref{s3}. The next sections \ref{s4},\ref{s5},\ref{s6} and\ref{s7} contain the details of these equations. Finally, we give an example on how to compute this cocycle in Section \ref{s8}.

\section{Gauss diagram and Gauss diagram formula}\label{s2}
Given a knot diagram $D$, we can also represent this knot by using \textit{Gauss diagram}. If we choose a base point at the knot diagram $D$, we go along the orientation of this knot, we will finally come back to our base point. Essentially, we just go along a circle, which we always assume to possess an anti-clockwise orientation, starting at a base point and ending at the first time that we come back to the base point. However, on this circle, we meet some crossings. We use arrows to represent these crossing. The small arc containing the foot of an arrow represents the piece of arc in the neighbourhood of the corresponding crossing which is at the bottom. And the small arc containing the head of the arrow represents the piece of arc in the neighbourhood of the corresponding crossing which is on the top. And we label the sign of each crossing on the corresponding arrow.

\begin{exmp}
Let $D$ be a trefoil and $G$ be the Gauss diagram of the trefoil. They are displayed as the following: 
\begin{figure}[H] 
\begin{subfigure}{0.4\textwidth}
	\includegraphics[width=0.8\textwidth]{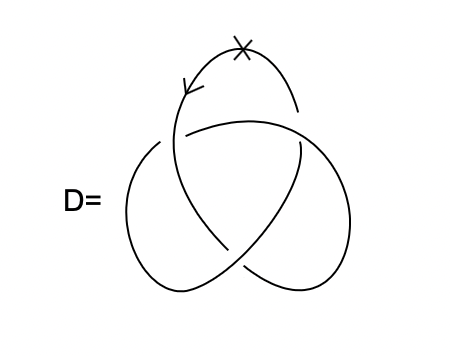} 
	\caption{trefoil} 
	\label{Fig.trefoil_l} 
\end{subfigure}
\begin{subfigure}{0.4\textwidth}
	\includegraphics[width=0.8\textwidth]{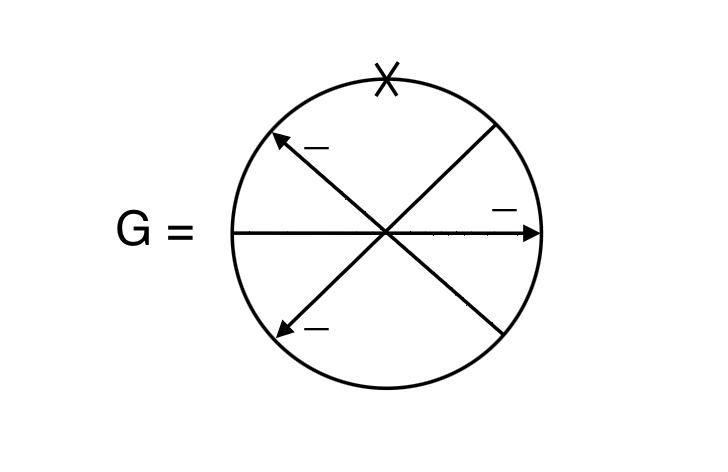} 
	\caption{Gauss diagram of the trefoil}
	\label{Fig.Gauss_trefoil} 
\end{subfigure}
\caption{trefoil and its Gauss diagram}
\label{Fig.example_Guass}
\end{figure}
\end{exmp}

\begin{rem}
Different authors may use different directions for the arrows in Gauss diagrams. For example, the direction we adopt here is consistent with that in \cite{3,4,5} contrast to that in \cite{1,2}.
\end{rem}

The Reidemeister moves for knot diagrams can also be demonstrated in their Gauss diagrams.

\begin{figure}[H] 
\centering 
\begin{subfigure}{0.6\textwidth}
	\includegraphics[width=\textwidth]{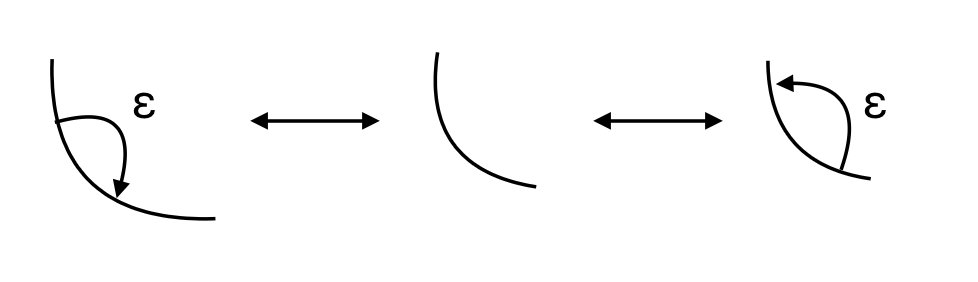} 
	\caption{Reidemeister move \uppercase\expandafter{\romannumeral1}} 
	\label{Fig.R1} 
\end{subfigure}
\begin{subfigure}{0.6\textwidth}
	\includegraphics[width=\textwidth]{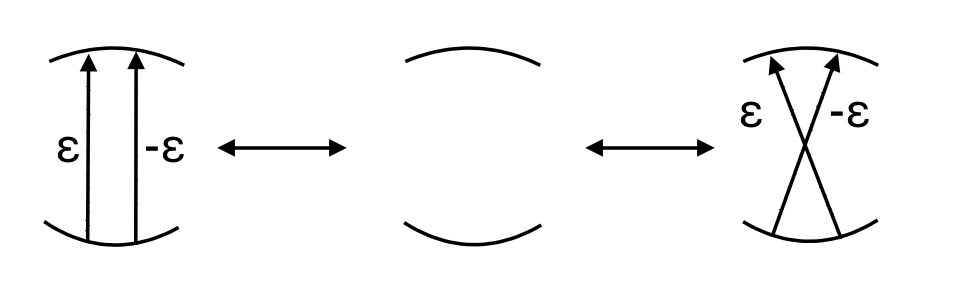} 
	\caption{Reidemeister move \uppercase\expandafter{\romannumeral2}}
	\label{Fig.R2} 
\end{subfigure}
\begin{subfigure}{0.6\textwidth}
	\includegraphics[width=\textwidth]{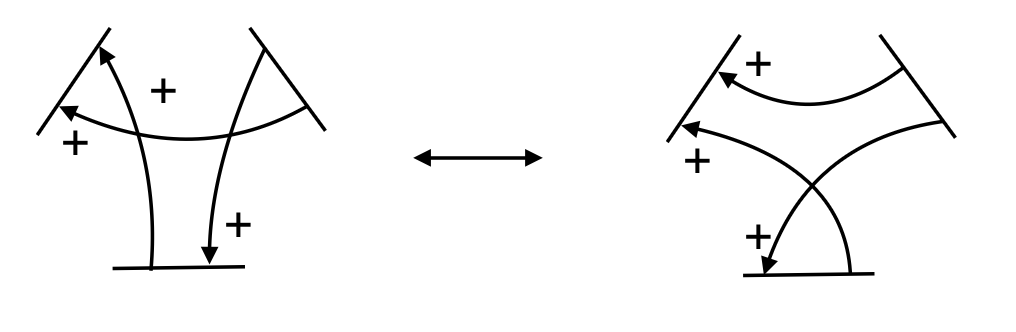} 
	\caption{Reidemeister move \uppercase\expandafter{\romannumeral3}}
	\label{Fig.R3} 
\end{subfigure}
\begin{subfigure}{0.6\textwidth}
	\includegraphics[width=\textwidth]{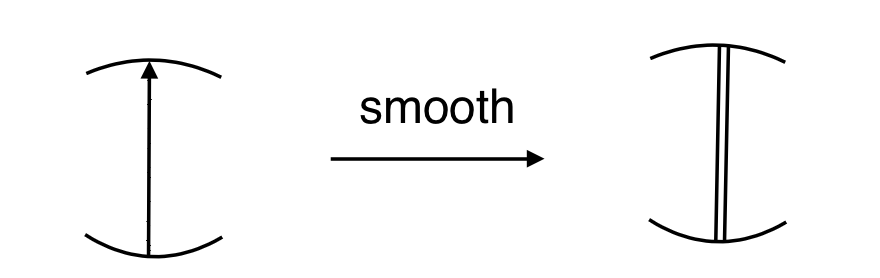} 
	\caption{smooth a crossing}
	\label{Fig.smooth_Gauss}
\end{subfigure}
\begin{subfigure}{0.6\textwidth}
	\includegraphics[width=\textwidth]{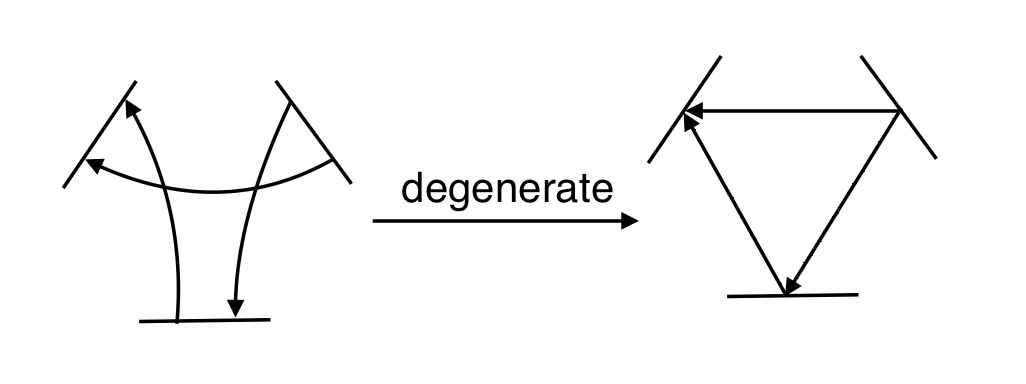} 
	\caption{degenerated Reidemeister move \uppercase\expandafter{\romannumeral3}}
	\label{Fig.degenerated_Gauss}
\end{subfigure}
\caption{Operations on Gauss diagram}
\label{Fig.Reidemeister_Gauss}
\end{figure}

\begin{rem}
\begin{enumerate}
	\item Reidemeister move \uppercase\expandafter{\romannumeral3} may have different kinds of shapes. For more details, see Section \ref{s5}.
	\item The arcs above just indicate the locally direction as if they are in one circle with an anti-clockwise direction. However, it may happen that these arcs come from different circles. And in the figure \ref{Fig.R3} the three pieces of arcs may not pass globally anti-clockwise as if they were a part of one circle. 
\end{enumerate}
\end{rem}

More generally, we can define a \textit{(signed) arrow diagram} which is an ordered sequence of circles with signed arrows from a point on one circle to another point on one circle. For an arrow diagram with one circle, it cannot always be realised by a knot. For example, the following arrow diagram is such a one. 
\begin{figure}[H]
\center
\includegraphics[width=0.15\textwidth]{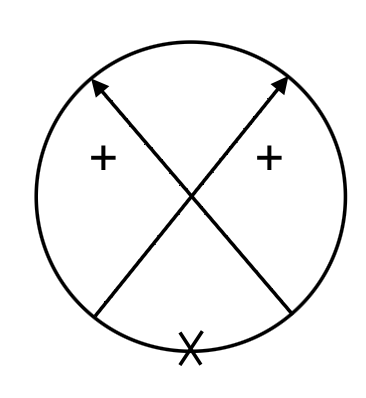}
\caption{an arrow diagram which cannot be realised by a knot}
\label{Fig.arrow_diagram}
\end{figure}

\begin{defn}
Two arrow diagrams with $m$ circles are the \textit{equivalent} if there is a 1-1 correspondence between the arrows in each of them such that the sequence of the head and foot of each arrows in one arrow diagram coincides with that sequence under the correspondence in the other arrow diagram. 

For two arrow diagrams, $A$ and $B$, with  $m$ circles, we say that $B$ is a \textit{subdiagram} of $A$ if after deleting some arrow(s) in $A$ we get an equivalence of $B$. In this case, we denote it by $B\subset A$.
\end{defn}

Let $\mathscr{A}$ be the free abelian group generated all the (equivalence classes of) arrow diagrams with $m$ circles. For two arrow diagram with $m$ circles, $A$ and $B$, we define their \textit{scalar-product} by 
\begin{equation}
 (A,B) = \left\{
\begin{array}{rcl}
1 && {\text{$A$ and $B$ are equivalent}}\\
0 && {\text{otherwise}}
\end{array}
\right.
\end{equation}
And we extend this scalar product to the whole $\mathscr{A}$ by linearity. 
We define an endomorphism $I: \mathscr{A} \rightarrow \mathscr{A}$ by $I(A) = \sum\limits_{ C \subset A}C$. This notation means that we sum on all subsets of arrows of $A$, and not just on every diagram $C$ that happens to verify $C \subset A$.

Given a fixed arrow diagram $A$ of $m$ circles, we define $$\langle A,G\rangle = (A,I(G))$$ for any arrow diagram $G$ with $m$ circles. In words, this product counts the number of times $A$ appears in $G$.

We can also view such a product in another perspective (see \cite{1}). 

\begin{defn}
For two signed arrow diagrams with $m$ circles $A$ and $G$, a \textit{homomorphism} $\phi : A\rightarrow G$ is a homeomorphism from the $m$ circles of $A$ to the $m$ circles of G which respects the orientations, and sends the arrows in $A$ to arrows in $G$, preserving their directions and signs. The set of homomorphisms from $A$ to $G$ is denoted by $Hom(A,G)$. The set of arrows in $G$ which are images of arrows in $A$ under $\phi$ is denoted by $S(\phi)$. And we denote $sign(\phi)=1$.
\end{defn}

With such a definition, we can interpret our product as:
\begin{equation}
\begin{aligned}
\langle A,G\rangle &= \sum\limits_{\phi \in Hom(A,G)}1\\
&= \#Hom(A,G)\\
&= \sum\limits_{\phi \in Hom(A,G)}sign(\phi).
\end{aligned}
\end{equation}

It is often convenient to consider \textit{unsigned arrow diagrams} which are simply the linear combination of all the corresponding signed arrow diagrams with coefficient being the product of the all the signs in each arrow diagram. (Hence if there are $n$ arrrows, there will be $2^n$ terms.)
See examples in figure \ref{Fig.C2_formula}.

\begin{defn}
For an unsigned arrow diagram $A$ with $m$ circles and a signed arrow diagram $G$ with $m$ circles, a \textit{homomorphism} $\phi : A\rightarrow G$ is a homeomorphism from the $m$ circles of $A$ to the $m$ circles of G which respects the order sends the arrows in $A$ to arrows in $G$ preserving their directions. The set of homomorphisms from $A$ to $G$ is also denoted by $Hom(A,G)$. The set of arrows in $G$ which are images of arrows in $A$ under $\phi$ is also denoted by $S(\phi)$. The \textit{sign} of the homomorphism $\phi$ denoted by $sign(\phi)$ is defined by the product of all signs or arrows in $S(\phi)$. 
\end{defn}

In this case, we can see that 
\begin{equation}
\begin{aligned}
\langle A,G\rangle &= \sum\limits_{\phi \in Hom(A,G)}sign(\phi)
\end{aligned}
\end{equation}

\begin{defn}
A \textit{Guass diagram formula} is a map $\mathcal{I}_A: \mathscr{A} \rightarrow \mathbb{Z}$ defined by $$\mathcal{I}_A(G) = \langle A,G\rangle,$$
where $A$ is a fixed element in $\mathscr{A}$. 
\end{defn}
\begin{rem}
Sometimes, we just call $A$ a Gauss diagram formula and each signed or unsigned arrow diagram appearing in $A$ the \textit{configurations}. 
\end{rem}

\begin{exmps}
Let m = 1.
\begin{enumerate}
	\item The coefficient $p_{0,2}$ of HOMFLYPT polynomial is given by $p_{0,2}=\mathcal{I}_{A_{0,2}}$ (see \cite{3}), where $A_{0,2}$ equals to : 
		\begin{figure}[H]
		\center
		\includegraphics[width=0.7\textwidth]{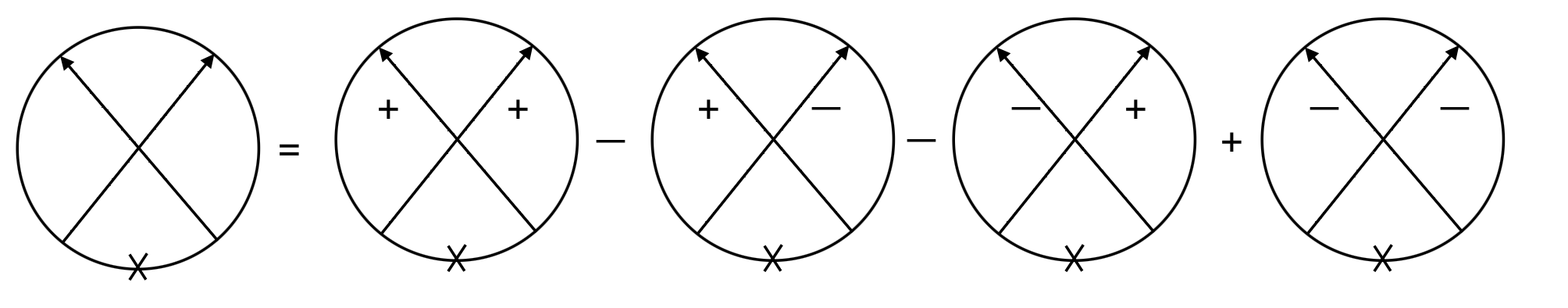}
		\caption{arrow diagrams giving $p_2$}
		\label{Fig.C2_formula}
		\end{figure}
	\item The coefficient $p_{1,2}= -2v_3$ of HOMFLYPT polynomial is given by $v_3=\mathcal{I}_A$(see \cite{2}), where $A$ equals to :
		\begin{figure}[H]
		\center
		\includegraphics[width=0.7\textwidth]{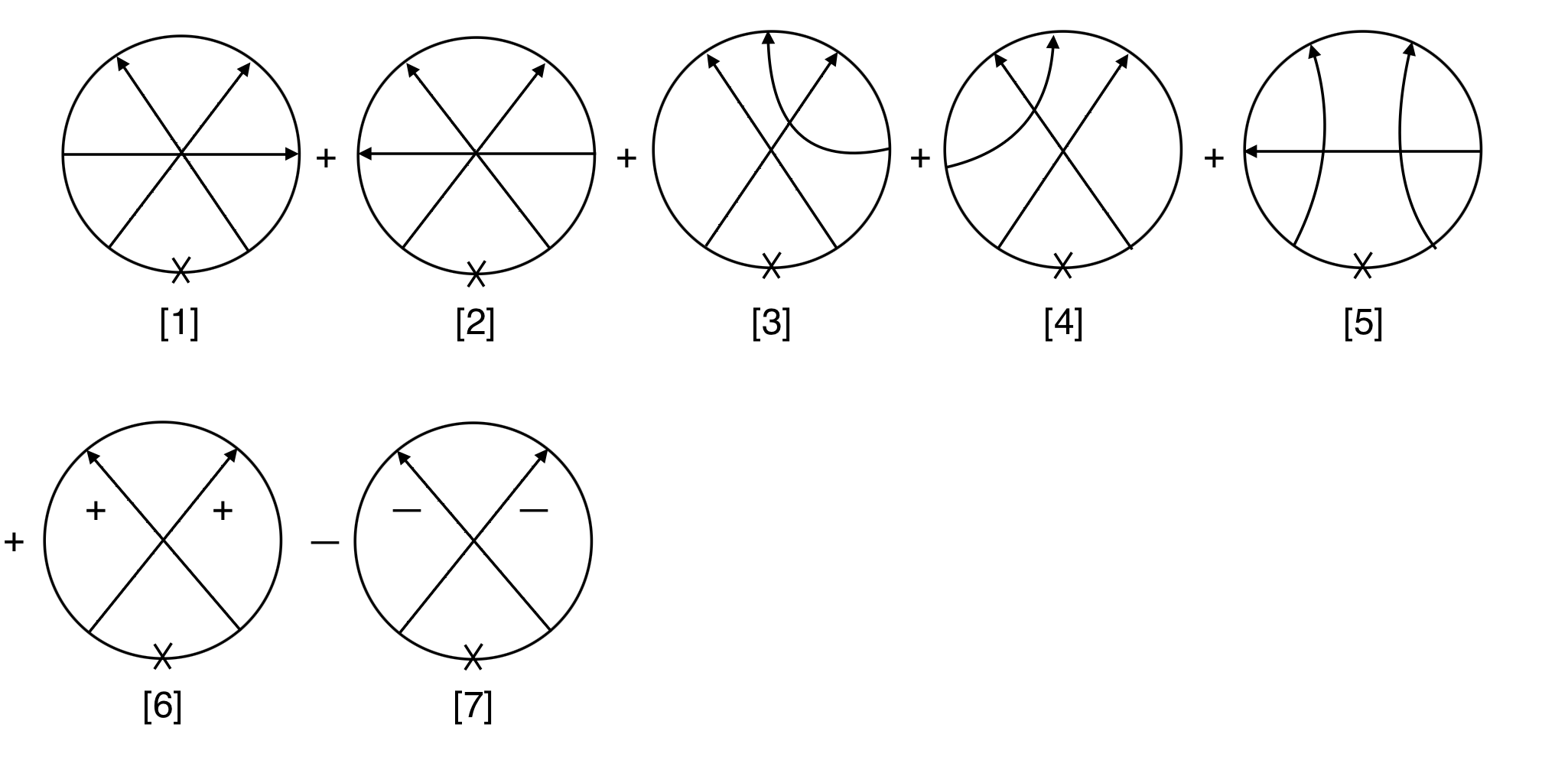}
		\caption{arrow diagrams giving $v_3$}
		\label{Fig.v3_formula}
		\end{figure}
\end{enumerate}
\end{exmps}

\begin{rem}
The Gauss diagram formula above for $v_3$ is the principal formula that we use in this paper. Our one-cocycle is derived from it. 
\end{rem}

\begin{prop}
For each of two pairs of arrow diagrams $A$ and $B$ below, we have $\mathcal{I}_A(G) = \mathcal{I}_{B}(G)$ for any Gauss diagram $G$ with one circle.
\begin{enumerate}
	\item\label{id5} \ 
		\begin{figure}[H]
		\begin{subfigure}{0.3\textwidth}
			\includegraphics[width=0.6\textwidth]{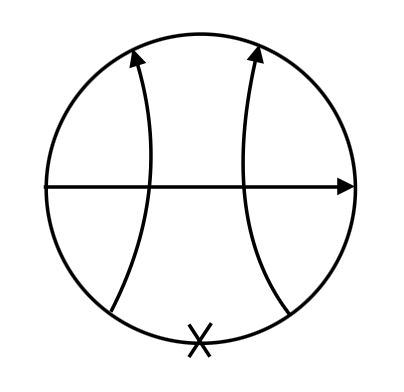}
			\caption{}
			\label{Fig.id5_l}
		\end{subfigure}
		\begin{subfigure}{0.3\textwidth}
			\includegraphics[width=0.65\textwidth]{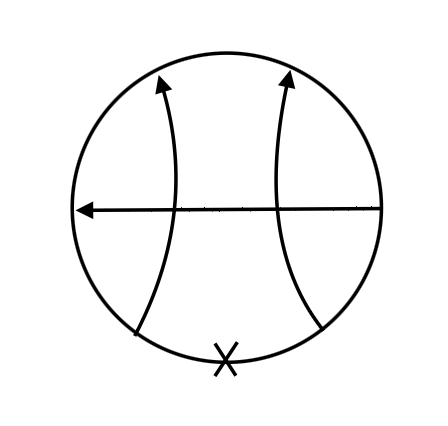}
			\caption{}
			\label{Fig.id5_r}
		\end{subfigure}
		\end{figure}
	\item\label{id1} \ 
		\begin{figure}[H]
		\begin{subfigure}{0.3\textwidth}
			\includegraphics[width=0.6\textwidth]{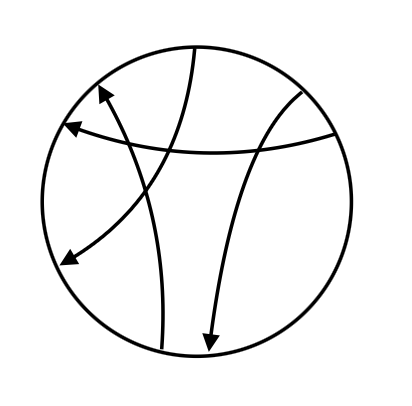}
			\caption{}
			\label{Fig.id1_l}
		\end{subfigure}
		\begin{subfigure}{0.3\textwidth}
			\includegraphics[width=0.63\textwidth]{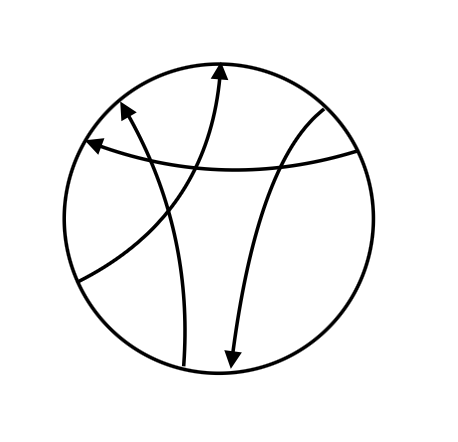}
			\caption{}
			\label{Fig.id1_r}
		\end{subfigure}
		\end{figure}
\end{enumerate}
\end{prop}
\begin{proof}
For case \ref{id5}, we first try to find the two vertical arrows in $G$. To find the third horizontal arrow, we first smooth the two vertical arrows  in $A$, $B$, and $G$. Then the third horizontal arrow just counts the linking number of the two separate circles (see figure \ref{Fig.prf_id5}). $A$ and $B$ correspond to the two different manners of counting it. Hence $\mathcal{I}_A(G) = \mathcal{I}_{B}(G)$. 
\begin{figure}[H]
\center
\includegraphics[width=0.5\textwidth]{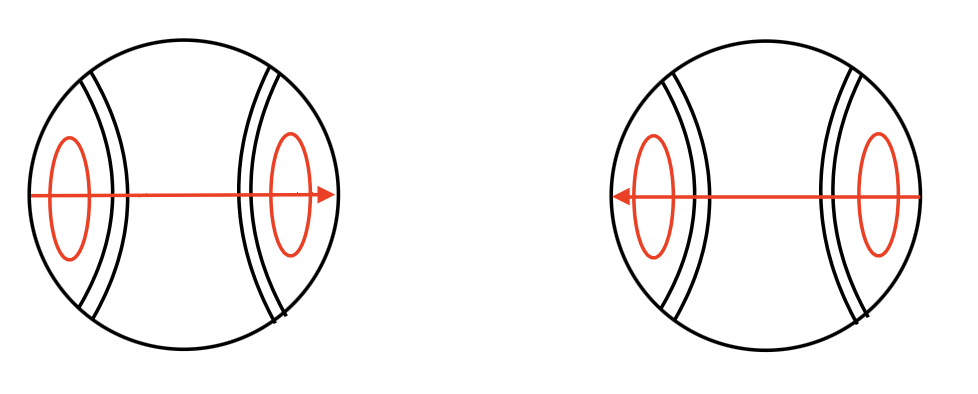}
\caption{linking number of the two circles}
\label{Fig.prf_id5}
\end{figure}
Case \ref{id1} is similar. We first find the three arrows representing Reidemeister move \uppercase\expandafter{\romannumeral3}. Then we smooth all of them. 
\begin{figure}[H]
\center
\includegraphics[width=0.5\textwidth]{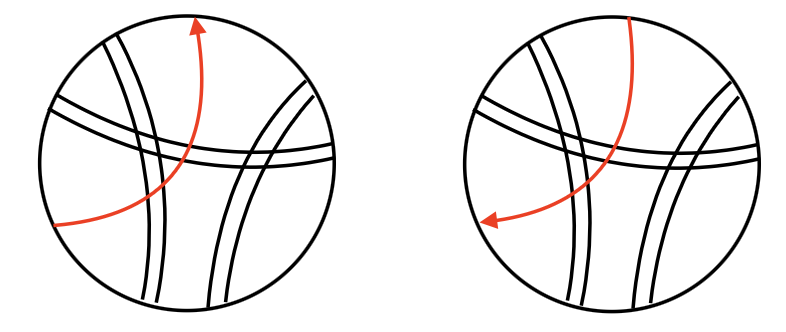}
\caption{linking number of the two circles}
\label{Fig.prf_id1}
\end{figure}
\end{proof}

\section{One-cocycle}\label{s3}

We will now briefly remind the methods used in \cite{4} and \cite{5}, which contain more precise explanations . The goal is to study the topology of a moduli space of knots, in order to construct 1-cocycles, and evaluate them in some loops, the result being a knot invariant. 

In order to define the moduli space that will be used, we need some definitions. First of all, we will consider knots in the solid torus $V$, embedded in $\mathbb{R}^3$. The set of all knots in the solid torus is endowed with the Whitney topology. We fix a projection $\mbox{pr} : \mathbb{R}^3 \to \mathbb{R}^2$ which sends the solid torus onto an annulus. The solid torus verifies $H_1(V,\mathbb{Z})  \simeq \mathbb{Z}  $ and we fix such an isomorphism so that counter-clockwise loops in the plane are positive. Then we restrain the moduli space to knots which homology class is $n$, and for which the projection into the annulus is an immersion. 

A crossing $q$ of a knot diagram of a knot in the plane can be smoothed as usual, and we note $q^+$ the knot in the smoothed diagram visiting $q$ from its lower strand to its higher strand. We also note $q^-$ the knot visiting $q$ from its higher to its lower strand. And we say that a knot has no \textit{negative loops} if after smoothing some crossings all the loops in the diagram have non-negative homology. We further restrict the moduli space to those knots that have no negative loops.

Finally, we fix a disc $D^1 \times \{0\} \subset D^1 \subset S^1$ in the solid torus, which we call \textit{disc at infinity}.  Any knot in the moduli space intersects this disc $n$ times, and the lowest intersection with respect to $\mbox{pr}$ is called the \textit{point at infinity}, noted $\infty$. We say that a knot verifies the \textit{separation condition} if a crossing can not move over the point $\infty$ at the same moment as a Reidemeister III move happens somewhere in the diagram. We lastly restrict the moduli space to the knots which verify the separation condition. This moduli space is noted $M_n$.\\

In this space, for any knot $K$, there are canonical loops based in $K$, (namely Gramain's loop, Fox-Hatcher's loop, and the "push" loop), so that an isotopy of knots gives an homotopy of the corresponding loops. Then composing a 1-cocycle with one of these canonical loops gives a knot invariant. For now on we focus on defining one candidate for a 1-cocycle, which will simply be noted $R$ later, and proving that it is indeed a 1-cocycle, which answers a conjecture in \cite{5} (Question 2.1). We use the same methods as in this book.

The goal is to show that if $\gamma$ and $\gamma '$ are two loops generical homotopic in $M_n$, $R(\gamma) = R(\gamma')$. An homotopy is generic if it does not intersect strata of codimension greater than 2, intersect strata of codimension 2 transversally
and it is tangential to strata of codimension 1 only in ordinary tangencies in the moduli space. In our case, one of the form :

\begin{enumerate}
\item\label{singularity_1} \includegraphics[height=2cm]{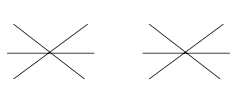} Two Reidemeister moves happening at the same time. A singularity of this form will be noted $\Sigma^1 \cap \Sigma^1$.
\item\label{singularity_2} \includegraphics[height=2cm]{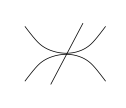} An ordinary tangency with a branch passing transversely through the double point. It is noted $\Sigma^2_{\mbox{trans-self}}$.
\item\label{singularity_3} \includegraphics[height=2cm]{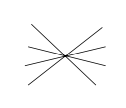}An ordinary quadruple point, meaning a quadruple a quadruple point where all intersections are transverse; it is noted $\Sigma^2_{\mbox{quad}}$.
\item\label{singularity_4} \includegraphics[height=1.5cm]{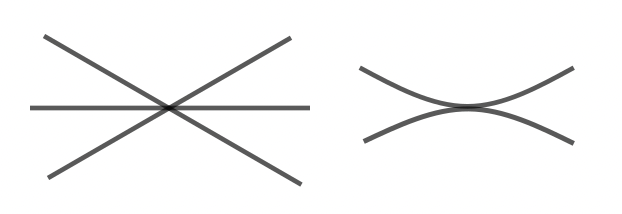} A Reidemeister move II and a Reidemeister move III at the same time.
\end{enumerate}

In general, the projection $\mbox{pr}$ gives a stratification of the moduli space, and the points where only one of those singularity happens forms a stratum of codimension $2$.  Similarly, a knot for which the diagram has exactly one singularity happening from a Reidemester point (that is, an ordinary tangency or a double point, since cusps are forbidden here) lies in a strata of codimension $1$.

The heart of the proof consists in fixing a small loop around one of these strata of codimension $2$ in the moduli space, and checking that it is sent to 0 by the potential 1-cocycle. Doing so is called solving the \textit{commutation equations}, the \textit{cube equations} and the \textit{tetrahedron equations}, respectively. We can now give the formula for $R$, after which we will solve those equations in the next sections.\\

The function $R$ will be defined by counting the number of times a generic loop $\gamma$ encounters a knot with a triple point, or equivalently the number of times it crosses a strata of codimension $1$. Each of these occurrences will be counted with a weight $W$ depending on the whole diagram at that moment. However we will only count the times the triple point satisfies certain properties, that are defined in \cite{4} and \cite{5}, and reminded below.

If a knot contains a triple point, in the corresponding Gauss diagram there are three arrows that form a triangle, which is of one of the two forms below.

\begin{figure}[H]
\centering
\includegraphics[width=6cm]{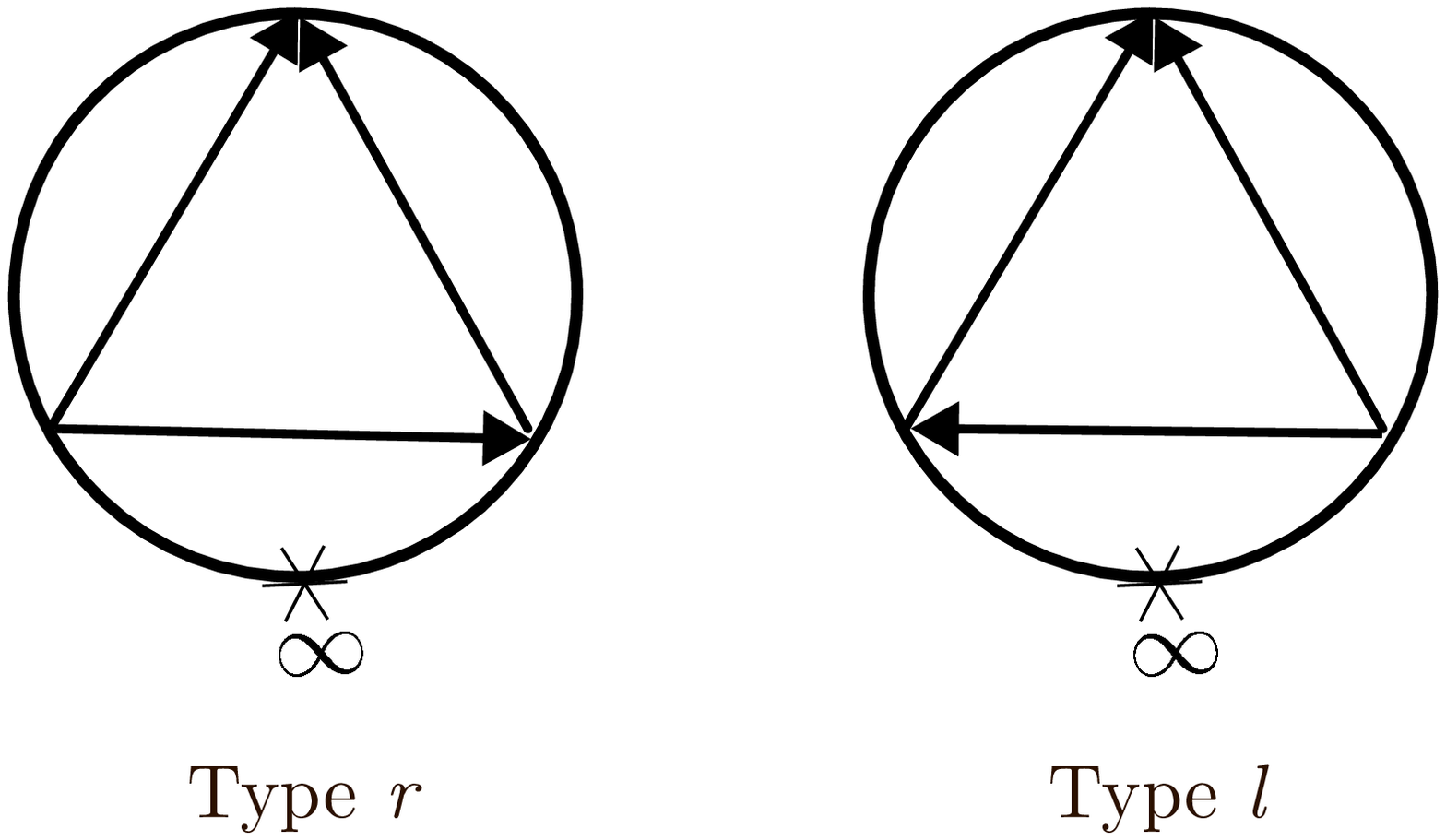}
\caption{Global type of a Reidemeister III move.}
\label{fig.global_type}
\end{figure}

We then say that the triple point is of \textit{global type r} or \textit{l}, respectively. We can also distinguish different types of triple points based on homology: for $q$ a crossing, we note $[q] \in \mathbb{Z}$ the homology class of $q^+$. This is called the \textit{homological marking} of the crossing. The crossings with homological marking $0$ or $n$ are called \textit{persistent crossings}. For $p$ a triple point, we also note $d$ (for "distinguished") the crossing between the lowest and the highest branch, and $hm$, respectively $ml$, the crossing between the middle branch and the highest, respectively the lowest. Then a triple point is said to be of type $r([d], [hm], [ml])$ or $l([d], [hm], [ml])$. The cocycle $R$ will depend on an integer parameter $a$ between $1$ and $n-1$, and will only count  the times $\gamma$ crosses a strata correspond to a Reidemeister move RIII of type $r(a,n,a)$. \\

Naturally, $R$ needs to take into account the direction in which a path crosses a strata of codimension $1$, more precisely two crossings in opposite directions must differ by a change of sign. The figure below sums up the coorientation of the strata of codimension $1$ corresponding to Reidemeister moves RIII. 
When $\gamma$ crosses such a strata transversely, we note $p$ the corresponding triple point, and we define $\mbox{sign}(p) = 1$ if it is done following the coorientation, and $\mbox{sign}(p)  = -1$ if it is done in the opposite direction.\\

\begin{figure}[H]
\centering
\includegraphics[width=10cm]{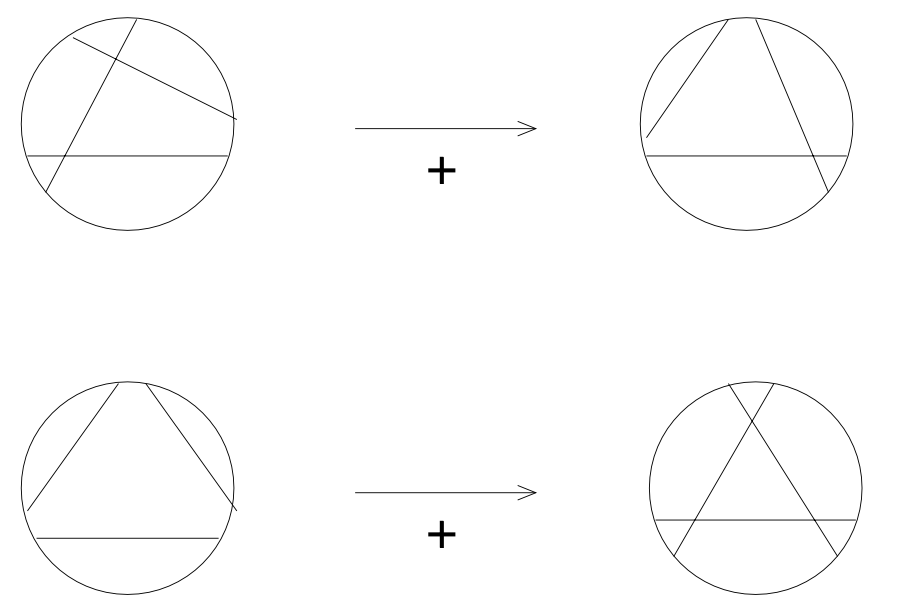}
\caption{sign of a Reidemeister III move.}
\label{fig.global_type}
\end{figure}

Finally, we can define the weight $W$ associated to a diagram with a triple point.
III. In \cite{5}, the weights were defined by slight modifications of Gauss diagram formulas, for example those of the coefficients of the Conway-Alexander polynomial. In our case, the weights come from the Gauss formula for $v_3$. We remind the formula in question :

\begin{figure}[H]
	\center
	\includegraphics[width=0.8\textwidth]{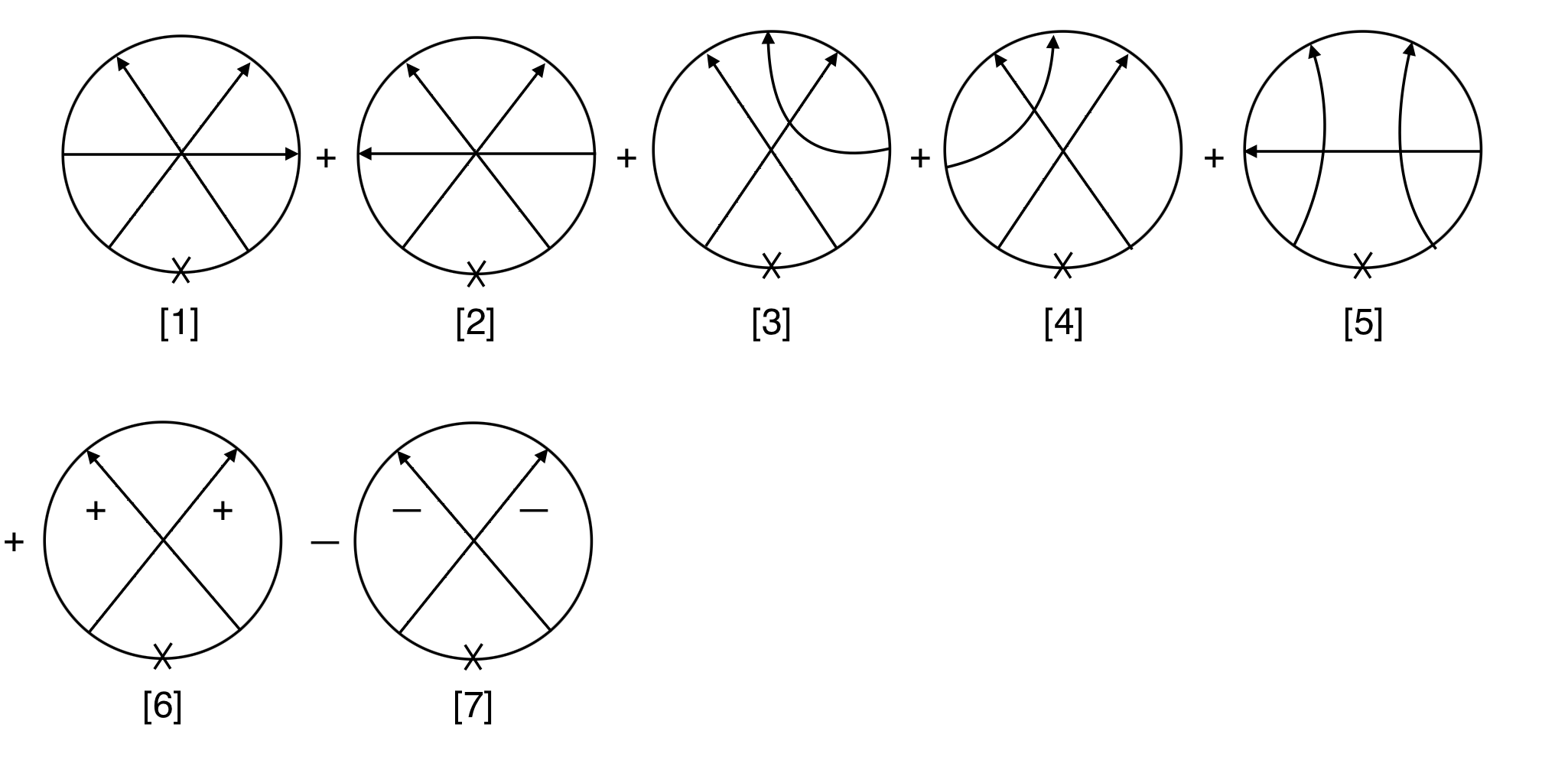}
	\caption{Arrow diagrams giving $v_3$.}
\end{figure}


Precisely, for a Reidemeister move III $p=r(a,n,a)$ in a Gauss diagram $G\in \widetilde{M_n}$, the \textit{weight} of $p$ is defined by 
$$W(p) = \sum\limits_{\phi \in \{ \phi \in \cup_{i} Hom(A_i,G_p) | hm\in S(\phi)\text{is the first image arrow from } \infty \} }c_{i(\phi)} sign(\phi),$$
where $\sum_{i}c_iA_i=A$ is the Gauss diagram formula for $v_3$ shown in figure \ref{Fig.v3_formula} ($c_i = \pm 1$ being the coefficients in $v_3$), $\widetilde{M_n}$ is the underlying space $M_n$ in which Reidemeister move I is also allowed and $G_p$ is the remaining of the Gauss diagram $G$ deleting the non-persistent crossings. Notice that $G_p$ is the Guass diagram of the lifting knot of $G$ under the n-folded covering projection of the torus.

That is to say, $v_3(G)$ counts the subdiagrams of $G$ that are of one of the forms below, and each subdiagram is counted either positively or negatively depending on the product sign of the arrows. We will evaluate the weights on "diagrams" where we allow three arrows to form a triangle, but we still call those "diagrams". In the computations, the role of the base point will be played by the point $\infty$. 
Then, for $G$ the Gauss diagram of a knot with a triple point $p$, the subdiagrams that will contribute to $W$ are those that also contribute in the formula for $v_3$ (meaning those of the form of one of the diagrams above), \uline{but such that the first arrow encountered after the base point is $hm$, and such that all arrows have homological marking $0$ or $n$}. 
Of course each subdiagram is counted positively if the product of the signs of the arrows is $+1$, and counted negatively otherwise. \\

Then, the cocycle $R$ is defined as

$$
R (\gamma) = \sum_{p = (r(a,n,a), \infty \in d^+) \in \gamma} \mbox{sign}(p) W(p) w(hm_p).
$$

Here, the sum is taken over all the times at which $\gamma$ contains a triple point $p$, that is of type $r(a,n,a)$, and such that the point $\infty$ is in the arc $d^+$. The function $w$ denotes the writhe. We used the notations $R$ and $W$ for simplicity, in \cite{5} those are noted $R^3_{a,d^+,hm}$ and $W_2(hm)$.

\begin{thm}
For any integer $a \in (0,n)$, $R$ is a well defined 1-cocycle in $\widetilde{M_n}$
\end{thm}

\begin{rem}
    Our one-cocycle is well-defined in $\widetilde{M_n}$ if it is well-defined in $M_n$. Indeed, when we consider about the meridian around a cusp. We only need to consider that the Reidemeister move \uppercase\expandafter{\romannumeral3} $p$ is of type $r(a,n,a)$. So for the resulting crossing of the Reidemeister move \uppercase\expandafter{\romannumeral1} $\beta$, $[\beta]=0$ or $n$. We only need to consider $[\beta]=n$ and $\beta = hm$. The involving two crossing from the Reidemeister move \uppercase\expandafter{\romannumeral2} are not persistent crossings. Therefore the $hm$ in the Gauss diagram $G_p$ is an isolated arrow which does not appear in any configurations of $v_3$.
    
    \begin{figure}[H]
        \begin{tikzpicture}
            \node[inner sep=0pt] at (-5,0) {\includegraphics[width = 0.25 \textwidth]{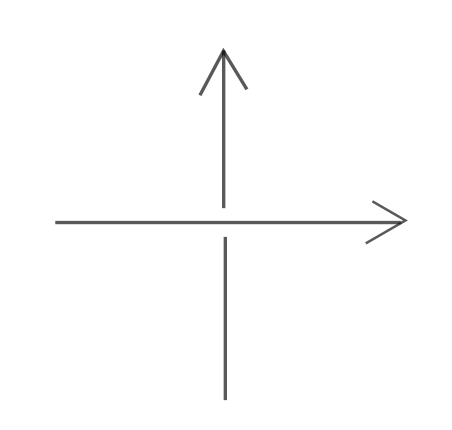}};

            \node[inner sep=0pt] at (0,0) {\includegraphics[width = 0.25 \textwidth]{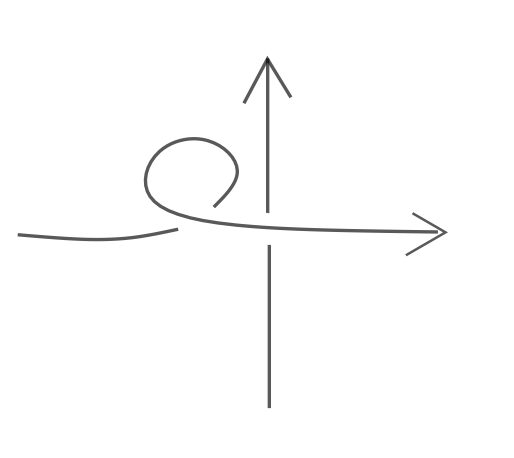}};

            \node[inner sep=0pt] at (5.6,0) {\includegraphics[width = 0.35 \textwidth]{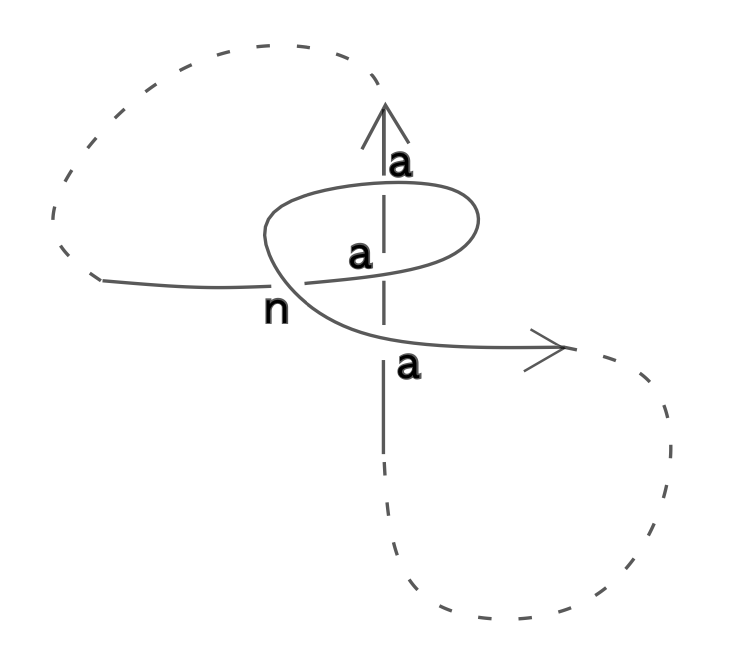}}; 
            
            \node[inner sep=0pt] at (3,-4) {\includegraphics[width = 0.35 \textwidth]{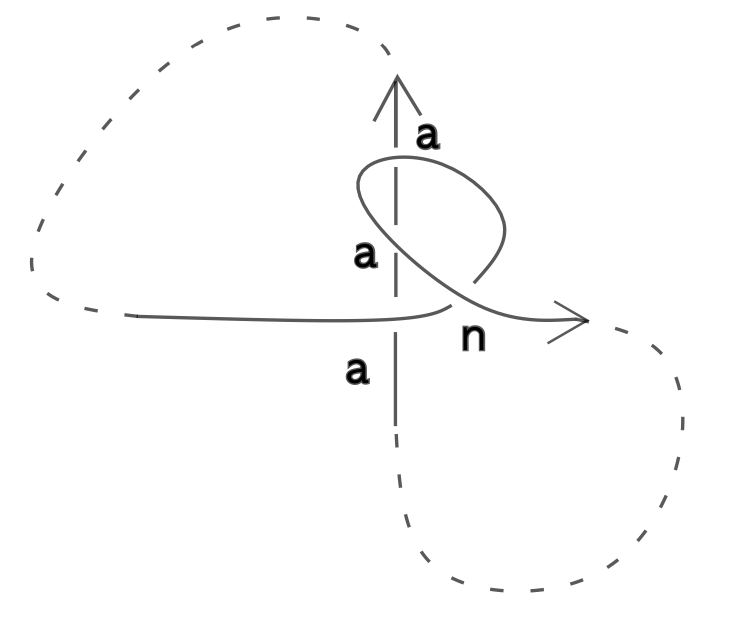}};

            \node[inner sep=0pt] at (-3,-4) {\includegraphics[width = 0.25 \textwidth]{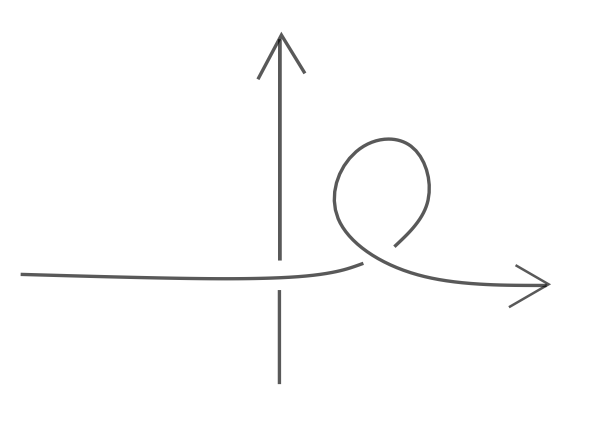}}; 
            
            \draw[-latex,very thick] (-3.5,-1) -- (-2,-1);
            \draw[-latex,very thick] (1.7,-1) -- (3.2,-1);
            \draw[-latex,very thick] (7,-2.5) -- (6,-4);
            \draw[-latex,very thick] (1,-5) -- (-1,-5);
            \draw[-latex,very thick] (-4,-3.7) -- (-5,-2.5);
            
        \end{tikzpicture}
    \label{Fig.cusp_meridian}
    \caption{meridian around a cusp}
    \end{figure}
    
    \begin{figure}[H]  
        \centering
	    \includegraphics[width=0.3\textwidth]{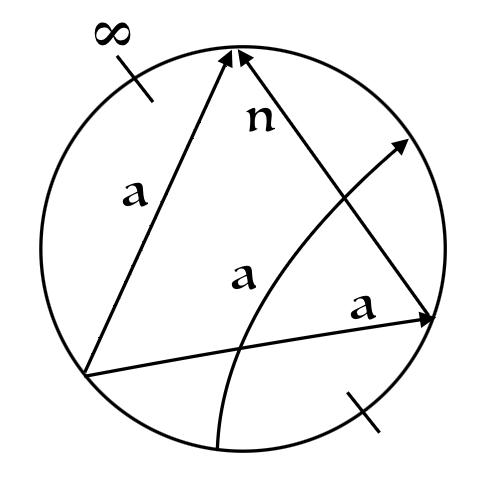} 
	    \caption{Guass diagram of the Reidemeister move \uppercase\expandafter{\romannumeral3} on the meridian}
	\label{Fig.cusp_Guass} 
    \end{figure}
\end{rem}
\section{Contribution correspondence under Reidemeister move \uppercase\expandafter{\romannumeral3} }\label{s4}
In \cite{1}, a contribution correspondence under Reidemeister move \uppercase\expandafter{\romannumeral3} with respect to $p_{0,2k}$ is given. Here we give the contribution correspondence with respect to $v_3$. For doing this, let us first indicate all types of Reidemeister move $\uppercase\expandafter{\romannumeral3}$.
Locally, there are 8 types of Reidemeister move $\uppercase\expandafter{\romannumeral3}$: 
\begin{figure}[H]
\center
\includegraphics[width=0.75\textwidth]{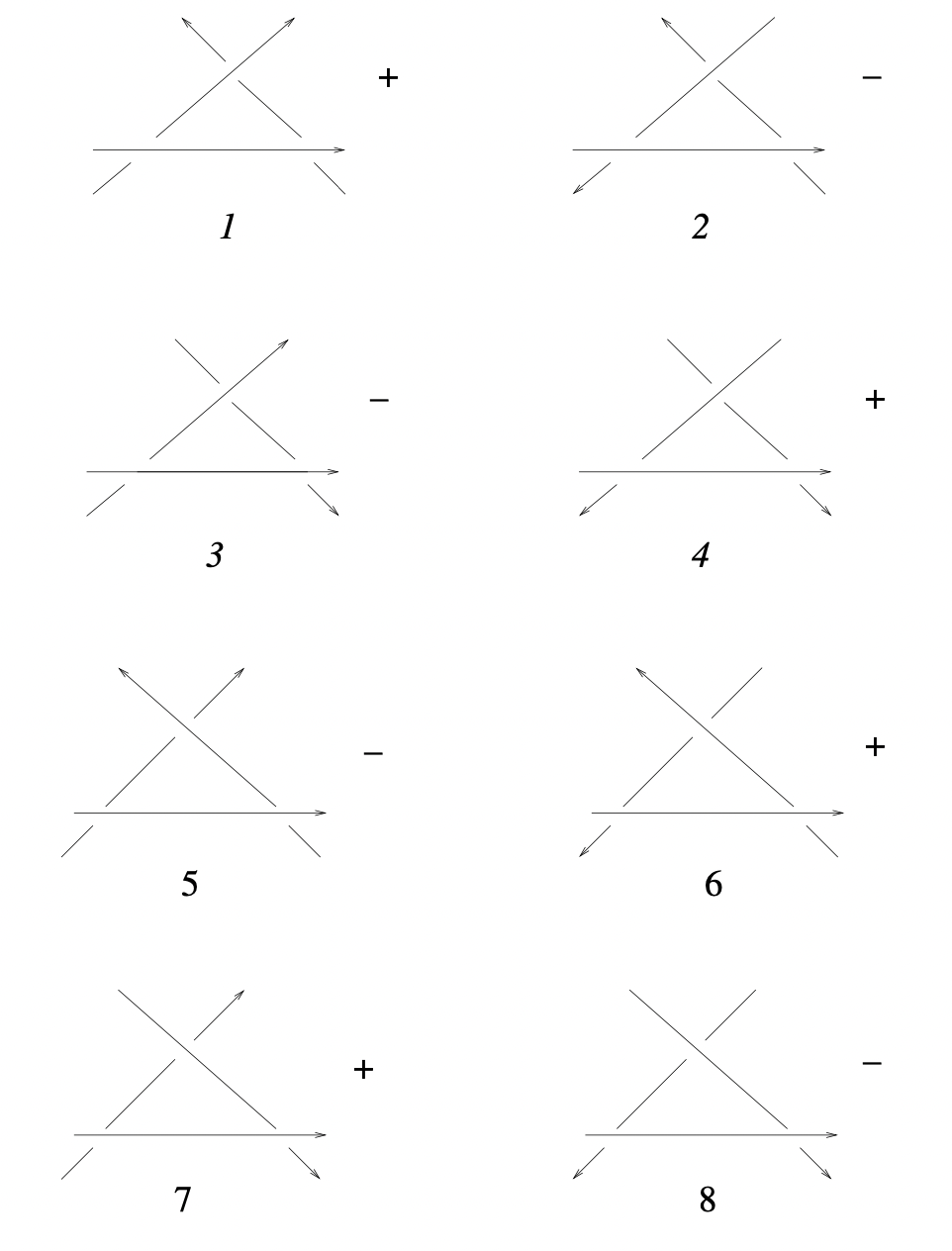}
\caption{8 local types of Reidemeister move \uppercase\expandafter{\romannumeral3}}
\label{Fig.local_types}
\end{figure}
Globally there are left and right two types  demonstrated in \ref{fig.global_type}.

In the whole section, we assume that all Gauss diagram for Reidemeister move $\uppercase\expandafter{\romannumeral3}$ with 3 separate arcs have the global anti-clockwise order. 

All the types of Reidemeister move $\uppercase\expandafter{\romannumeral3}$ are displayed by Gauss diagrams as the following, where numbers represent local types:
\begin{figure}[H]
	\includegraphics[width=0.9\textwidth]{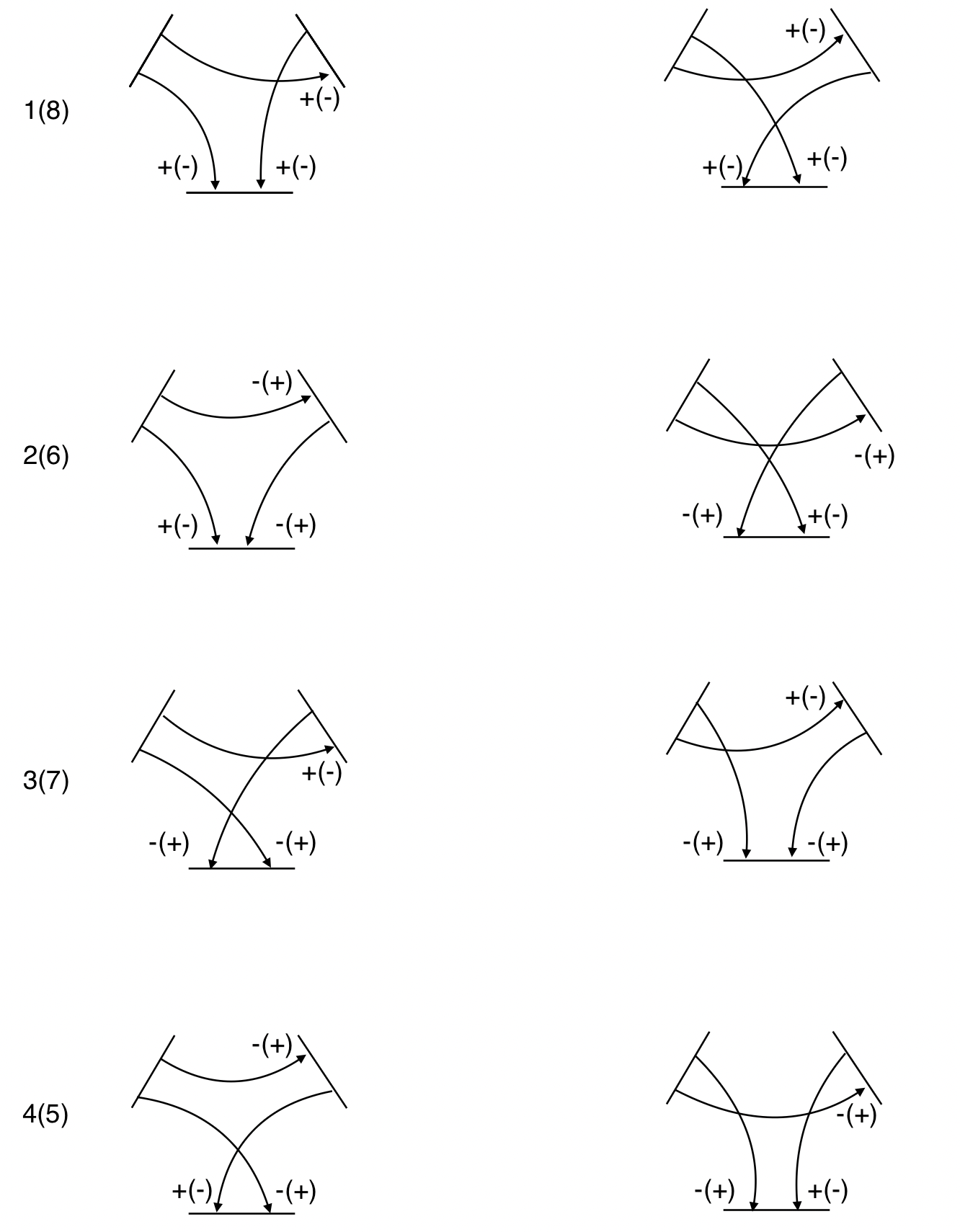} 
	\caption{left types} 
	\label{Fig.left_Gauss} 
\end{figure}
\begin{figure}[H]
	\includegraphics[width=0.9\textwidth]{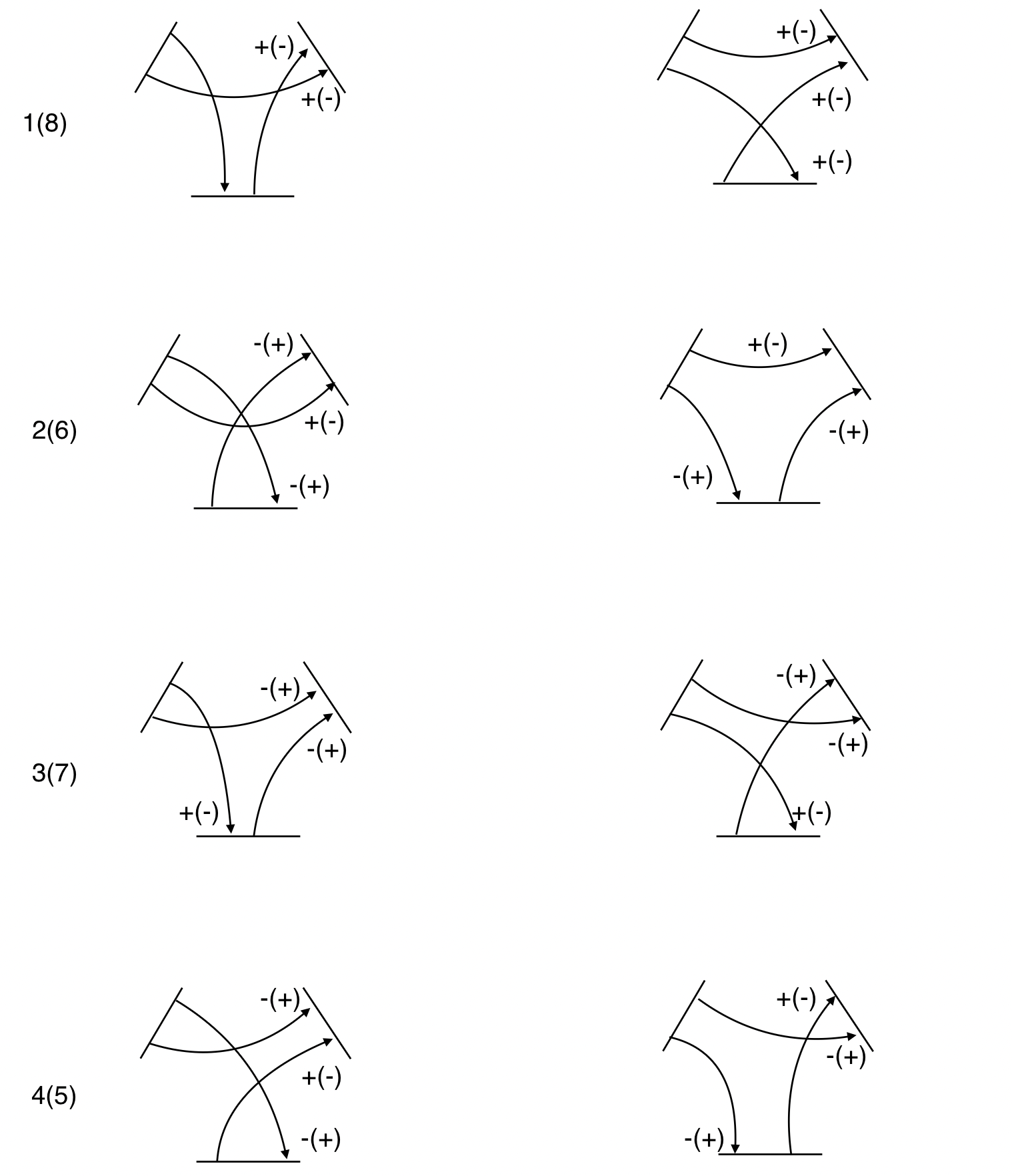} 
	\caption{right types}
	\label{Fig.right_Gauss} 
\end{figure}

Let $G$ be a Gauss diagram of a knot $K$ and we perform a \textit{positive} Reidemeister move $\uppercase\expandafter{\romannumeral3}$, meaning of type 1, on $G$ and we get a Gauss diagram $\widetilde{G}$. For each $\phi \in \cup_{A}Hom(A,G)$, where $A$ runs for configurations respect to $v_3$ (see \ref{Fig.v3_formula}), we try to find a homomorphism $\theta(\phi) \in \cup_{A}Hom(A,\widetilde{G})$, i.e., to construct a map $$\theta : \cup_{A}Hom(A,G) \rightarrow \cup_{A}Hom(A,\widetilde{G}).$$ What's more, we hope that $\theta$ is bijective.

We construct it in 4 cases:
\begin{enumerate}
	\item\label{Cas.corresp1} $S(\phi)$ contains no arrows in the Reidemeister move $\uppercase\expandafter{\romannumeral3}$;
	\item\label{Cas.corresp2} $S(\phi)$ contains exactly one arrow in the Reidemeister move $\uppercase\expandafter{\romannumeral3}$;
	\item\label{Cas.corresp3} $S(\phi)$ contains exactly two arrows in the Reidemeister move $\uppercase\expandafter{\romannumeral3}$;
	\item\label{Cas.corresp4} $S(\phi)$ contains all three arrows in the Reidemeister move $\uppercase\expandafter{\romannumeral3}$.
\end{enumerate}
However, Case \ref{Cas.corresp4} is impossible by  comparing the formula $v_3$ in figure \ref{Fig.v3_formula} and the the Reidemeister move $\uppercase\expandafter{\romannumeral3}$ (type 1 in \ref{Fig.left_Gauss}, \ref{Fig.right_Gauss}). 

In Case \ref{Cas.corresp1}, if $\phi \in Hom(A_i,G)$, where $A_i$ is a configuration in $v_3$, we choose the same arrows in $\widetilde{G}$ as $S(\theta(\phi))$. $\theta(\phi) \in Hom(A_i, \widetilde{G})$ is well-defined because Reidemeister move $\uppercase\expandafter{\romannumeral3}$ just changes locally near the three related crossings and the rest remain unchanged. 

In Case \ref{Cas.corresp2}, if $\phi \in Hom(A_i,G)$, where $A_i$ is a configuration in $v_3$ and $\alpha \in S(\phi)$ is the arrow in the Reidemeister move $\uppercase\expandafter{\romannumeral3}$, we choose $S(\theta(\phi)) = \{ \widetilde{\alpha} \} \cup (S(\phi) - \{ \alpha \})$, where $\widetilde{\alpha}$ is the resulting arrow of $\alpha$ after the the Reidemeister move $\uppercase\expandafter{\romannumeral3}$. $S(\theta(\phi))$ determines a well-defined homomorphism because Reidemeister move $\uppercase\expandafter{\romannumeral3}$ preserves signs and the relative order of $\alpha$ and the arrows that do not participate in the Reidemeister move $\uppercase\expandafter{\romannumeral3}$. 

Case \ref{Cas.corresp3} is complicated. We first give all possibilities of the two arrows in the Reidemeister move $\uppercase\expandafter{\romannumeral3}$. And then we shall give the map $\theta$ by listing it out directly. 

In the following figures, the Roman numbers $\uppercase\expandafter{\romannumeral1}$, $\uppercase\expandafter{\romannumeral2}$,$\uppercase\expandafter{\romannumeral3}$ represent which two arrows are in $S(\phi)$. The numbers 1,2,3 represent where the base point lies. And the letter $L$ and $R$ represent the global type of the positive Reidemeister move $\uppercase\expandafter{\romannumeral3}$. "Prime" means the Gauss diagrams performing the Reidemeister move $\uppercase\expandafter{\romannumeral3}$.
\begin{figure}[H]
	\centering
	\includegraphics[width=0.8\textwidth]{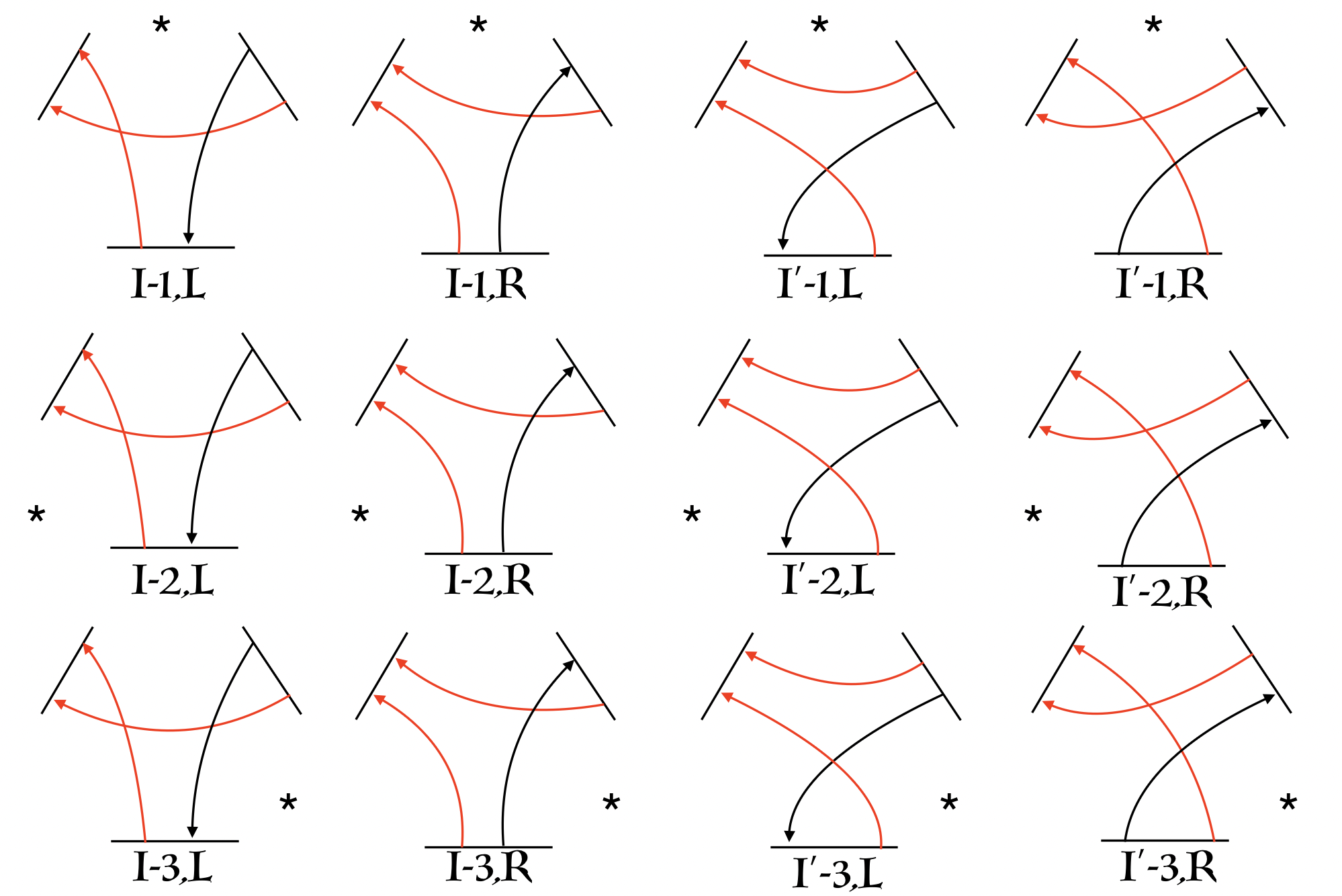}
	\caption{type-{\uppercase\expandafter{\romannumeral1}}}
	\label{Fig.type_I}
\end{figure}
\begin{figure}[H]
	\centering
	\includegraphics[width=0.8\textwidth]{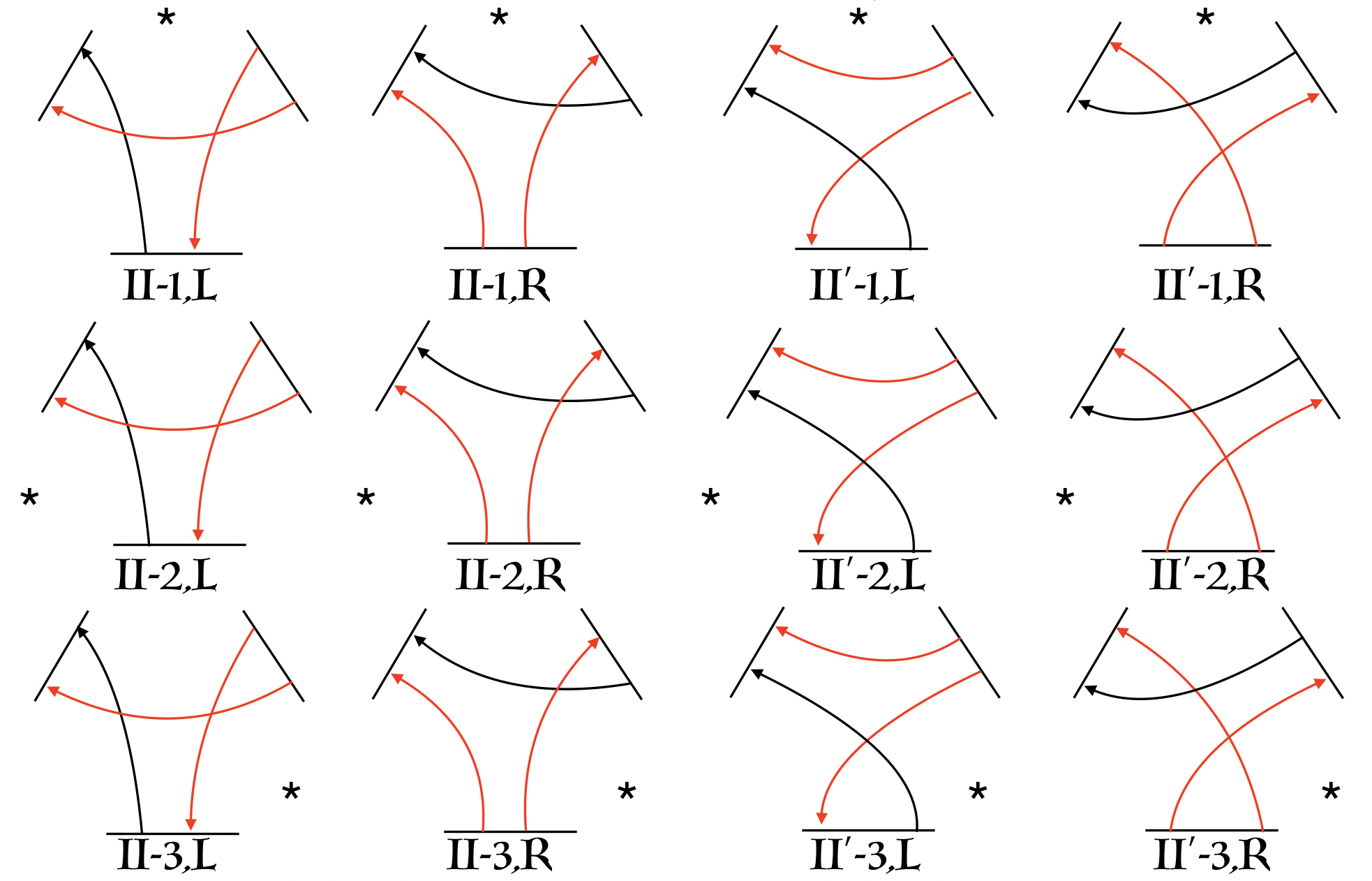}
	\caption{type-{\uppercase\expandafter{\romannumeral2}}}
	\label{Fig.type_II}
\end{figure}
\begin{figure}[H]
	\centering
	\includegraphics[width=0.8\textwidth]{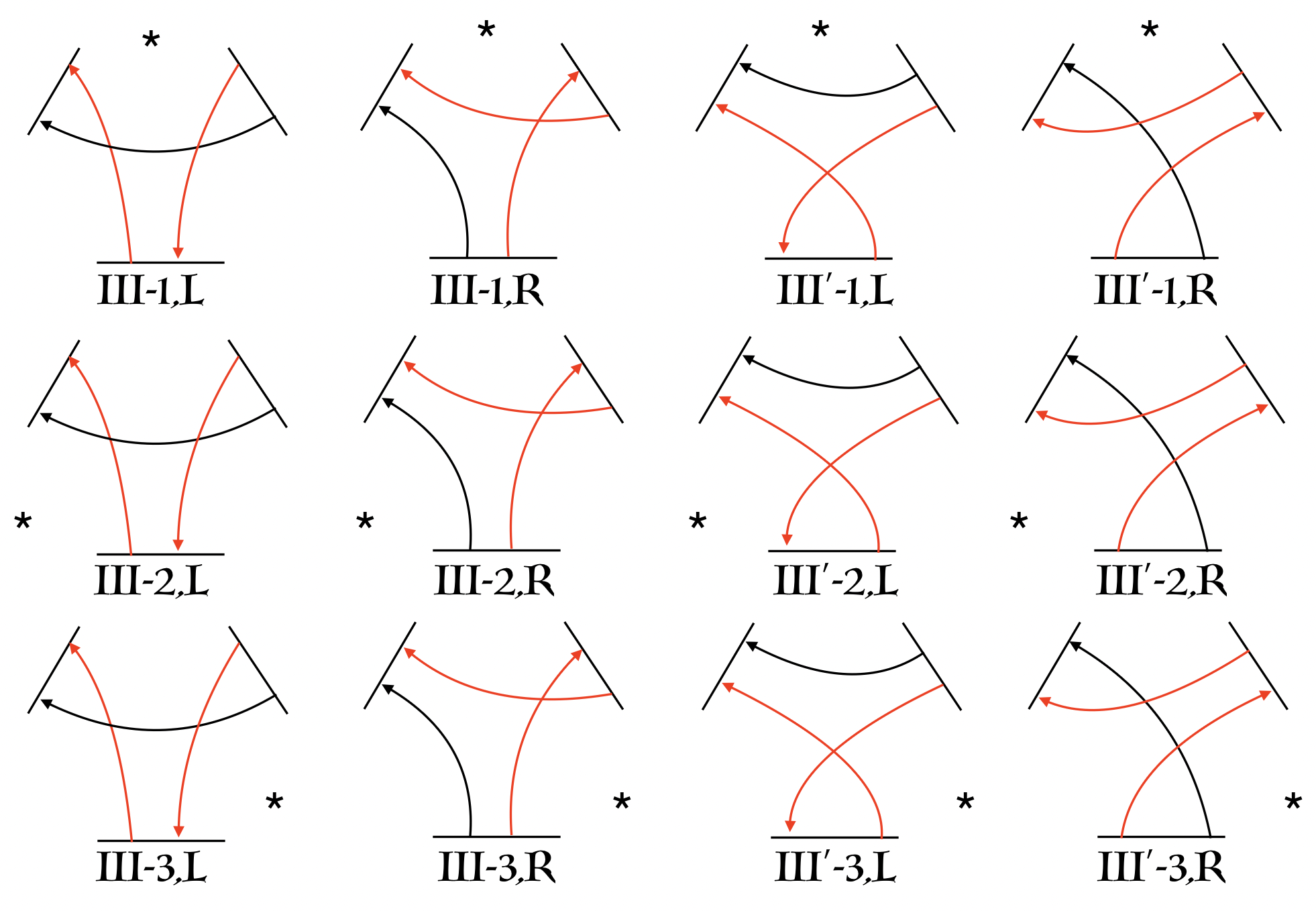}
	\caption{type-{\uppercase\expandafter{\romannumeral3}}}
	\label{Fig.type_III}
\end{figure}
\begin{rem}
The three figures above show all possibilities for the two red arrows which are in $S(\phi)$. 
\end{rem}
Now we can construct $\theta$ in Case \ref{Cas.corresp3}. There are 6 subcases depending on where the base point is and the global type of the Reidemeister move $\uppercase\expandafter{\romannumeral3}$. 
\begin{enumerate}
	\item\label{Cas.1L} 1,L:
	\begin{figure}[H]
		\centering
		\includegraphics[width = 0.47\textwidth]{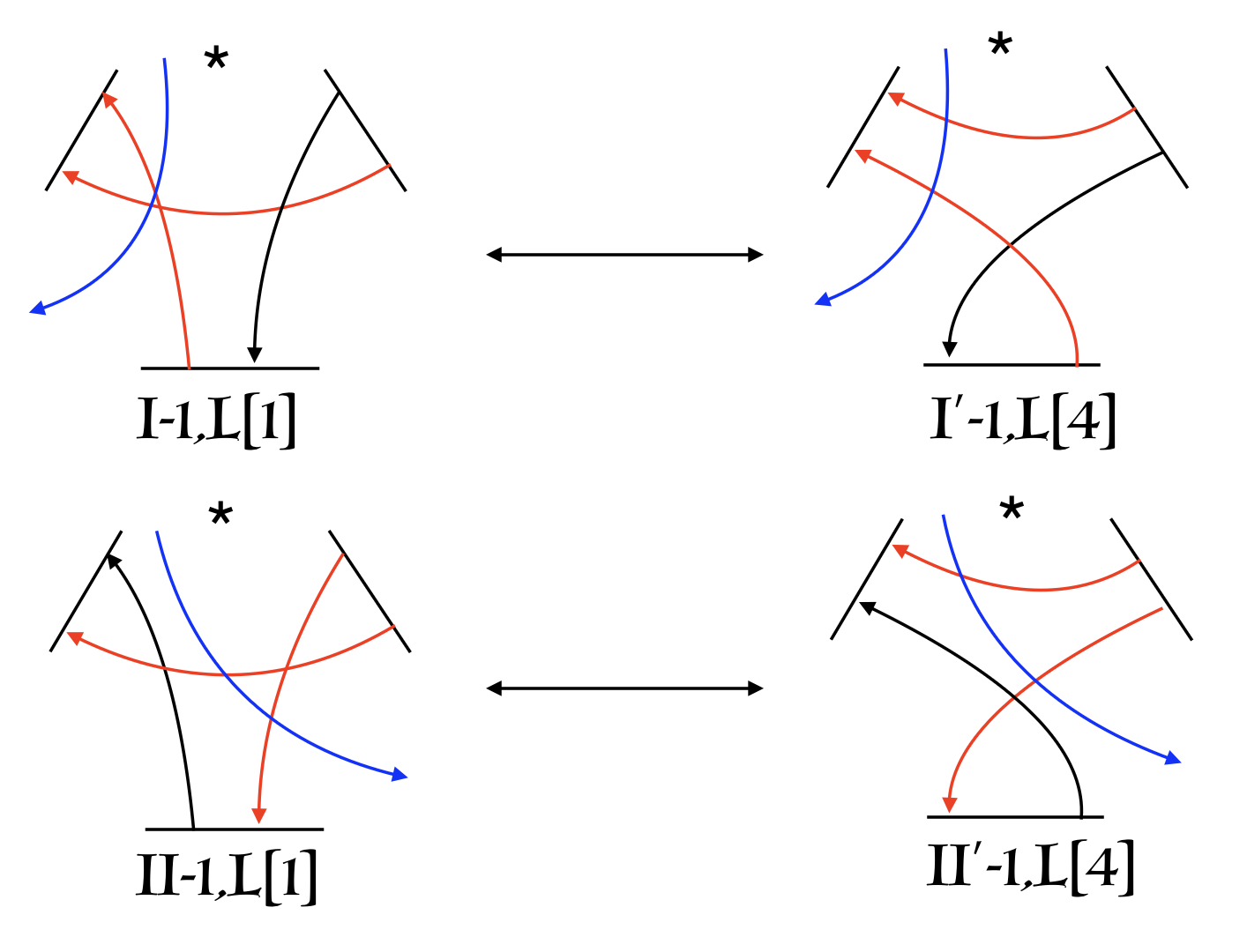}
		\label{Fig.1L}
	\end{figure}
	\item\label{Cas.2L} 2,L:
	\begin{figure}[H]
		\centering
		\includegraphics[width = 0.47\textwidth]{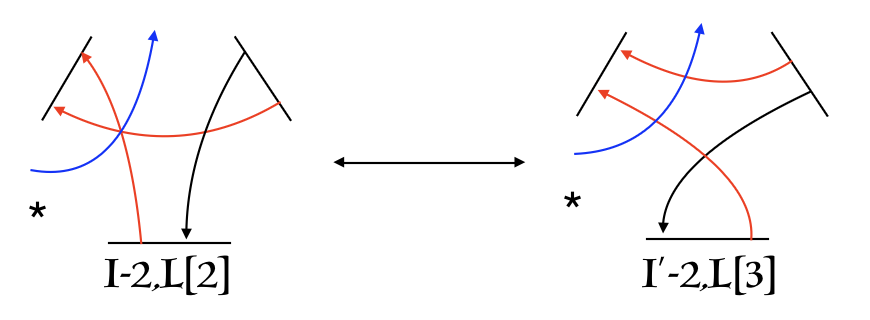}
		\label{Fig.2L}
	\end{figure}
	\item\label{Cas.3L} 3,L:
	\begin{figure}[H]
		\centering
		\includegraphics[width = 0.75\textwidth]{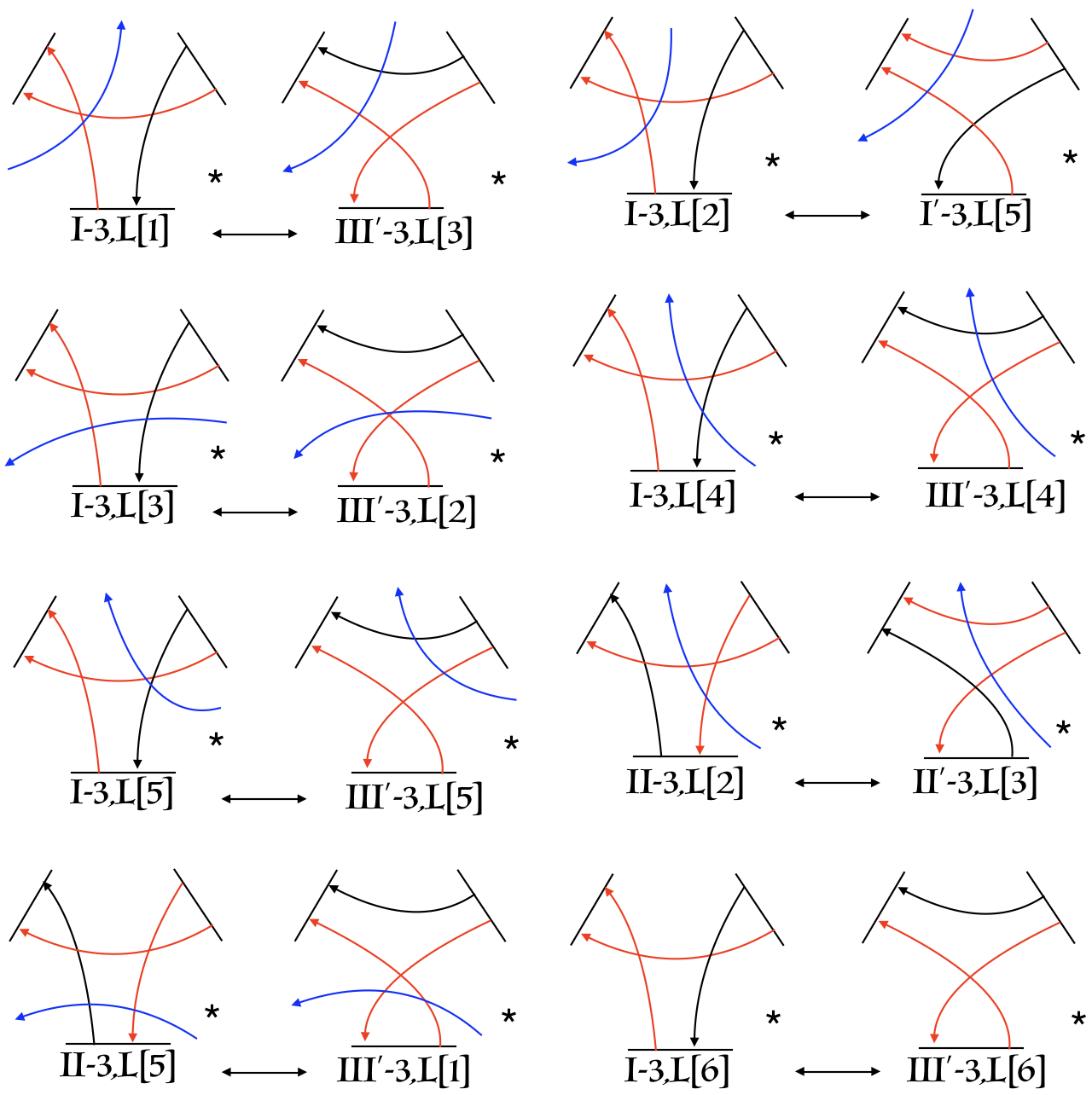}
		\label{Fig.3L}
	\end{figure}
	\item\label{Cas.1R} 1,R:
	\begin{figure}[H]
		\centering
		\includegraphics[width = 0.47\textwidth]{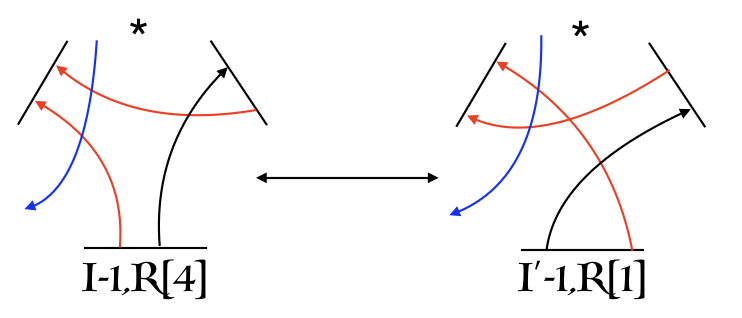}
		\label{Fig.1R}
	\end{figure}
	\item\label{Cas.2R} 2,R:
	\begin{figure}[H]
		\centering
		\includegraphics[width = 0.45\textwidth]{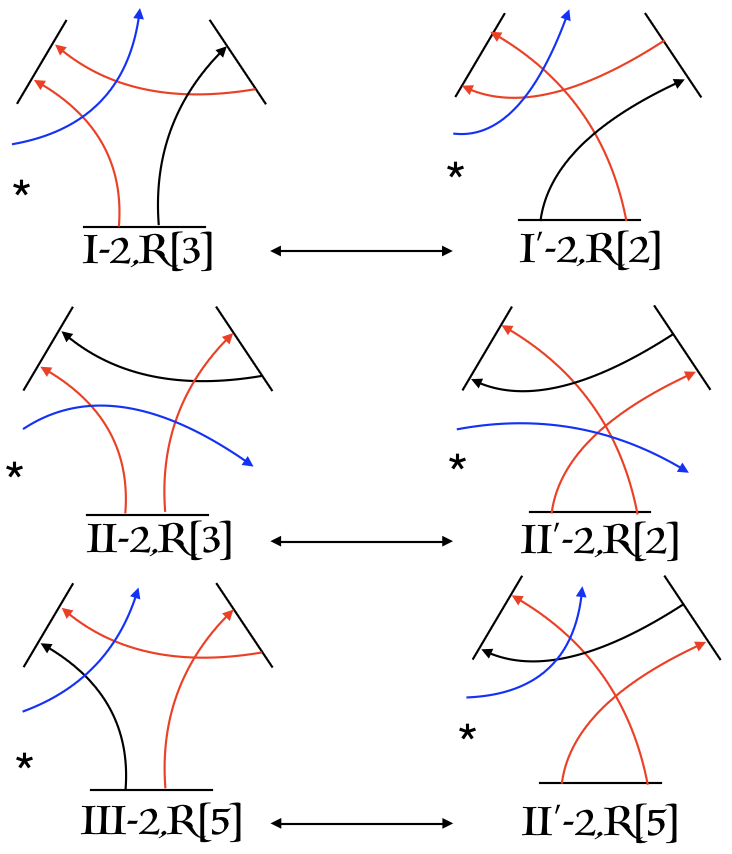}
		\label{Fig.2R}
	\end{figure}
	\item\label{Cas.3R} 3,R:
	\begin{figure}[H]
		\centering
		\includegraphics[width = 0.74\textwidth]{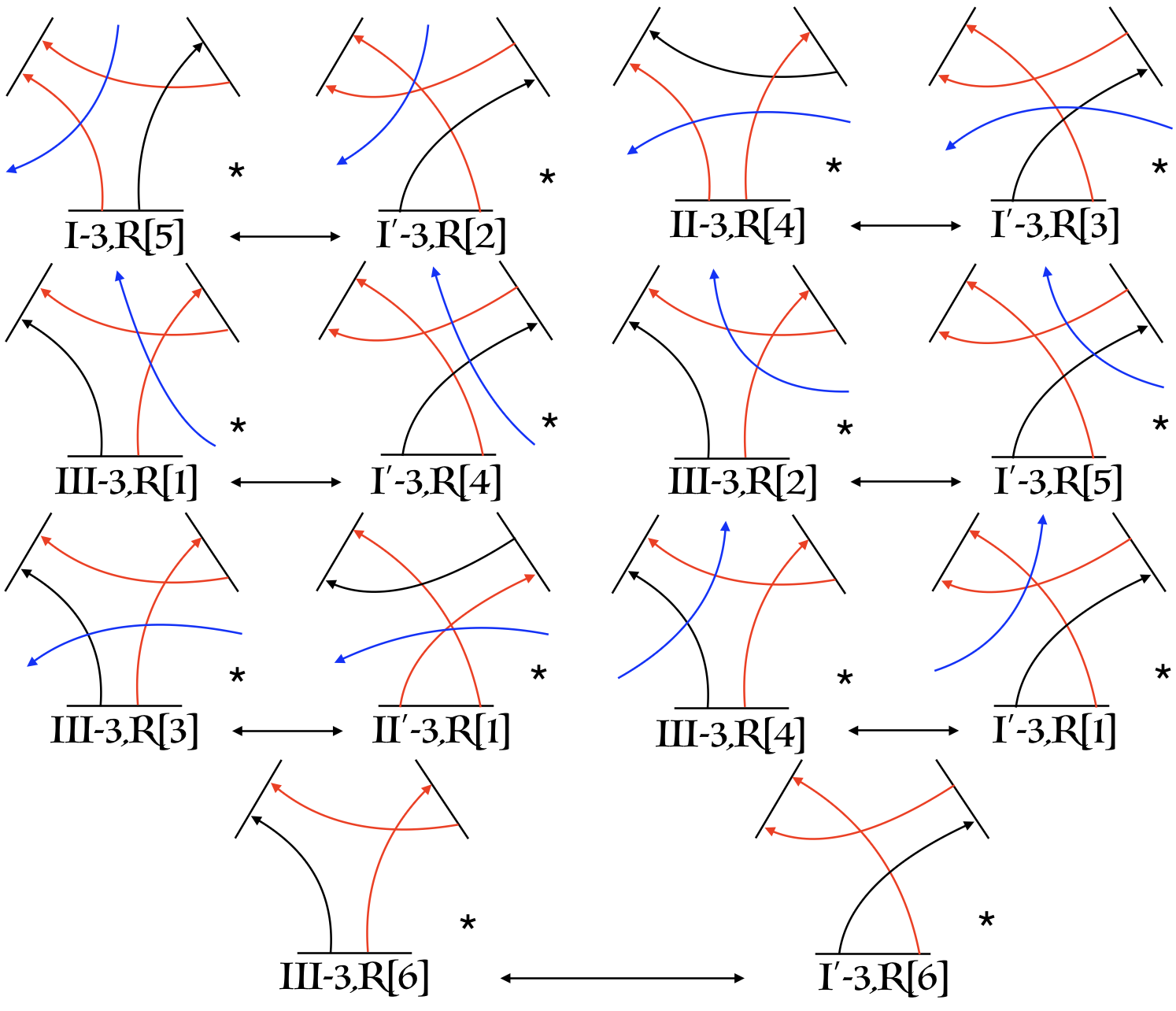}
		\label{Fig.3R}
	\end{figure}
\end{enumerate}
The figures above indicate how $\theta$ maps at $\phi$: if it is a diagram above, just map it following the arrow between the diagrams. The coloured arrows are exactly $S(\phi)$ or $S(\theta(\phi))$. The numbers in the bracket indicate the configuration in the Gauss diagram formula for $v_3$. The bijective $\theta$ is constructed. 
However such a map is not unique. 

\section{Commutation equations}\label{s5}
The commutation equations mean that for a loop $\gamma$ which encounters two successive Reidemeister move $\uppercase\expandafter{\romannumeral3}$ $p$ and $q$, then our one-cocycle is independent of the order of $p,q$. 
\begin{figure}[H]
\centering
\includegraphics[width= 0.35\textwidth]{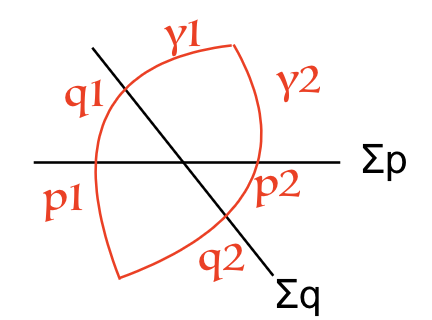}
\caption{commutation equations}
\label{Fig.commutation_eq}
\end{figure}
In fact we have a stronger relation: 
\begin{prop}
Suppose we have two Reidemeister moves $\uppercase\expandafter{\romannumeral3}$ $p,q$ of right types to be act on a Gauss diagram $G\in \widetilde{M_n}$. Let $W(p)$ and $\widetilde{W}(p)$ be the weights of $p$ before and after the Reidemeister move $\uppercase\expandafter{\romannumeral3}$ $q$ acting on $G$. Then $W(p) = \widetilde{W}(p)$. 
\end{prop}
\begin{proof}
Any Reidemeister move $\uppercase\expandafter{\romannumeral2}$ does not influence the weight. Indeed, the two arrows in a Reidemeister move \uppercase\expandafter{\romannumeral2} cannot be in $S(\phi)$ together for our configurations $A_i$ of $v_3$. If one of the two arrows contributes, then the other arrow just does the same thing except with an opposite sign. (Note the type $[6]$ and $[7]$ in $v_3$) 

Therefore we can assume $q$ to be a positive Reidemeister move $\uppercase\expandafter{\romannumeral3}$, for other types could be generated by the positive one and  Reidemeister moves $\uppercase\expandafter{\romannumeral2}$. Since Reidemeister move $\uppercase\expandafter{\romannumeral3}$ is a local deformation, if there is a homomorphism $\phi \in \cup_{i} Hom(A_i,G_p)$ such that $hm_p$ lies in $S(\phi)$ and is the first image arrow from $\infty$ (resp. not satisfying one of the two conditions) then $\theta_q(\phi)$ still satisfies the two conditions (resp. not satisfies one of the two conditions). 
\end{proof}

Hence going along the meridian in the singularity case \ref{singularity_1} we have $R=0$. In a similar way, we can see that $R$ would have the value 0 going along the meridian in the case \ref{singularity_4}. Indeed, the two crossing of a Reidemeister II move do not contribute together in any configuration of $v_3$, and the the contribution of the each single crossing would cancel out since they have opposite signs.

\section {Cube equations}\label{s6}

We now want to solve the cube equations. Let $\gamma$ be a small loop in $M_n$ around a strata of codimension $2$ of the form $\Sigma^2_{\mbox{self-flex}}$. The intersection of this strata with a meridional $2$-disc has the following form :

\begin{figure}[H]
		\center
		\includegraphics[width=0.7\textwidth]{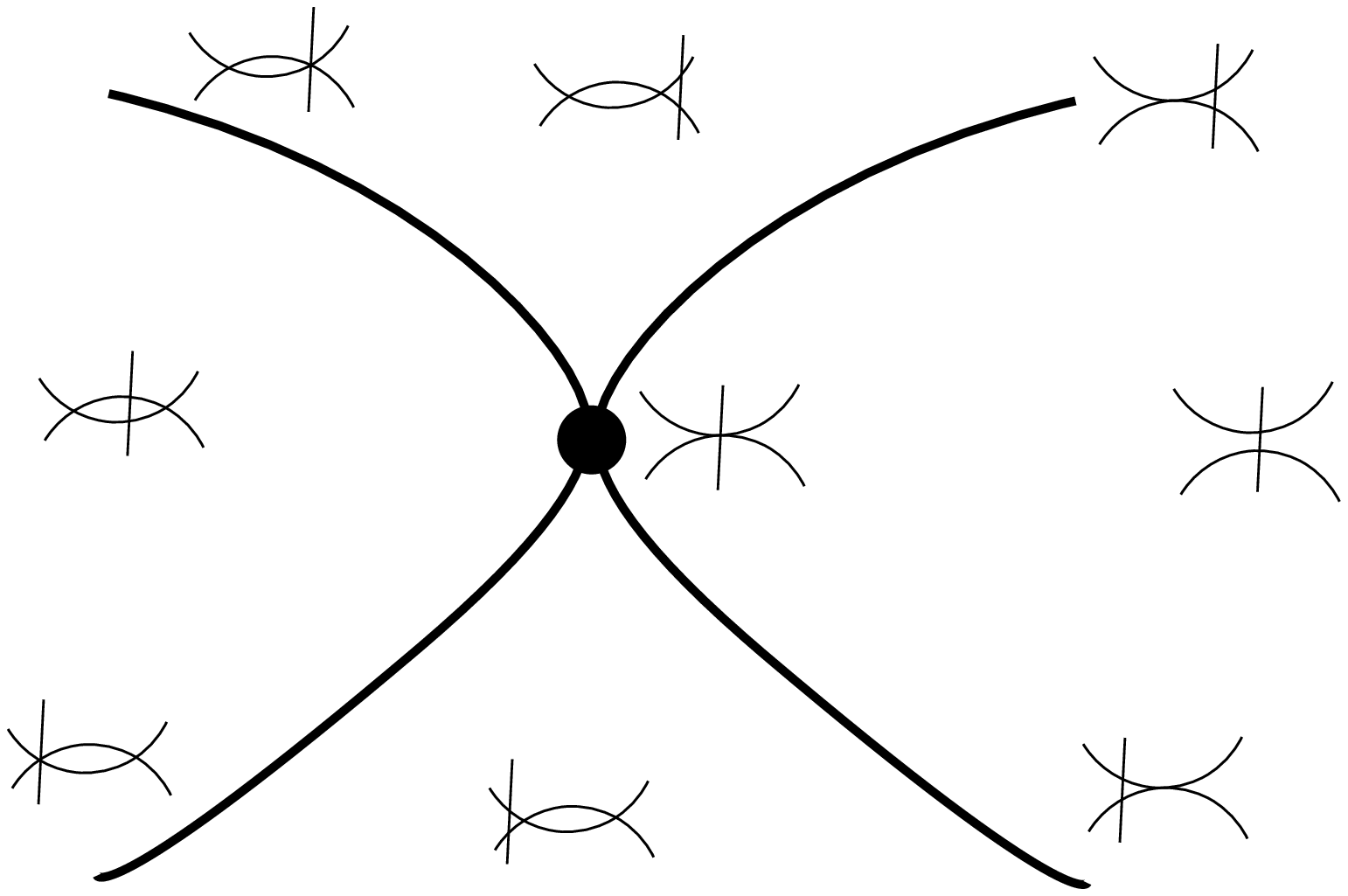}
		\caption{Strata $\Sigma^2_{\mbox{trans-self}}$ intersected with a transversal $2$-disk}
		\label{Fig.strata_trans-self}
		\end{figure}

The strata of codimension $2$ is the intersection of two strata of codimension $1$, one of them corresponding to a Reidemeister III move, and the other one corresponding to a Reidemeister II. Therefore, $\gamma$ only contributes to $R$ when it crosses the former. In fact, $\gamma$ crosses this strata twice, and once in each direction. We can also see that the knot diagrams at each moment $\gamma$ crosses the strata are almost similar. In fact, the diagrams (noted $G$ and $G'$) only differ in one of the following ways :

\begin{figure}[H]
		\center
		\includegraphics[width=0.7\textwidth]{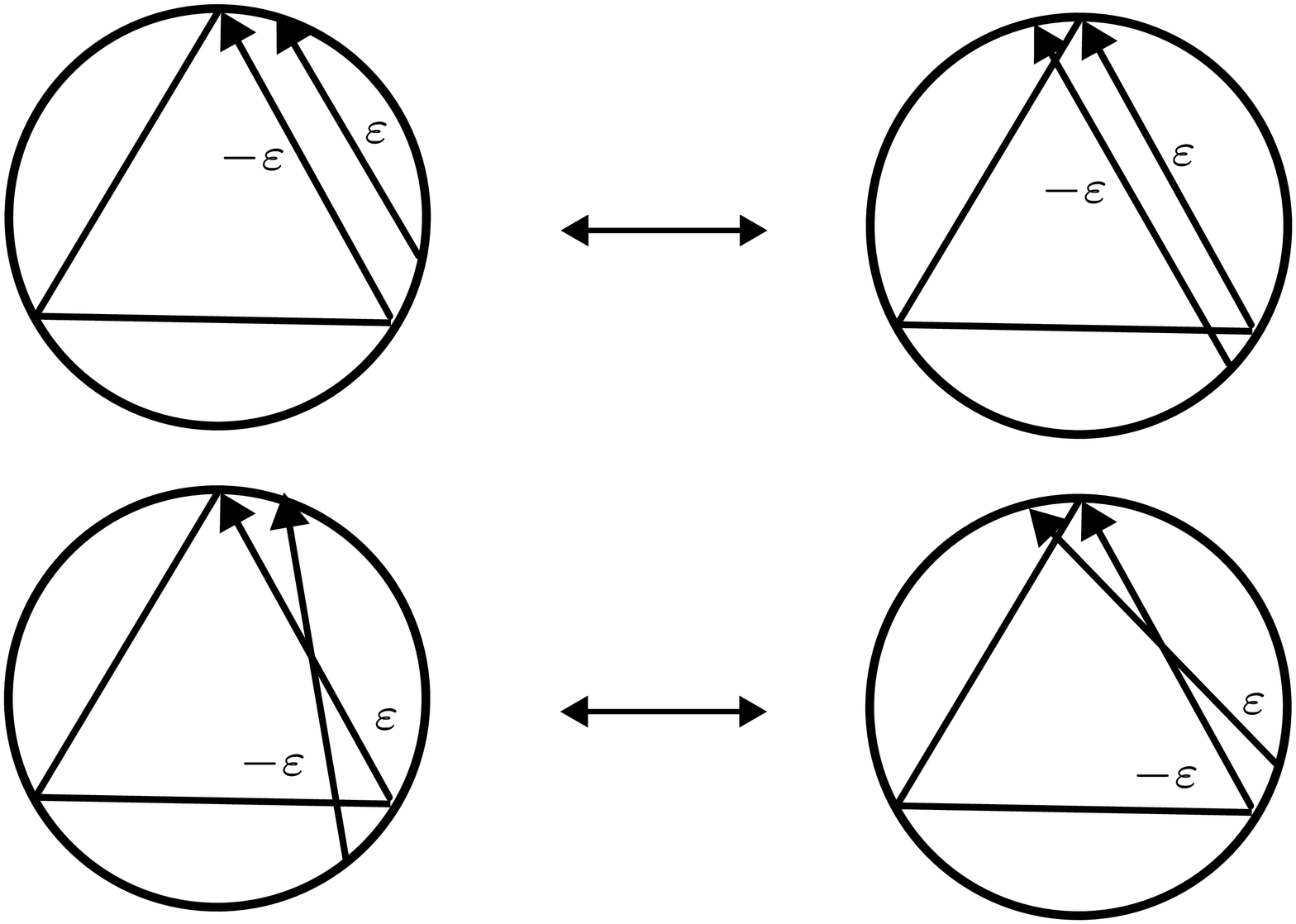}
		
		\label{Fig.trans-self_arrow_slide}
		\end{figure}

We will note $p$ the triple point associated with the triangle in the left side of the figure above, and similarly $p'$ for the right side. They have the same global type. The corresponding arrows are noted $d$, $hm$, $ml$, $d'$, $hm'$, $ml'$. Within the pair of arrows that differ from one side to the other, the arrow belonging to $p$ is noted $q$, the one belonging to $p'$ is noted $q'$. Clearly $q$ is $d$ if and only if $q'$ is $d'$, and the same holds for $hm$ and $ml$.

Now remember that $\gamma$ only contributes to $R$ when crossing this strata if the Reidemeister move is of type $r(a,n,a)$. Therefore unless $q$ is $hm$ and $q'$ is $hm'$, $q$ and $q'$ will have homological marking $a$ so the will not contribute to the weights. We will then have $W(p) = W(p')$ because any subdiagram that contributes in $G$ cannot contain $q$ therefore it also contributes in $G'$, with same sign, and conversely. We will then also have that $w(hm) = w(hm')$ and $sign(p) = - sign(p')$. Therefore, unless $q$ and $q'$ are $hm$ and $hm'$ respectively, will have $R(\gamma) = 0$.

We now suppose that $q=hm$, $q' = hm'$, and that $p$ and $p'$ are of type $r(a,n,a)$. We want to check that contributions to $R$ cancel out; this time $\mbox{sign}(p) w(hm) =\mbox{sign}(p') w(hm') $. 

Let $S$ be a subdiagram of $G$. We first suppose that it has three arrows and that it doesn't contain the arrow $hm'$. The if we replace the $h$ by $hm'$ in $S$ and note $S'$ the subdiagram of $G'$ obtained, it contributes to $G'$ if and only if $S$ contributes to $G$, and the contributions have opposite signs, therefore they cancel each other. 

Then, we suppose that $S$ contains three arrows, including $hm$ and $hm'$. If it contributes to $W(p)$ or to $W(p')$, it has to be of the form \includegraphics[width = 1.2cm]{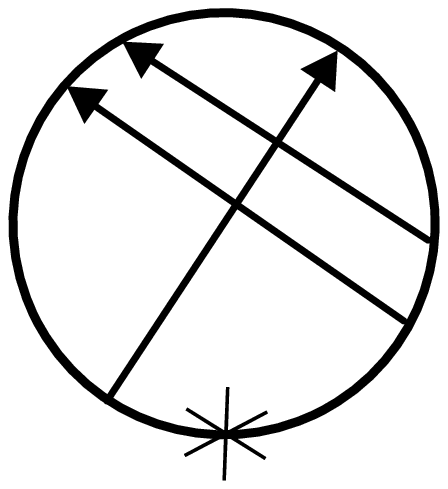} or \includegraphics[width = 1.2cm]{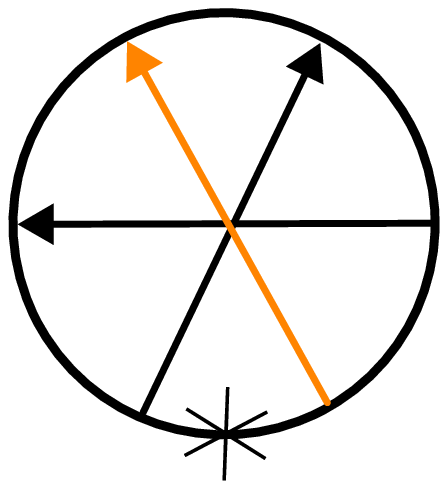}, depending on whether we are in the "above" or "below" case in the previous figure. The subdiagram $S$ either contributes to $W(p)$ or to $W(p')$, depending on which foot is closer to the point $\infty$, and the contribution has sign opposite to the writhe of the arrow that crosses both $hm$ and $hm'$. Therefore, the sum of all those contributions to $R$ is $- \mbox{sign}(p)w(hm)$ times the "number" of arrows crossing both $hm$ and $hm'$, going from $hm^+$ to $hm^-$, and with homological marking $0$, each arrow being counted with sign its writhe.

Finally, we look at the subdiagrams with only two crossings. Such a diagram cannot contribute if it contains both $hm$ and $hm'$. If it contributes to $G$ for example, it contains $hm$ and the second arrow has homological marking $0$, and crosses $hm$ from $hm^+$ to $hm^-$. Then the same must remain true if we replace $hm$ by $hm'$. If look back at the term of order two in the formula for $v_3$, we see that the sum of all those contribution to $R$ is  $ \mbox{sign}(p)w(hm)$ times the "number" of arrows crossing both $hm$ and $hm'$, going from $hm^+$ to $hm^-$, and with homological marking $0$, each arrow being counted with sign its writhe. Therefore this sum cancels out with the sum of contributions from the previous case.\\

Finally, we have shown that for $\gamma$ a small loop around a strata $\Sigma^2_{\mbox{self-flex}}$ , $R(\gamma) = 0$, in other words we have solved the cube equations.

\section{Tetrahedron equations}\label{s7}
Due to cube equations and commutation equations, we can restrict ourselves to the positive \textit{quadruple point}. 
Given 4 lines intersecting at one common points which is called a \textit{quadruple point}, with a perturbation we get 6 ordinary crossings. Every 3 lines can make a Reidemeister move $\uppercase\expandafter{\romannumeral3}$ hence there are 4 stratas of Reidemeister move $\uppercase\expandafter{\romannumeral3}$ in the moduli space. 
\begin{figure}[H]
\centering
\includegraphics[width= 0.46\textwidth]{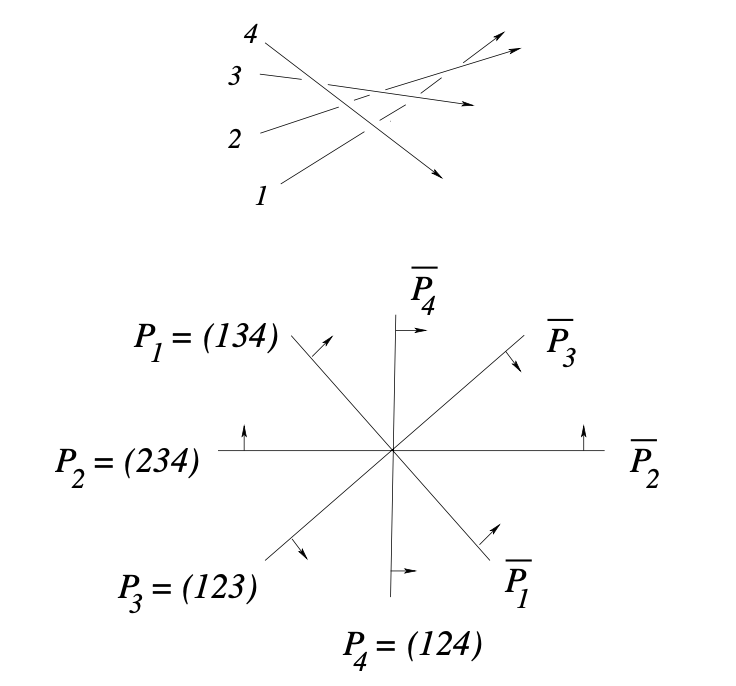}
\caption{quadruple point and 4 stratas. The numbers in the bracket indicate which 3 lines to perform a Reidemeister move.}
\label{Fig.quadruple}
\end{figure}
A meridian around the quadruple point is a circle going through the 4 stratas: 
\begin{figure}[H]
\centering
\includegraphics[width= 0.7\textwidth]{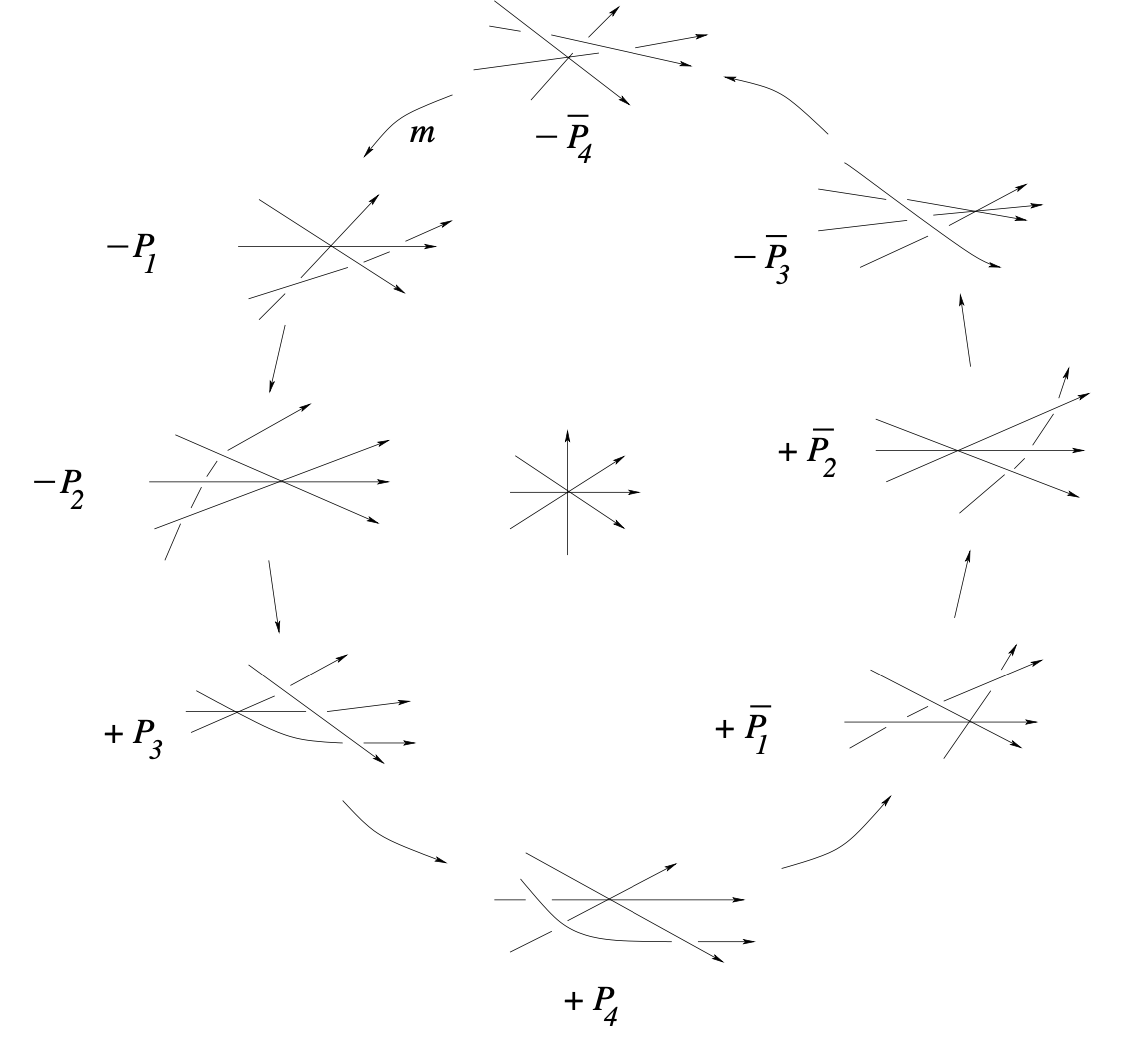}
\caption{meridian around a quadruple point}
\label{Fig.tetrahedron}
\end{figure}
Going along this circle, our one-cocycle equals to 0, which is the \textit{tetrahedron equation}. There are 6 global types of tetrahedron equations, 
whose Gauss diagram are displayed as: 

\begin{figure}[H]
\centering
\includegraphics[width= 0.7\textwidth]{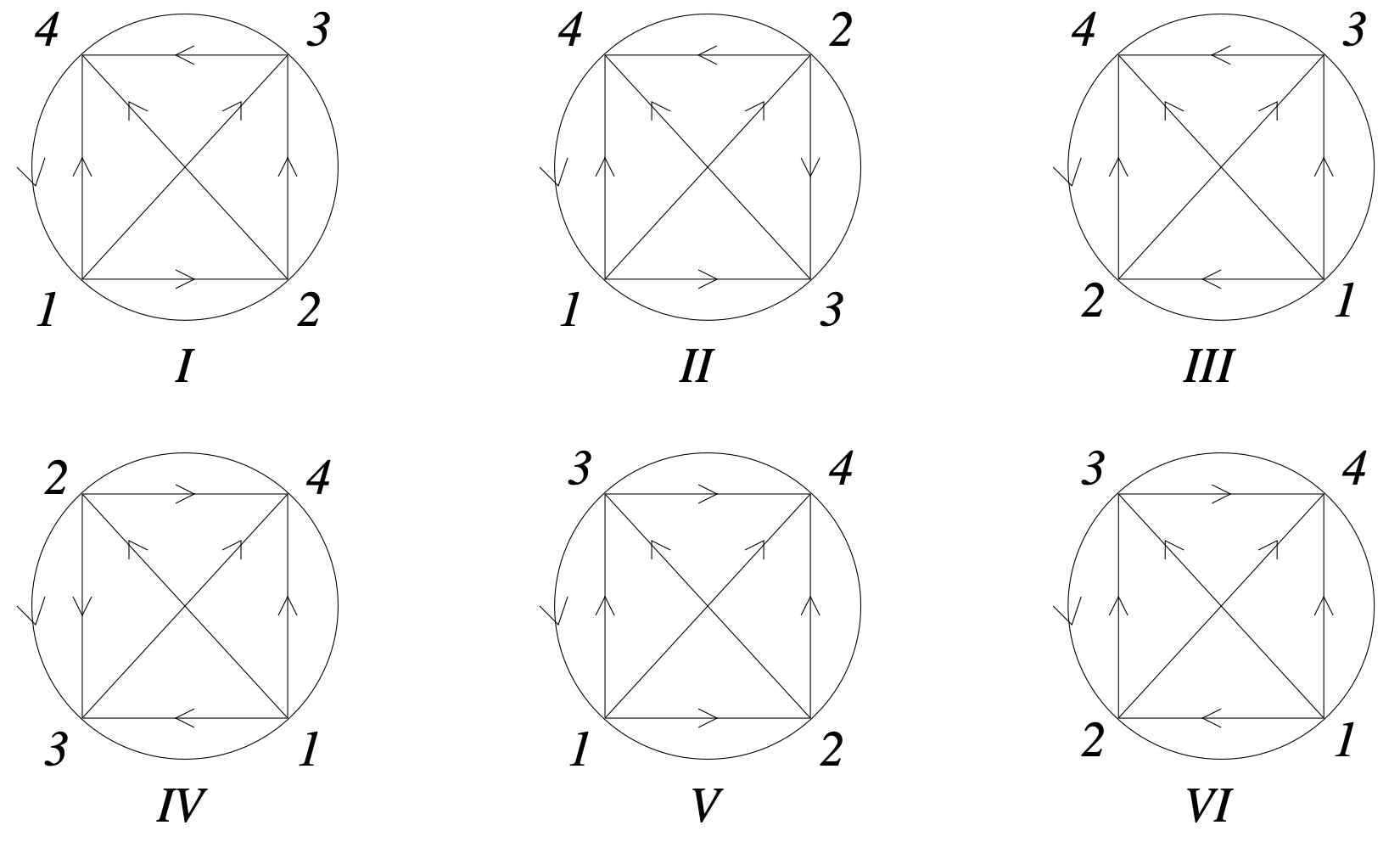}
\caption{6 global types of quadruple points}
\label{Fig.quadruple_global}
\end{figure}

\begin{lem}\label{Lem.marking}
Let $G\in \widetilde{M_n}$ be an Gauss diagram. If an arrow $\alpha$ in $G$ is a persistent crossing, then from $\infty$ as we go along the circle, we first encounter the foot (resp. head) of $\alpha$ if and only if $[\alpha] = n$ (resp. $0$). 
\end{lem}

To prove the tetrahedron equation for our one-cocycle, we need to calculate it in these 6 types. However, in Type $\uppercase\expandafter{\romannumeral6}$ all Reidemeister move $\uppercase\expandafter{\romannumeral3}$ are of global left type, so the tetrahedron equation is satisfied trivially. Let us work for the other 5 types. See figures \ref{Fig.Type1}, \ref{Fig.Type2}, \ref{Fig.Type3}, \ref{Fig.Type4}, \ref{Fig.Type5}. 
\begin{enumerate}
	\item\label{Type1}Type $\uppercase\expandafter{\romannumeral1}$. 
	To decide the possible cases in each type, we need to consider in two aspects: the Reidemeister move $\uppercase\expandafter{\romannumeral3}$ is $r(a,n,a)$ and the other 3 arrows are persistent crossings. 
	\begin{enumerate}
		\item $\alpha = a, \beta = n-a, \gamma = 0$:$P_2$,$\overline{P}_2$ do not contribute.
		By lemma \ref{Lem.marking}, $\infty = 1$. 
		By the first $hm$ condition, in $P_1$ and $\overline{P}_1$ only $hm$ could contribute, i.e., there exists $\phi \in \cup_{i} Hom(A_i,G_p)$ such that  $hm\in S(\phi)$ is the first image arrow from $\infty$. Then $W(P_1)=W(\overline{P}_1)$. 
		We notice that from $G_{P_3}=G_{P_4}$ to $G_{\overline{P}_3}=G_{\overline{P}_4}$ there performs a Reidemeister move $\uppercase\expandafter{\romannumeral3}$ by the three arrows $23, 24, 34$ of the case $2,R$ and that the difference between $G_{P_3}$ and $G_{P_4}$ is the place of $hm$. So do $G_{\overline{P}_3}$ and $G_{\overline{P}_4}$. We define $$(23,24)_{P_3\uppercase\expandafter{\romannumeral1}'[2]} = \sum\limits_{\phi \in \{ \phi \in Hom([2],G_{P_3}) | \{ 23, 24\}\subset S(\phi), \phi \text{ is of the type } \uppercase\expandafter{\romannumeral1}' \} }c_{i(\phi)} sign(\phi).$$ We are considering positive quadruple point, so $sign(\phi)=1$ and $c_{i(\phi)}=1$. Hence $(23,24)_{P_3\uppercase\expandafter{\romannumeral1}'[2]}=\#\{ \phi \in Hom([2],G_{P_3}) | \{ 23, 24\}\subset S(\phi), \phi \text{ is of the type } \uppercase\expandafter{\romannumeral1}'  \}$and we have 
		\begin{equation}
		\begin{aligned}
		R &= W(P_3)+W(P_4)-W(\overline{P}_3)-W(\overline{P}_4)\\
		&= (23,24)_{P_3\uppercase\expandafter{\romannumeral2}'[2]}+(23,24)_{P_3\uppercase\expandafter{\romannumeral2}'[5]}+(24,34)_{P_4\uppercase\expandafter{\romannumeral1}'[2]}\\
		&-[(23,34)_{\overline{P}_3\uppercase\expandafter{\romannumeral3}[5]}+(24,34)_{\overline{P}_4\uppercase\expandafter{\romannumeral1}[3]}+(24,23)_{\overline{P}_4\uppercase\expandafter{\romannumeral2}[3]}]\\
		\end{aligned}
		\end{equation}
		By the bijective map $\theta$ respect to the Reidemeister move $\uppercase\expandafter{\romannumeral3}$ of the three arrows $23, 24, 34$, we have $$(23,24)_{P_3\uppercase\expandafter{\romannumeral2}'[2]}= (24,23)_{\overline{P}_4\uppercase\expandafter{\romannumeral2}[3]}$$
		$$(23,24)_{P_3\uppercase\expandafter{\romannumeral2}'[5]} = (23,34)_{\overline{P}_3\uppercase\expandafter{\romannumeral3}[5]}$$
		$$(24,34)_{P_4\uppercase\expandafter{\romannumeral1}'[2]} = (24,34)_{\overline{P}_4\uppercase\expandafter{\romannumeral1}[3]}$$
		Hence $R=0$. 
		
		\item $\alpha = 0, \beta = n-a, \gamma = 0$: $P_1$,$\overline{P}_1$, $P_2$,$\overline{P}_2$, $P_4$,$\overline{P}_4$ do not contribute. By the first $hm$ condition, in $P_3$ and $\overline{P}_3$ only $hm$ could contribute, i.e., there exists $\phi \in \cup_{i} Hom(A_i,G_p)$ such that  $hm\in S(\phi)$ is the first image arrow from $\infty$. Then $W(P_3)=W(\overline{P}_3)$ and $R=0$. 
		\item $\alpha = 0, \beta = a, \gamma = n-a$: $P_1$,$\overline{P}_1$, $P_3$,$\overline{P}_3$, $P_4$,$\overline{P}_4$ do not contribute. 
		By lemma \ref{Lem.marking}, $\infty = 4$. We notice that from $G_{P_2}$ to $G_{\overline{P}_2}$ there performs a Reidemeister move $\uppercase\expandafter{\romannumeral3}$ of the case $3,R$. $W(P_2) = (34,14)_{P_2\uppercase\expandafter{\romannumeral1}'[1]}+(34,14)_{P_2\uppercase\expandafter{\romannumeral1}'[2]}+(34,14)_{P_2\uppercase\expandafter{\romannumeral1}'[4]}+(34,14)_{P_2\uppercase\expandafter{\romannumeral1}'[6]}.$
		$W(\overline{P_2}) = (34,13)_{\overline{P}_2\uppercase\expandafter{\romannumeral3}[4]}+(34,14)_{\overline{P}_2\uppercase\expandafter{\romannumeral1}[5]}+(34,13)_{\overline{P}_2\uppercase\expandafter{\romannumeral3}[1]}+(34,13)_{\overline{P}_2\uppercase\expandafter{\romannumeral3}[6]}.$
		By the bijective map $\theta$ respective to the Reidemeister move $\uppercase\expandafter{\romannumeral3}$ of the three arrows $34, 13, 14$, $W(P_2) = W(\overline{P_2}) = $ and $R = 0$. 
		\item $\alpha = a, \beta = 0, \gamma = n-a$: $P_3$,$\overline{P}_3$, $P_4$,$\overline{P}_4$ do not contribute. $P_1$ with $\overline{P}_1$ and $P_4$ with $\overline{P}_4$ cancell out by the first $hm$ condition.
	\end{enumerate}

	\item\label{Type2}Type $\uppercase\expandafter{\romannumeral2}$. Only $P_1, \overline{P_1}, P_4, \overline{P}_4$ are of right global type. 
	\begin{enumerate}
		\item $\alpha = a, \beta = 0, \gamma = n-a$: We only need to consider about $\infty = 1$. 
		By the first $hm$ condition, $W(P_4) = W(\overline{P}_4)$. From $P_1$ to $\overline{P}_1$, there is a Reidemeister move $\uppercase\expandafter{\romannumeral3}$ of case $2,L$. We have $W(P_1) = (34,24)_{P_1\uppercase\expandafter{\romannumeral1}[2]}$ and $W(\overline{P}_1) = (34,24)_{\overline{P}_1\uppercase\expandafter{\romannumeral1}'[3]}$. By the one-one contribution correspondence, $W(P_1) = W(\overline{P}_1)$ and hence $R = 0$. 
		\item $\alpha = a, \beta = n-a, \gamma = 0$: $P_1$, $\overline{P}_1$ do not contribute. 
		By the first $hm$ condition, $W(P_4) = W(\overline{P}_4)$ and then $R=0$.
	\end{enumerate}
	
	\item\label{Type3}Type $\uppercase\expandafter{\romannumeral3}$. Only $P_1, \overline{P_1}, P_2, \overline{P}_2$ are of right global type. 
	\begin{enumerate}
		\item $\alpha = 0, \beta = a, \gamma = n-a$: $\infty = 1$ by lemma \ref{Lem.marking}. 
		Due to the first $hm$ condition, $W(P_1) = W(\overline{P}_1)$  and $W(P_2) = W(\overline{P}_2)$. Hence $R=0$. 
		\item $\alpha = n-a, \beta = a, \gamma = 0$: $P_1$, $\overline{P}_1$ do not contribute and $\infty = 1$. 
		By the first $hm$ condition, $W(P_2) = W(\overline{P}_2)$. So $R=0$. 
		\item $\alpha = a, \beta = 0, \gamma = n-a$: $P_2$, $\overline{P}_2$ do not contribute and $\infty = 3$ by lemma \ref{Lem.marking}. 
		From $P_1$ to $\overline{P}_1$, there is a Reidemeister move $\uppercase\expandafter{\romannumeral3}$ of case $3,R$. We have $W(P_1) = (34,24)_{P_1\uppercase\expandafter{\romannumeral1}[5]}+(34,23)_{P_1\uppercase\expandafter{\romannumeral3}[1]}+(34,23)_{P_1\uppercase\expandafter{\romannumeral3}[4]}+(34,23)_{P_1\uppercase\expandafter{\romannumeral3}[6]} = (34,24)_{\overline{P}_1\uppercase\expandafter{\romannumeral1}'[2]}+(34,24)_{\overline{P}_1\uppercase\expandafter{\romannumeral1}'[4]}+(34,24)_{\overline{P}_1\uppercase\expandafter{\romannumeral1}'[1]}+(34,24)_{\overline{P}_1\uppercase\expandafter{\romannumeral1}'[6]}= W(\overline{P}_1)$. So $R=0$. 
	\end{enumerate}
	\item\label{Type4}Type $\uppercase\expandafter{\romannumeral4}$. Only $P_2, \overline{P_2}, P_3, \overline{P}_3$ are of right global type. 
	\begin{enumerate}
		\item $\alpha = 0, \beta = a, \gamma = 0$:  $P_3$ and $\overline{P}_3$ do not contribute and $\infty = 1$. 
		From $P_2$ to $\overline{P}_2$, there is a Reidemeister move $\uppercase\expandafter{\romannumeral3}$ of case $2,L$. We have $W(P_2) = (34,14)_{P_2\uppercase\expandafter{\romannumeral1}'[3]} = (34,14)_{\overline{P}_2\uppercase\expandafter{\romannumeral1}[2]} = W(\overline{P}_2)$. So $R = 0$.  
		\item $\alpha = a, \beta = 0, \gamma = n-a$: $P_2$ and $\overline{P}_2$ do not contribute and $\infty = 2$. $W(P_3) = W(\overline{P}_2)$ by the first $hm$ condition and hence $R=0$.
		\item $\alpha = a, \beta = n-a, \gamma = 0$: $P_2$ and $\overline{P}_2$ do not contribute and $\infty = 2$. $W(P_3) = W(\overline{P}_2)$ by the first $hm$ condition and hence $R=0$.
	\end{enumerate}
	
	\item\label{Type5}Type $\uppercase\expandafter{\romannumeral5}$. Only $P_3, \overline{P_3}, P_4, \overline{P}_4$ are of right global type. 
	From $G_{P_3}=G_{P_4}$ to $G_{\overline{P}_3}=G_{\overline{P}_4}$ there performs a Reidemeister move $\uppercase\expandafter{\romannumeral3}$ by the three arrows $23, 24, 34$ of the case $3,L$. 
	\begin{enumerate}
		\item $\alpha = a, \beta = 0, \gamma = n-a$: We only need to consider $\infty=2$. 
		We have 
		\begin{equation}
		\begin{aligned}
		R &= W(P_3)+W(P_4)-W(\overline{P}_3)-W(\overline{P}_4)\\
		&= (23,24)_{P_3\uppercase\expandafter{\romannumeral2}'[3]}+(23,34)_{P_3\uppercase\expandafter{\romannumeral3}'[1]}+(23,34)_{P_3\uppercase\expandafter{\romannumeral3}'[3]}+(23,34)_{P_3\uppercase\expandafter{\romannumeral3}'[4]}+(24,34)_{P_4\uppercase\expandafter{\romannumeral1}'[5]}+(23,34)_{P_3\uppercase\expandafter{\romannumeral3}'[6]}\\
		&-[0+(24,34)_{\overline{P}_4\uppercase\expandafter{\romannumeral1}[1]}+(24,34)_{\overline{P}_4\uppercase\expandafter{\romannumeral1}[2]}+(24,34)_{\overline{P}_4\uppercase\expandafter{\romannumeral1}[4]}+(24,34)_{\overline{P}_4\uppercase\expandafter{\romannumeral2}[2]}+(24,34)_{\overline{P}_4\uppercase\expandafter{\romannumeral2}[5]}+(24,34)_{\overline{P}_4\uppercase\expandafter{\romannumeral1}[6]}]\\
		&=0
		\end{aligned}
		\end{equation}
		\item $\alpha = 0, \beta = a, \gamma = n-a$: $P_3$ and $\overline{P}_3$ do not contribute and $\infty = 1$ by lemma $\ref{Lem.marking}$. 
		By the first $hm$ condition, $W(P_4) = W(\overline{P}_4)$ and $R=0$
	\end{enumerate} 
\end{enumerate}
Tetrahedron equations are proved. 

\begin{figure}[H]
	\begin{subfigure}{0.26\textwidth}
		\includegraphics[width=0.9\textwidth]{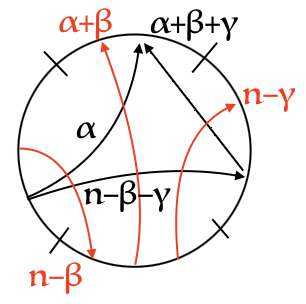} 
		\caption{-$P_1$} 
		\label{Fig.T1P1} 
	\end{subfigure}
	\begin{subfigure}{0.26\textwidth}
		\includegraphics[width=0.9\textwidth]{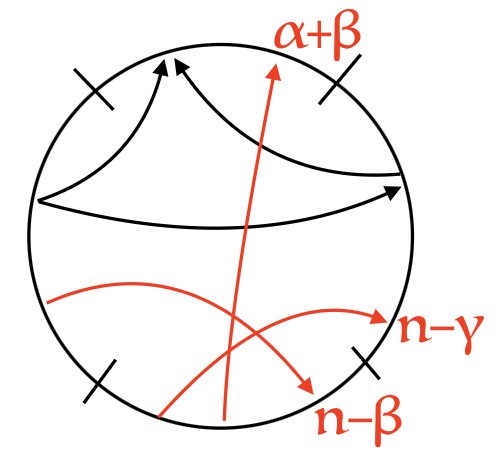} 
		\caption{$+\overline{P}_1$}
		\label{Fig.T1P'1} 
	\end{subfigure}
	\begin{subfigure}{0.26\textwidth}
		\includegraphics[width=\textwidth]{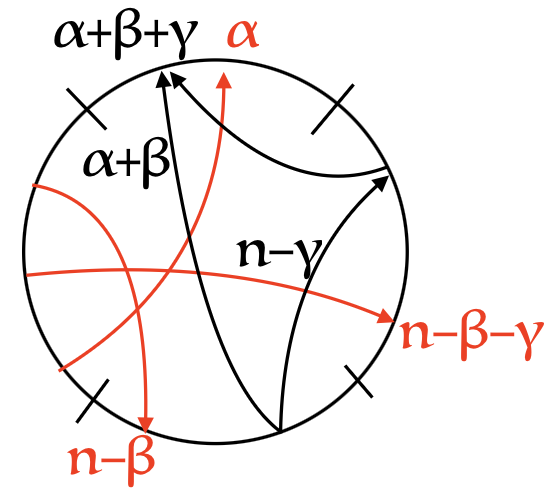} 
		\caption{-$P_2$} 
		\label{Fig.T1P2} 
	\end{subfigure}
	\begin{subfigure}{0.26\textwidth}
		\includegraphics[width=0.9\textwidth]{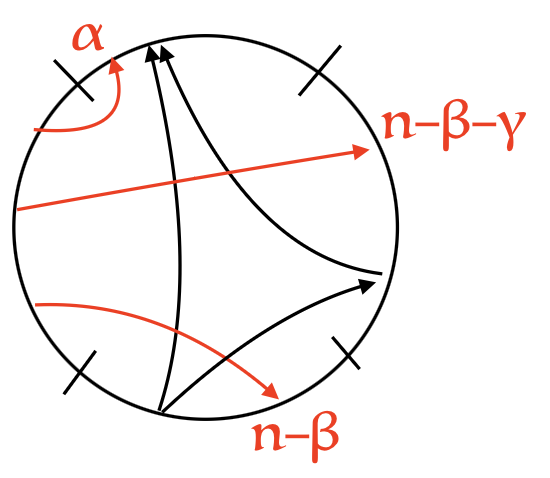} 
		\caption{$+\overline{P}_2$}
		\label{Fig.T1P'2} 
	\end{subfigure}
	\begin{subfigure}{0.26\textwidth}
		\includegraphics[width=0.9\textwidth]{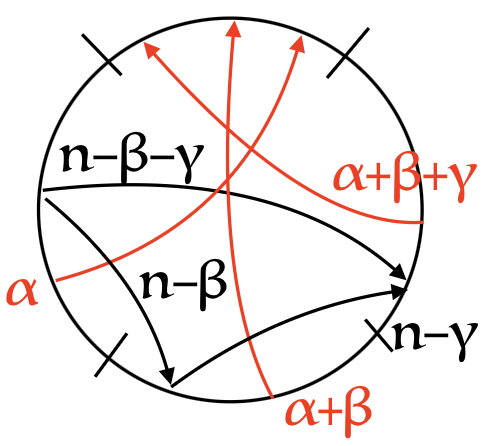} 
		\caption{+$P_3$} 
		\label{Fig.T1P3} 
	\end{subfigure}
	\begin{subfigure}{0.26\textwidth}
		\includegraphics[width=0.9\textwidth]{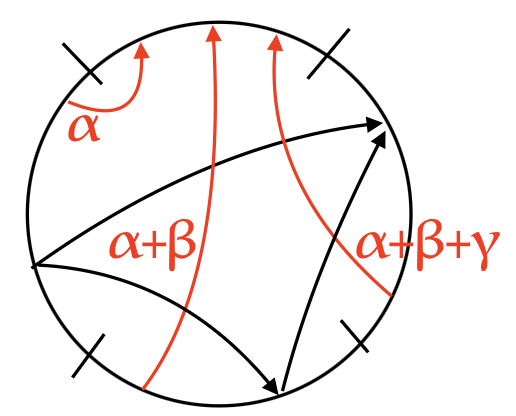} 
		\caption{$-\overline{P}_3$}
		\label{Fig.T1P'3} 
	\end{subfigure}
	\begin{subfigure}{0.26\textwidth}
		\includegraphics[width=0.9\textwidth]{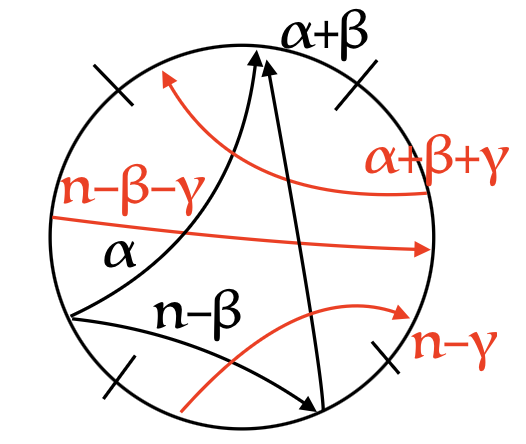} 
		\caption{+$P_4$} 
		\label{Fig.T1P4} 
	\end{subfigure}
	\begin{subfigure}{0.26\textwidth}
		\includegraphics[width=0.9\textwidth]{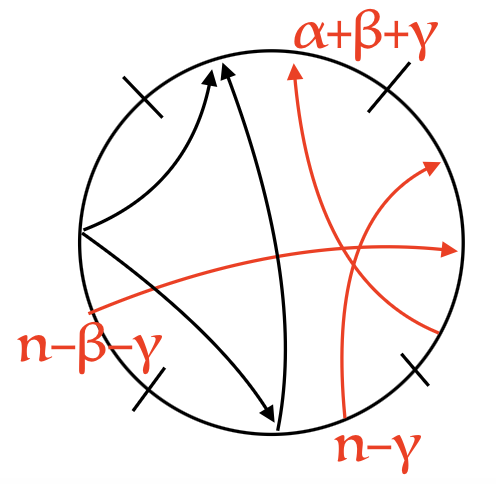} 
		\caption{$-\overline{P}_4$}
		\label{Fig.T1P'4} 
	\end{subfigure}
	\begin{subfigure}{0.26\textwidth}
		\includegraphics[width=1.15\textwidth]{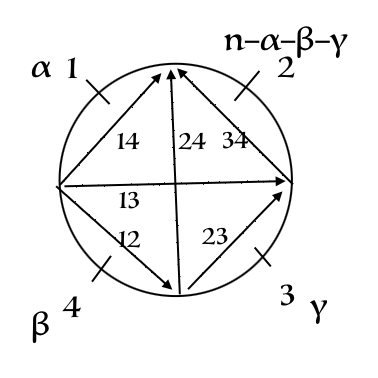} 
		\caption{Global type \uppercase\expandafter{\romannumeral1}}
		\label{Fig.T1} 
	\end{subfigure}
	\caption{Type $\uppercase\expandafter{\romannumeral1}$}
	\label{Fig.Type1}
\end{figure}

\begin{figure}[H]
	\begin{subfigure}{0.25\textwidth}
		\includegraphics[width=0.9\textwidth]{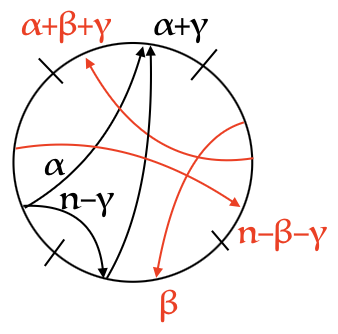} 
		\caption{-$P_1$} 
		\label{Fig.T2P1} 
	\end{subfigure}
	\begin{subfigure}{0.25\textwidth}
		\includegraphics[width=0.9\textwidth]{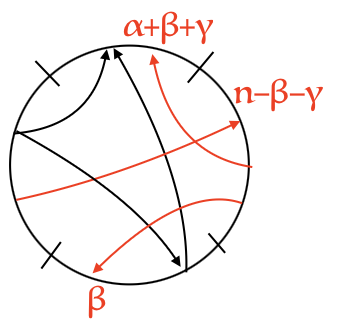} 
		\caption{$+\overline{P}_1$}
		\label{Fig.T2P'1} 
	\end{subfigure}
	\begin{subfigure}{0.25\textwidth}
		\includegraphics[width=0.9\textwidth]{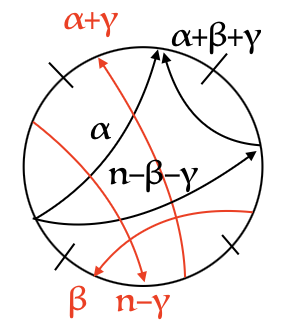} 
		\caption{+$P_4$} 
		\label{Fig.T2P4} 
	\end{subfigure}
	\begin{subfigure}{0.25\textwidth}
		\includegraphics[width=0.9\textwidth]{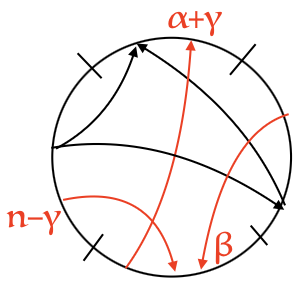} 
		\caption{$-\overline{P}_4$}
		\label{Fig.T2P'4} 
	\end{subfigure}
	\begin{subfigure}{0.25\textwidth}
		\includegraphics[width=0.9\textwidth]{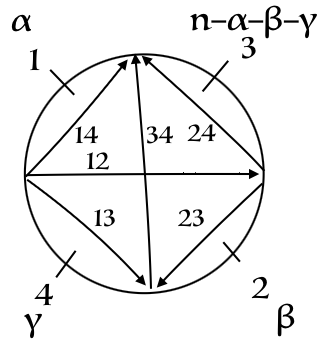} 
		\caption{Global type \uppercase\expandafter{\romannumeral2}}
		\label{Fig.T2} 
	\end{subfigure}
	\caption{Type $\uppercase\expandafter{\romannumeral2}$}
	\label{Fig.Type2}
\end{figure}

\begin{figure}[H]
	\begin{subfigure}{0.25\textwidth}
		\includegraphics[width=0.9\textwidth]{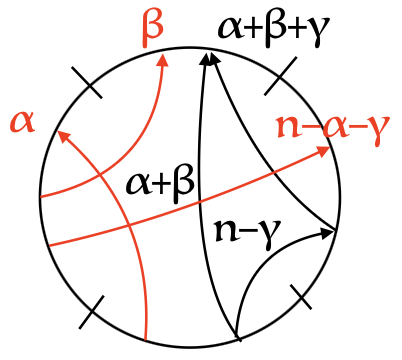} 
		\caption{-$P_1$} 
		\label{Fig.T3P1} 
	\end{subfigure}
	\begin{subfigure}{0.25\textwidth}
		\includegraphics[width=0.9\textwidth]{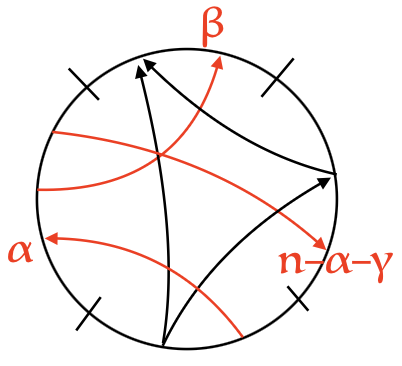} 
		\caption{$+\overline{P}_1$}
		\label{Fig.T3P'1} 
	\end{subfigure}
	\begin{subfigure}{0.25\textwidth}
		\includegraphics[width=0.9\textwidth]{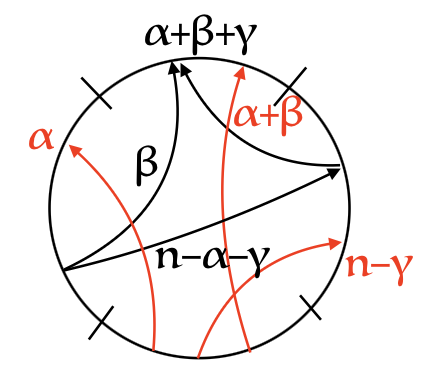} 
		\caption{-$P_2$} 
		\label{Fig.T3P2} 
	\end{subfigure}
	\begin{subfigure}{0.25\textwidth}
		\includegraphics[width=0.9\textwidth]{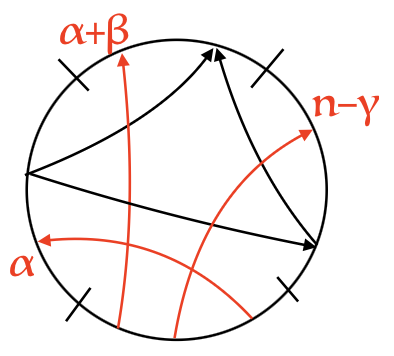} 
		\caption{$+\overline{P}_2$}
		\label{Fig.T3P'2} 
	\end{subfigure}
	\begin{subfigure}{0.25\textwidth}
		\includegraphics[width=0.9\textwidth]{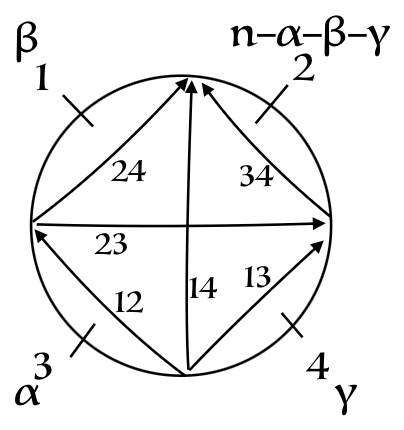} 
		\caption{Global type \uppercase\expandafter{\romannumeral3}}
		\label{Fig.T3} 
	\end{subfigure}
	\caption{Type $\uppercase\expandafter{\romannumeral3}$}
	\label{Fig.Type3}
\end{figure}

\begin{figure}[H]
	\begin{subfigure}{0.25\textwidth}
		\includegraphics[width=0.9\textwidth]{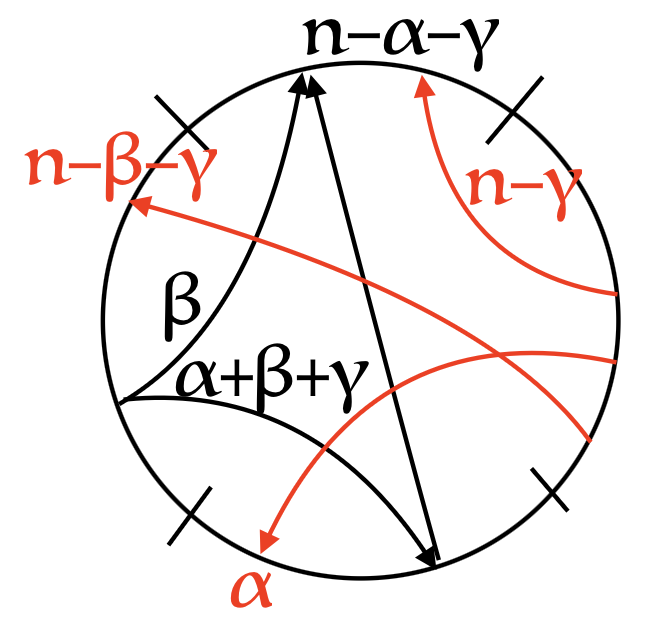} 
		\caption{-$P_2$} 
		\label{Fig.T4P2} 
	\end{subfigure}
	\begin{subfigure}{0.25\textwidth}
		\includegraphics[width=0.9\textwidth]{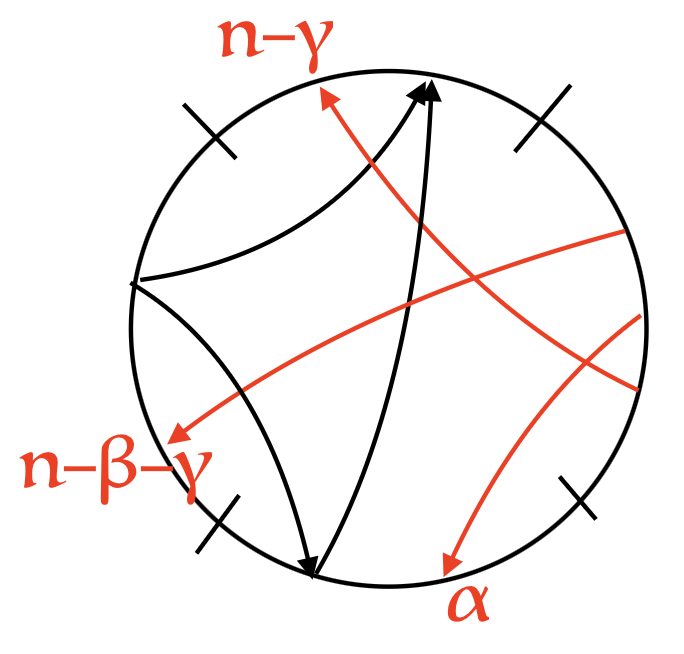} 
		\caption{$+\overline{P}_2$}
		\label{Fig.T4P'2} 
	\end{subfigure}
	\begin{subfigure}{0.25\textwidth}
		\includegraphics[width=0.9\textwidth]{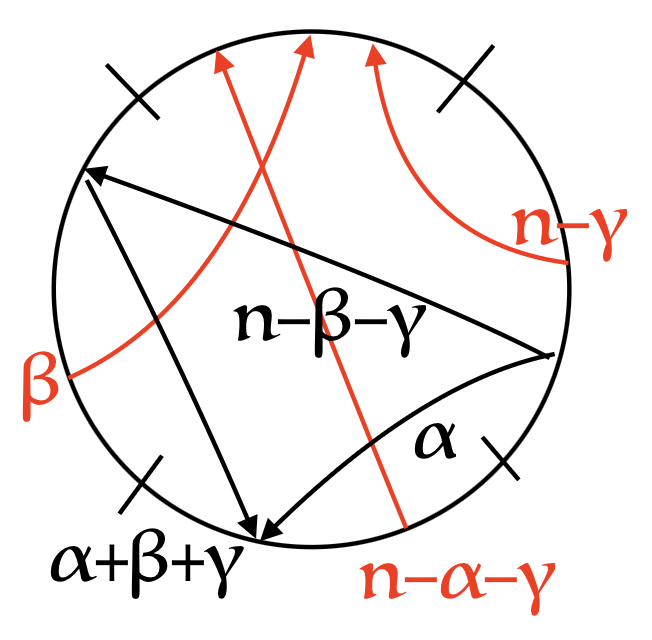} 
		\caption{+$P_3$} 
		\label{Fig.T4P3} 
	\end{subfigure}
	\begin{subfigure}{0.25\textwidth}
		\includegraphics[width=0.9\textwidth]{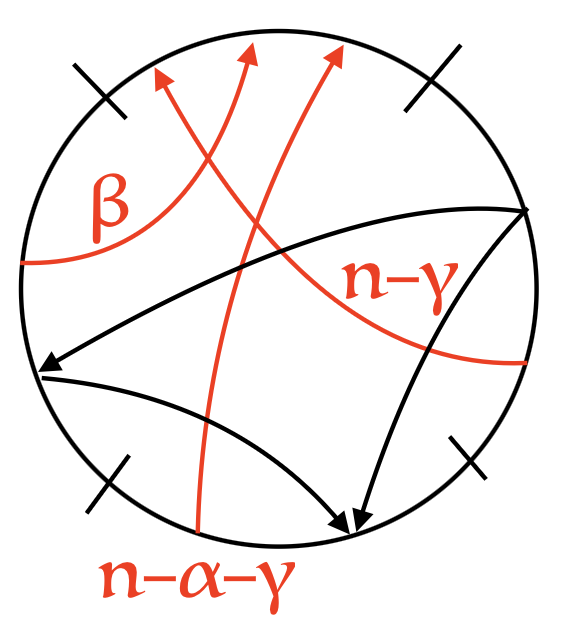} 
		\caption{$-\overline{P}_3$}
		\label{Fig.T4P'3} 
	\end{subfigure}
	\begin{subfigure}{0.25\textwidth}
		\includegraphics[width=0.9\textwidth]{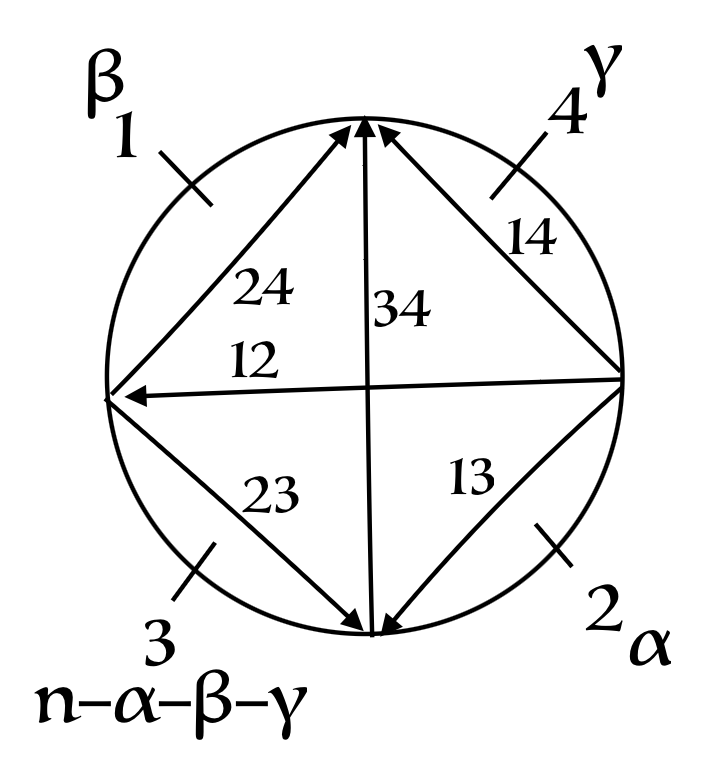} 
		\caption{Global type \uppercase\expandafter{\romannumeral4}}
		\label{Fig.T4} 
	\end{subfigure}
	\caption{Type $\uppercase\expandafter{\romannumeral4}$}
	\label{Fig.Type4}
\end{figure}

\begin{figure}[H]
	\begin{subfigure}{0.25\textwidth}
		\includegraphics[width=0.9\textwidth]{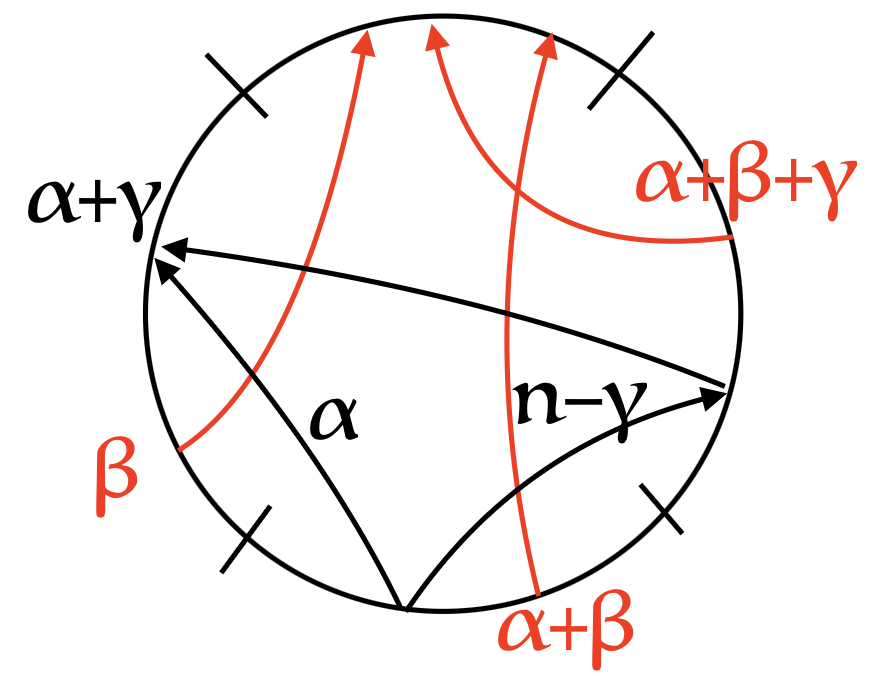} 
		\caption{+$P_3$} 
		\label{Fig.T5P3} 
	\end{subfigure}
	\begin{subfigure}{0.25\textwidth}
		\includegraphics[width=0.9\textwidth]{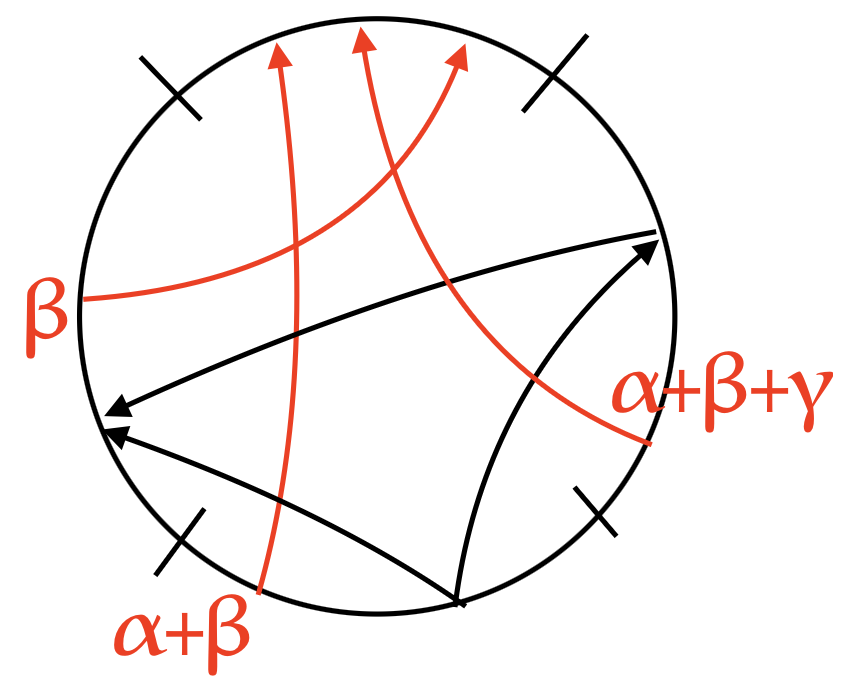} 
		\caption{$-\overline{P}_3$}
		\label{Fig.T5P'3} 
	\end{subfigure}
	\begin{subfigure}{0.25\textwidth}
		\includegraphics[width=0.9\textwidth]{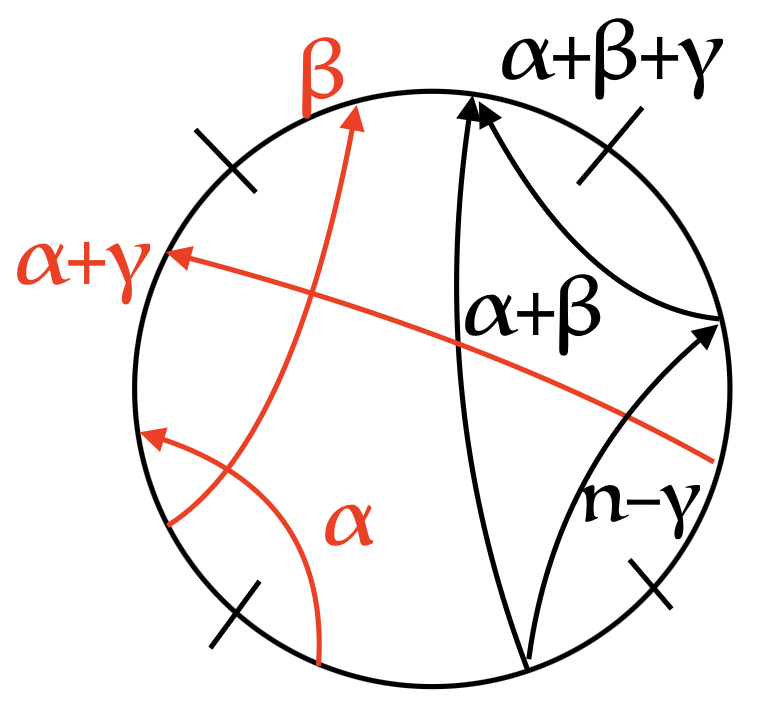} 
		\caption{+$P_4$} 
		\label{Fig.T5P4} 
	\end{subfigure}
	\begin{subfigure}{0.25\textwidth}
		\includegraphics[width=0.9\textwidth]{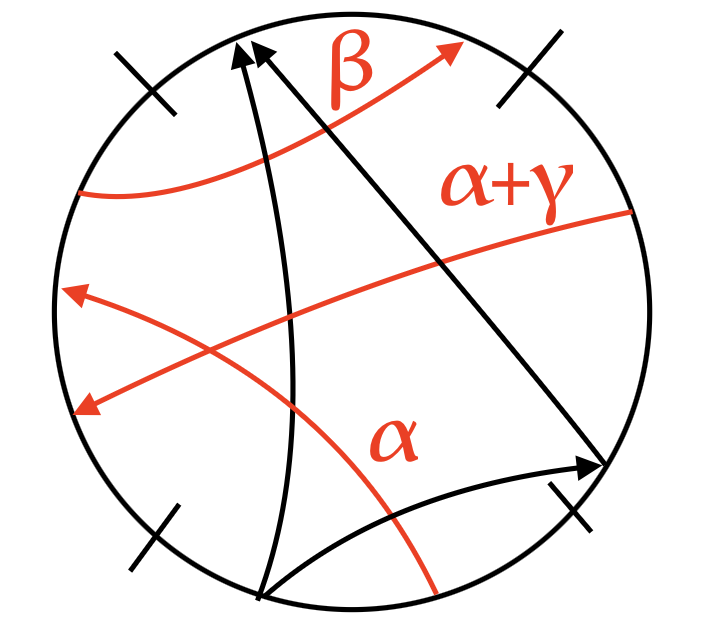} 
		\caption{$-\overline{P}_4$}
		\label{Fig.T5P'4} 
	\end{subfigure}
	\begin{subfigure}{0.25\textwidth}
		\includegraphics[width=0.9\textwidth]{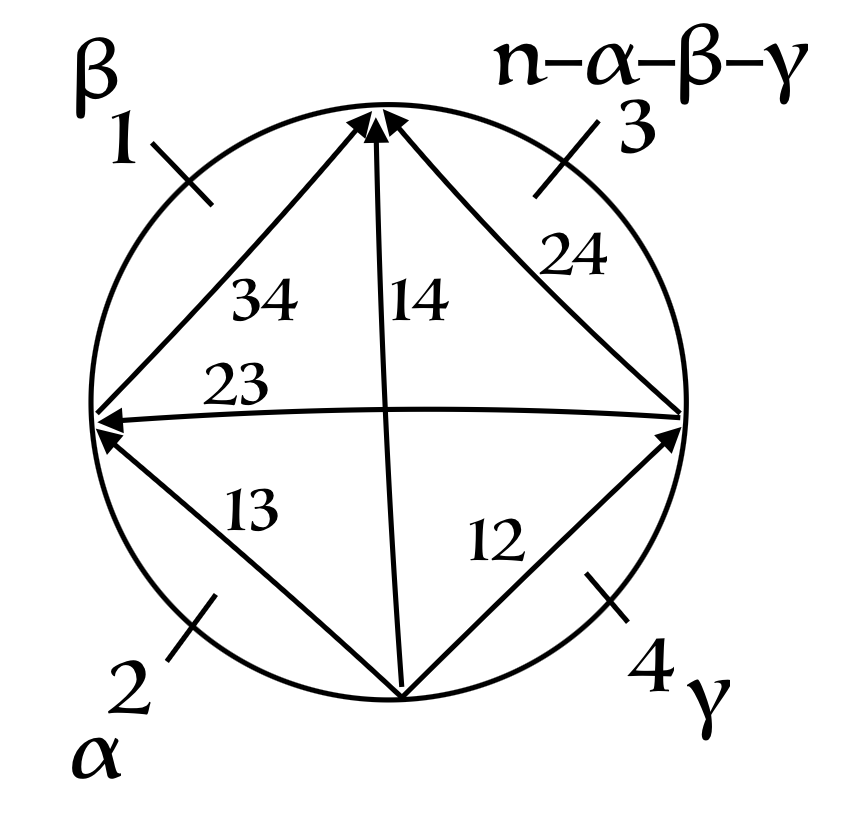} 
		\caption{Global type \uppercase\expandafter{\romannumeral5}}
		\label{Fig.T5} 
	\end{subfigure}
	\caption{Type $\uppercase\expandafter{\romannumeral5}$}
	\label{Fig.Type5}
\end{figure}

\begin{rem}
The neighbourhoods near the 4 points on each circle are not decided while the 4 arcs subtracting the 4 neighbourhood are exactly what they look like. 
\end{rem}


\section{Examples}\label{s8}

One goal of these 1-cocycles is to find new knot invariants. The general method, explained in more details in \cite{4, 5}, consists in associating, to each knot in $M_n$, a canonical loop in $M_n$, based on this knot, and then evaluating the 1-cocycle on this loop. These canonical loops are only defined on some of the knots in $M_n$. Therefore, the first step consists in taking a knot in the usual sense, and sending it to a "nice knot" in $M_n$, i.e. a knot on which the canonical loops are defined.
Let us start with a long knot $K'$. We take $V'$ a tubular neighborhood of $K'$. We can then consider a set of $n$ lines in an usual cylinder $V$, so that an homeomorphism between $V$ and $V'$ sends those $n$ lines to $n$ stacked copies of $K'$ in the space, and we note $nK'$ the result. To obtain a knot we use an $n$-braid  $\sigma$ that induces a cyclical permutation on its strands to close this link : we glue $nK$ to $\sigma$, and then close the result to a knot in the solid torus. This knot is noted $K = \sigma \cup nK'$. By hypothesis this knot is in $M_n$. For example see (16) in Figure \ref{Fig.push}.

With $K = \sigma \cup nK'$, we note $\mbox{push}(nK',\sigma)$ the loop based on $K$ obtained by pushing $\sigma$ through the cable around the torus, until its starting point. Such a loop is depicted in Figure \ref{Fig.push} with $n=2$ and $K$ a right-handed trefoil. 

We can easily compute our 1-cocycle valuated on the loop above. For $n=2$ $K$ a right-handed trefoil (and therefore $a=1$). In this case, there is only one RIII move that is counted towards $R$ (the move from (3) to (4)).
For this move the persistent Gauss diagram is:
\begin{figure}[H]
	\centering
	\includegraphics[width=0.4\textwidth]{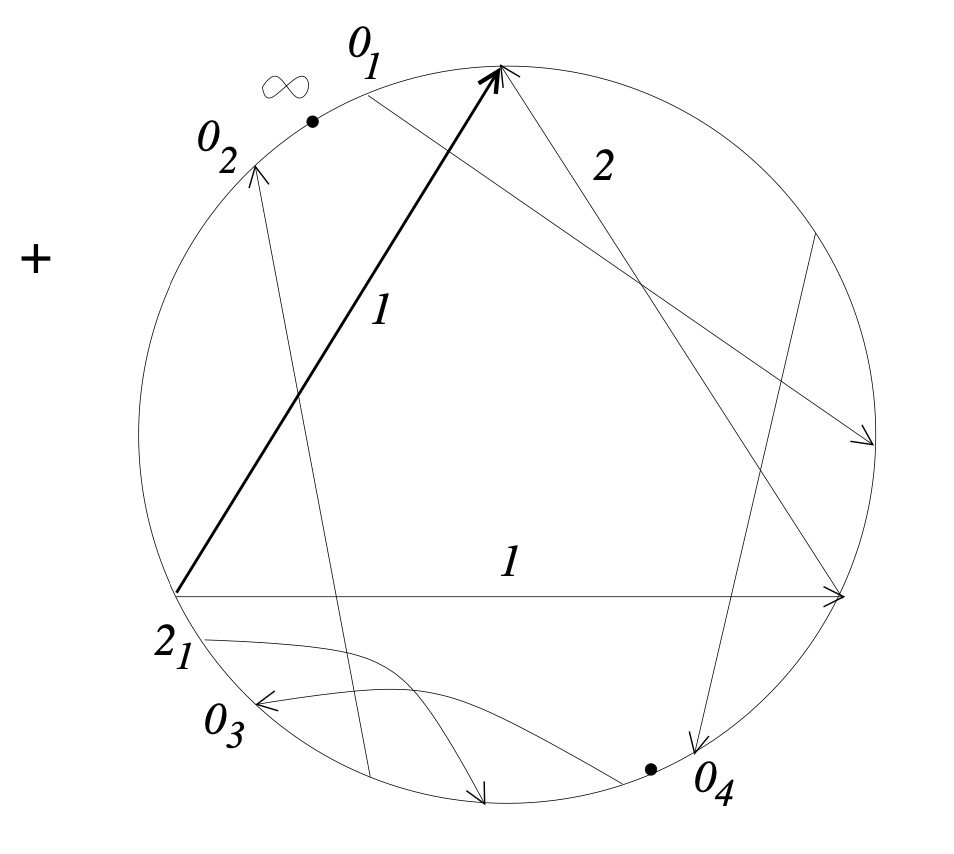}
	\caption{Contributing Reidemeister move III}
	\label{Fig.R3contributing}
\end{figure}
Hence $R(push(2-\mbox{right-handed trefoil},\sigma) = 1$ by the definition of the 1-cocycle. A similar calculation shows that $R(push(2-\mbox{left-handed trefoil},\sigma) = -1$.
In fact, for any long knot $K$ we have $R(push(2K,\sigma_1))=v_3(K)$, which will be discussed in detail in the Ph.D. thesis of the second author (also see \cite{6}).
\begin{figure}[h]
  \centering
  \begin{tikzpicture}
    
    \node[inner sep=0pt] at (-8,6) {\includegraphics[width = 0.35 \textwidth]{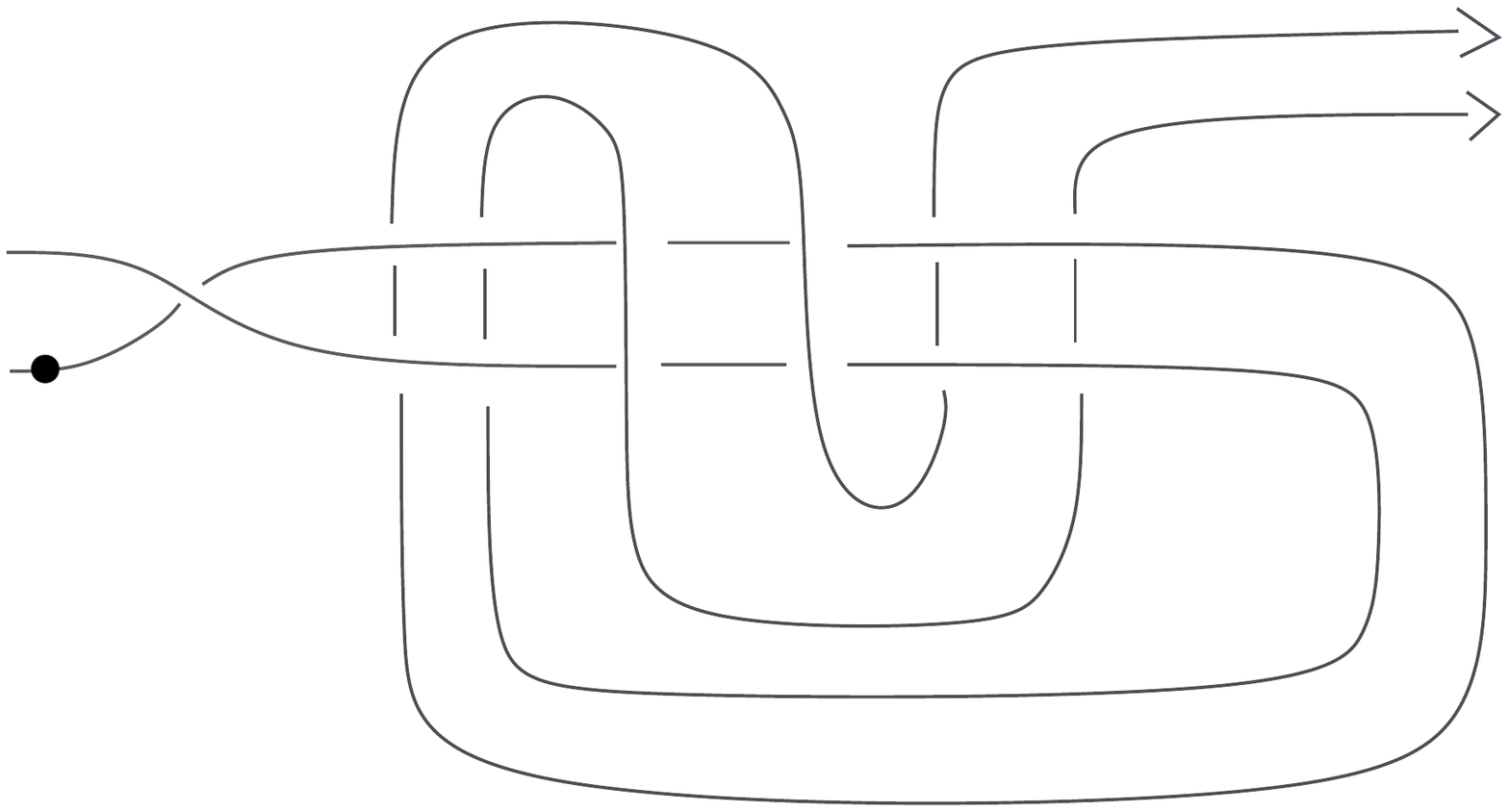}}; 
    \draw(-8.4,4.7) node {(1)};
    \draw[->,very thick] (-6.8,6) -- (-6.1,6) node[right] {};

    \node[inner sep=0pt] at (-3.6,6) {\includegraphics[width = 0.35 \textwidth]{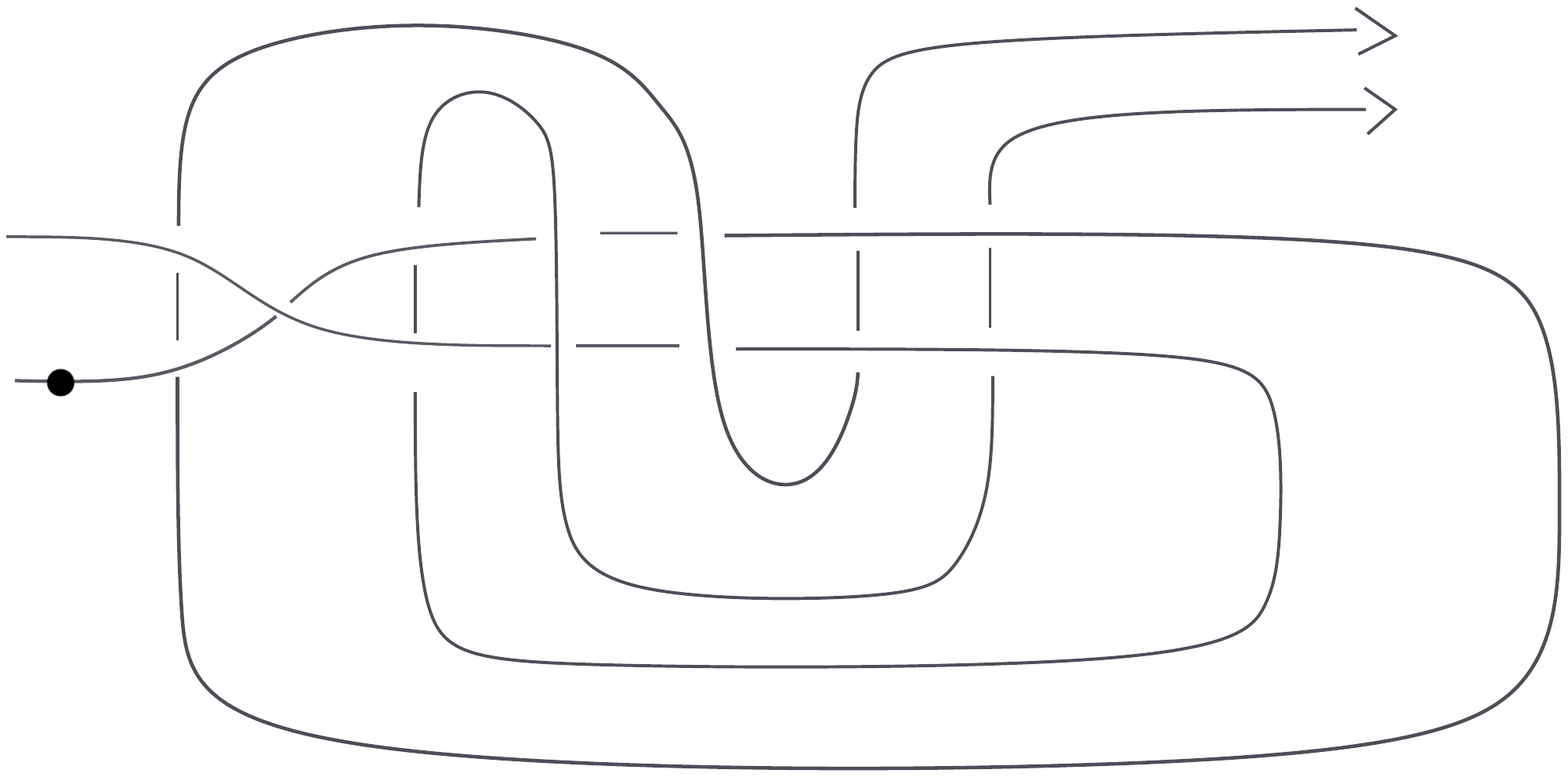}};
    \draw(-4,4.7) node {(2)};
    \draw[->,very thick] (-2.1,6) -- (-1.3,6) node[right] {};
    
    \node[inner sep=0pt] at (1,6) {\includegraphics[width = 0.35 \textwidth]{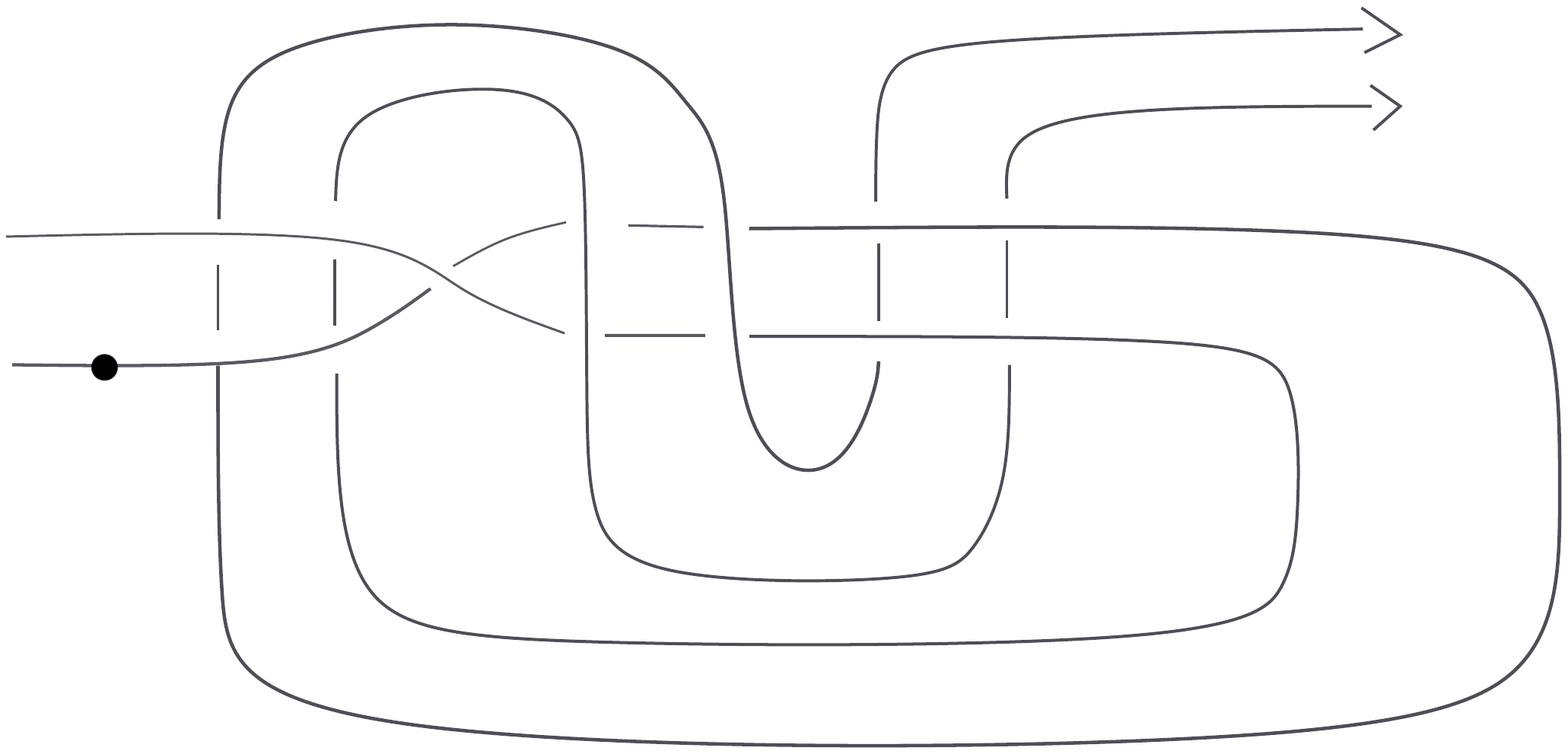}};
    \draw(0.6,4.7) node {(3)};
    \draw[->,very thick] (2.5,6) -- (3.2,6) node[right] {};
    
    \node[inner sep=0pt] at (5.5,6) {\includegraphics[width = 0.35 \textwidth]{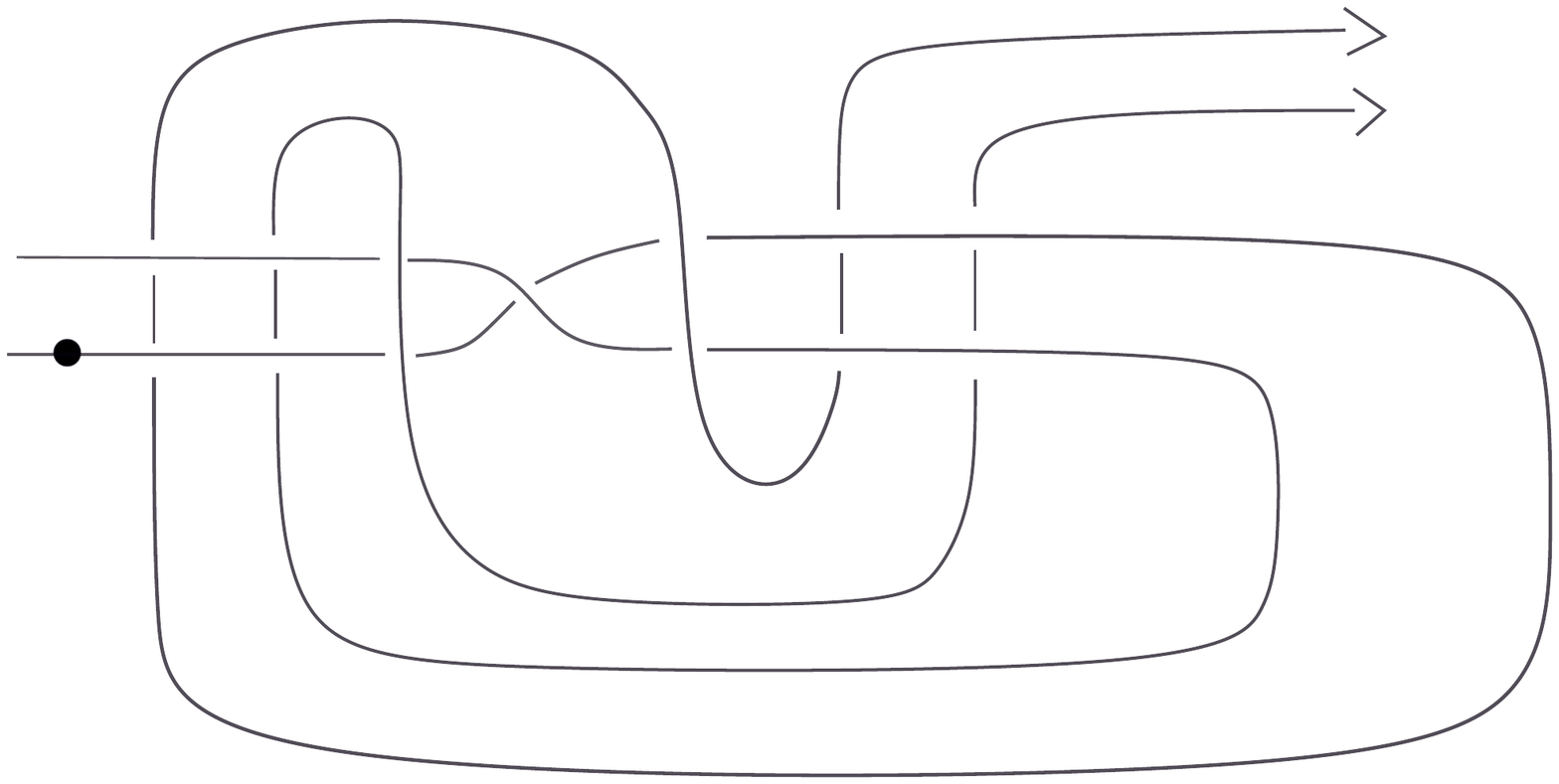}};
    \draw(5.3,4.7) node {(4)};
    \draw[->,very thick] (7.1,6) -- (7.7,6) node[right] {};
    \node[inner sep=0pt] at (-8,3) {\includegraphics[width = 0.35 \textwidth]{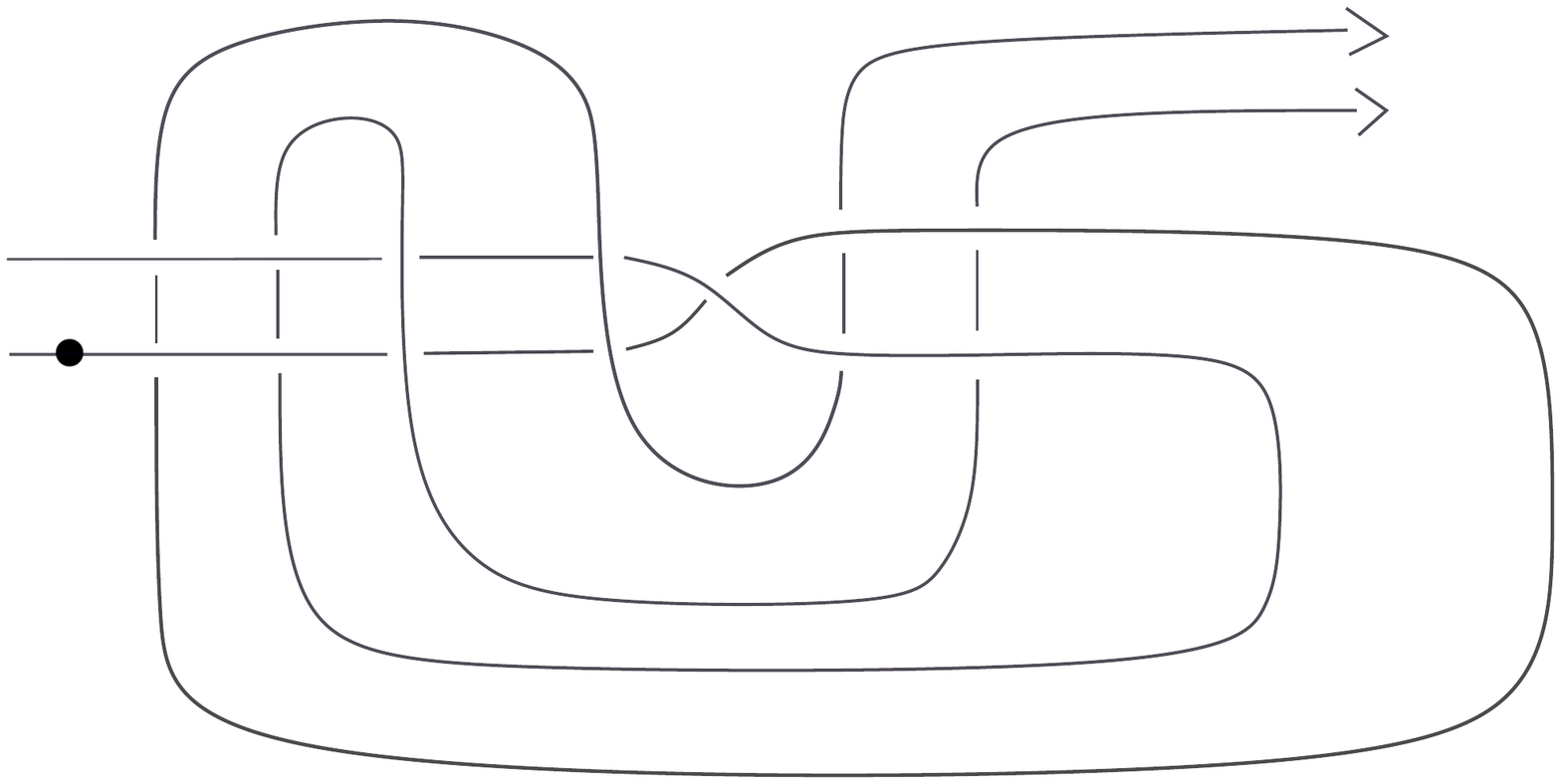}};
    \draw(-8.4,1.7) node {(5)};
    \draw[->,very thick] (-6.6,3) -- (-5.9,3) node[right] {};
    
    \node[inner sep=0pt] at (-3.6,3) {\includegraphics[width = 0.35 \textwidth]{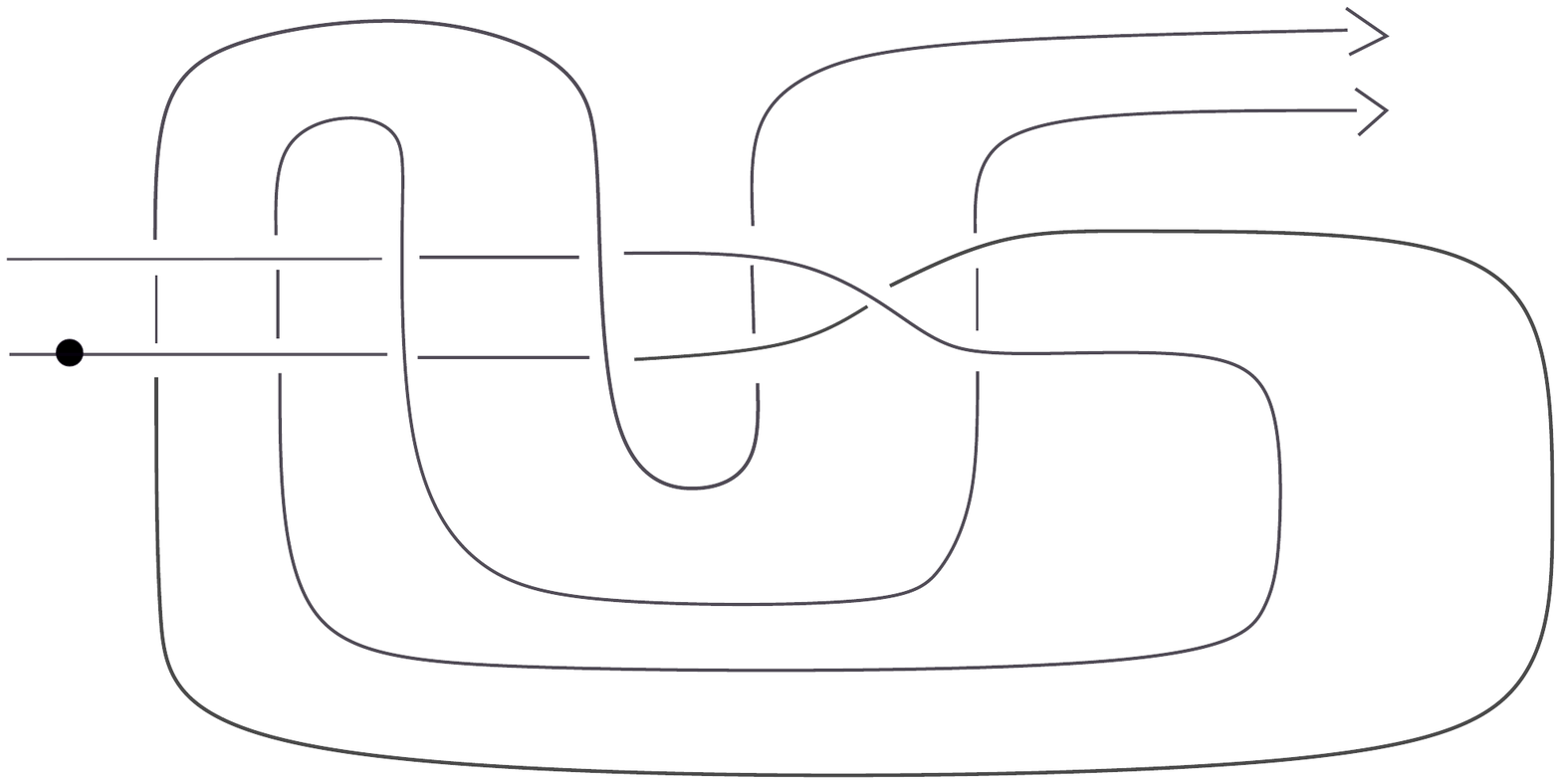}};
    \draw(-4,1.7) node {(6)};
    \draw[->,very thick] (-2.1,3) -- (-1.3,3) node[right] {};
    
    \node[inner sep=0pt] at (1,3) {\includegraphics[width = 0.35 \textwidth]{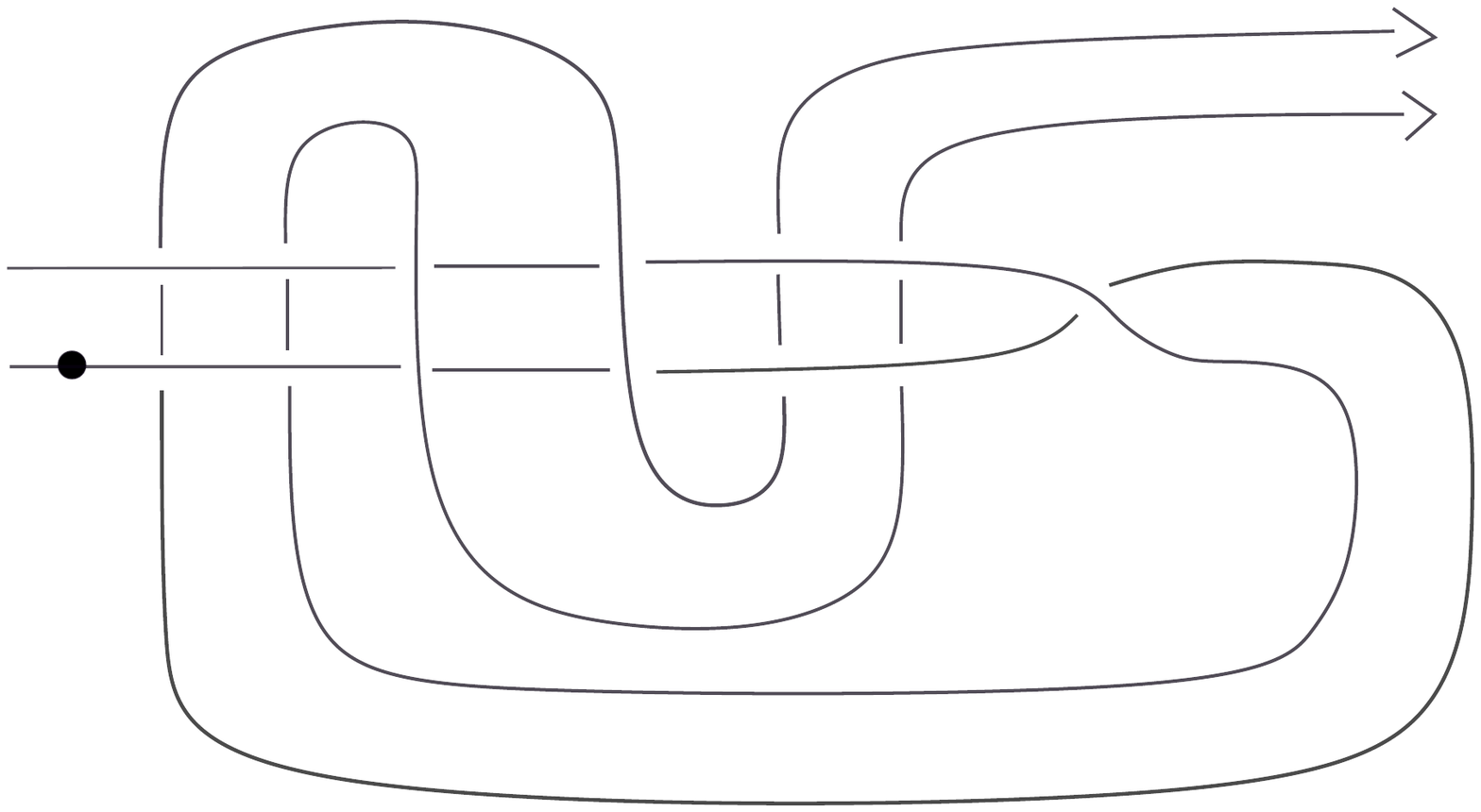}};
    \draw(0.6,1.7) node {(7)};
    \draw[->,very thick] (2.5,3) -- (3.2,3) node[right] {};
    
    \node[inner sep=0pt] at (5.7,3.2) {\includegraphics[width = 0.35 \textwidth]{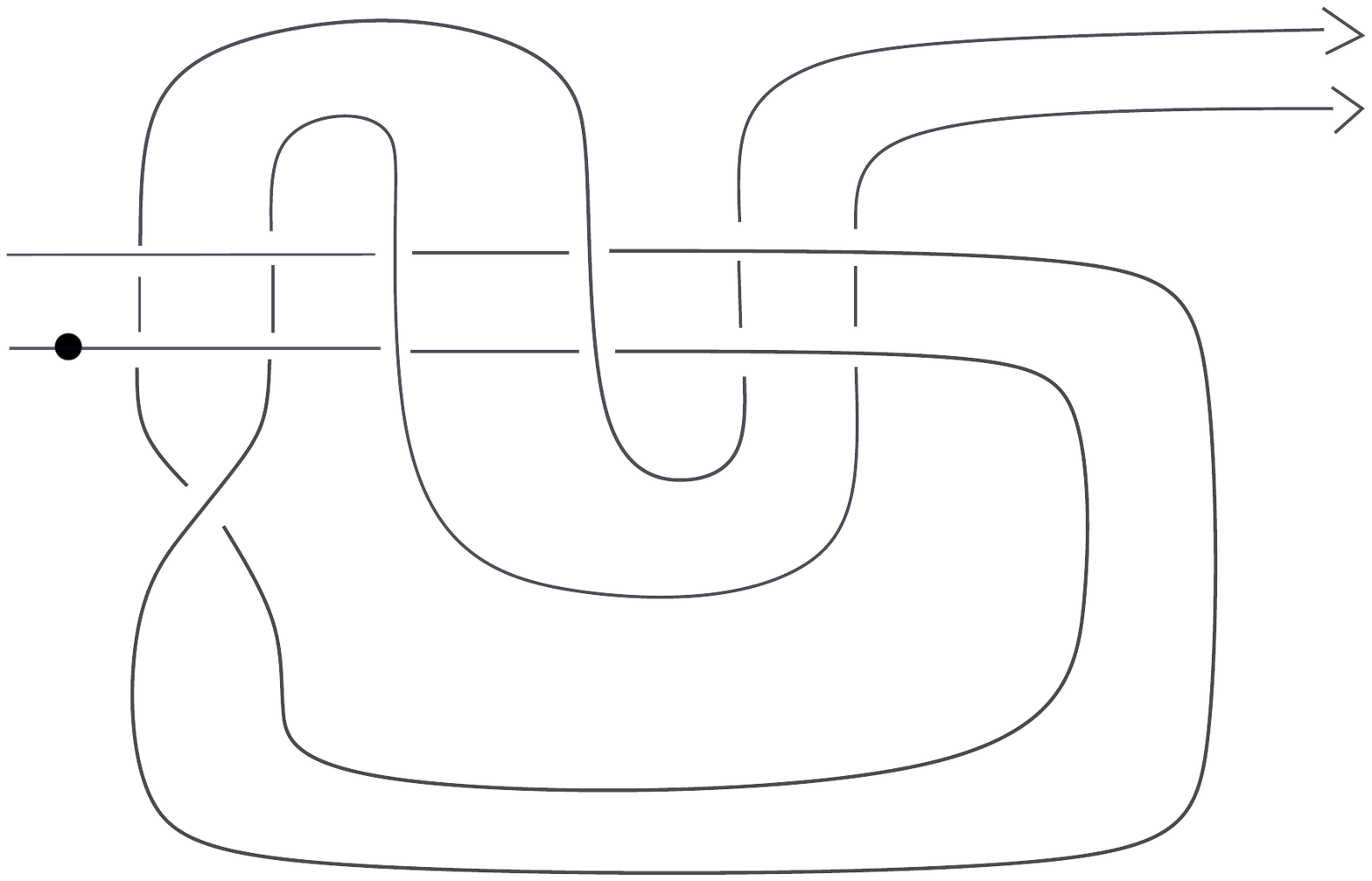}};
    \draw(5.3,1.7) node {(8)};
    \draw[->,very thick] (7.1,3) -- (7.7,3) node[right] {};
    \node[inner sep=0pt] at (-8,0) {\includegraphics[width = 0.35 \textwidth]{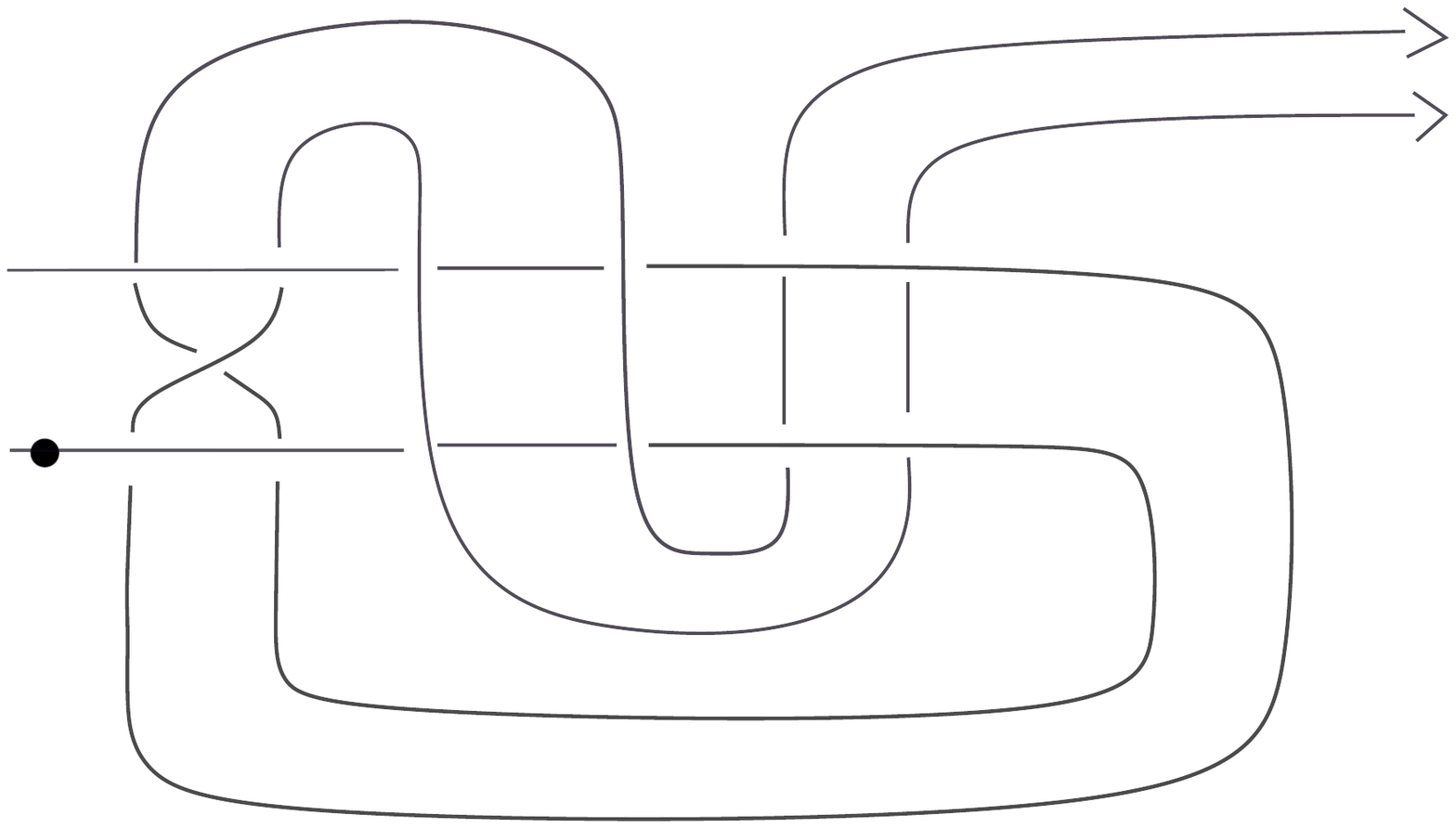}};
    \draw(-8.4,-1.3) node {(9)};
    \draw[->,very thick] (-6.8,0) -- (-6.1,0) node[right] {};
    
    \node[inner sep=0pt] at (-3.6,0) {\includegraphics[width = 0.35 \textwidth]{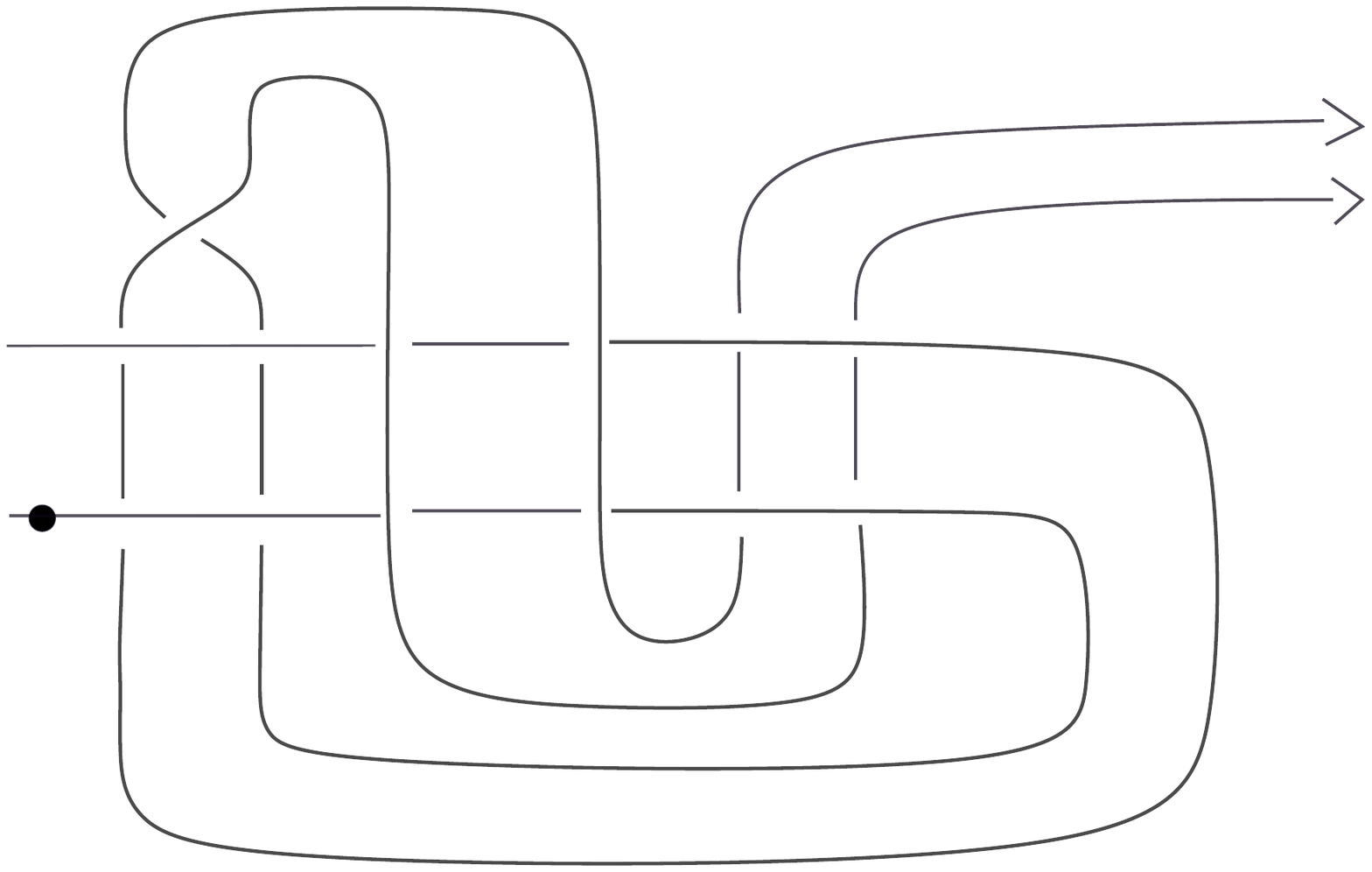}};
    \draw(-4,-1.3) node {(10)};
    \draw[->,very thick] (-2.3,0) -- (-1.6,0) node[right] {};
    
    \node[inner sep=0pt] at (1,0) {\includegraphics[width = 0.35 \textwidth]{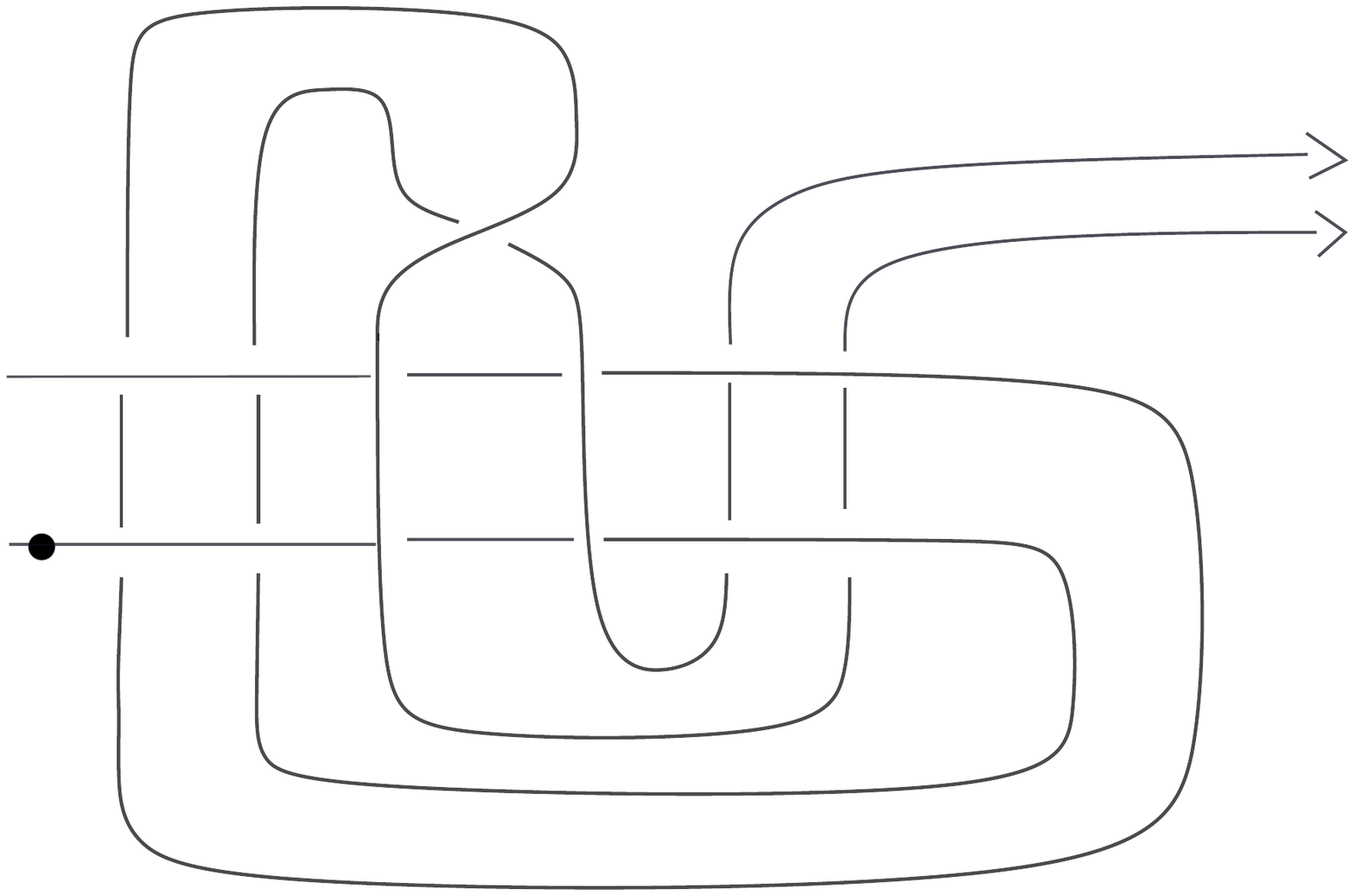}};
    \draw(0.6,-1.3) node {(11)};
    \draw[->,very thick] (2.3,0) -- (3.0,0) node[right] {};
    
    \node[inner sep=0pt] at (5.7,0) {\includegraphics[width = 0.35 \textwidth]{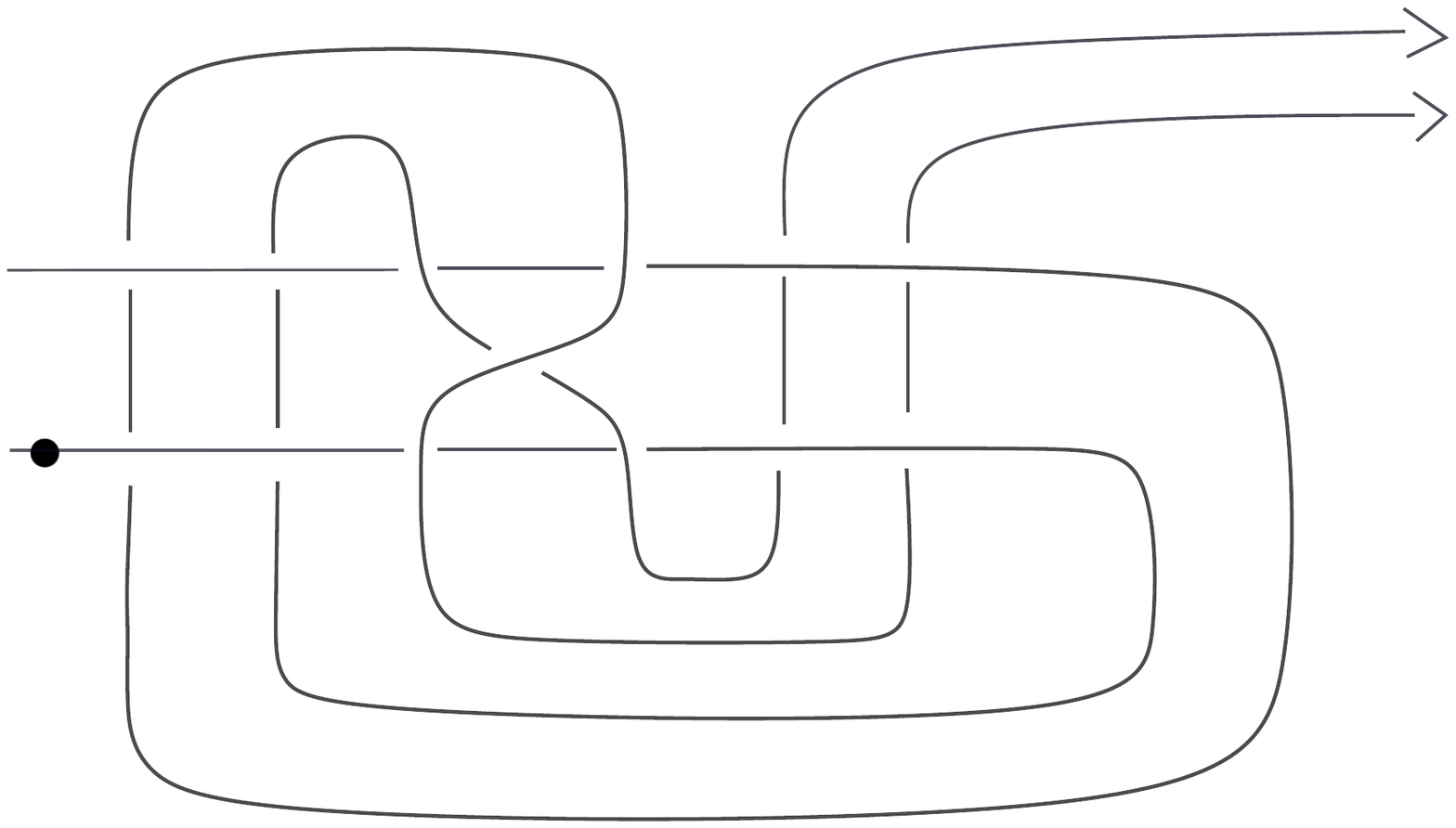}};
    \draw(5.3,-1.3) node {(12)};
    \draw[->,very thick] (7.1,0) -- (7.7,0) node[right] {};
    \node[inner sep=0pt] at (-8,-3) {\includegraphics[width = 0.35 \textwidth]{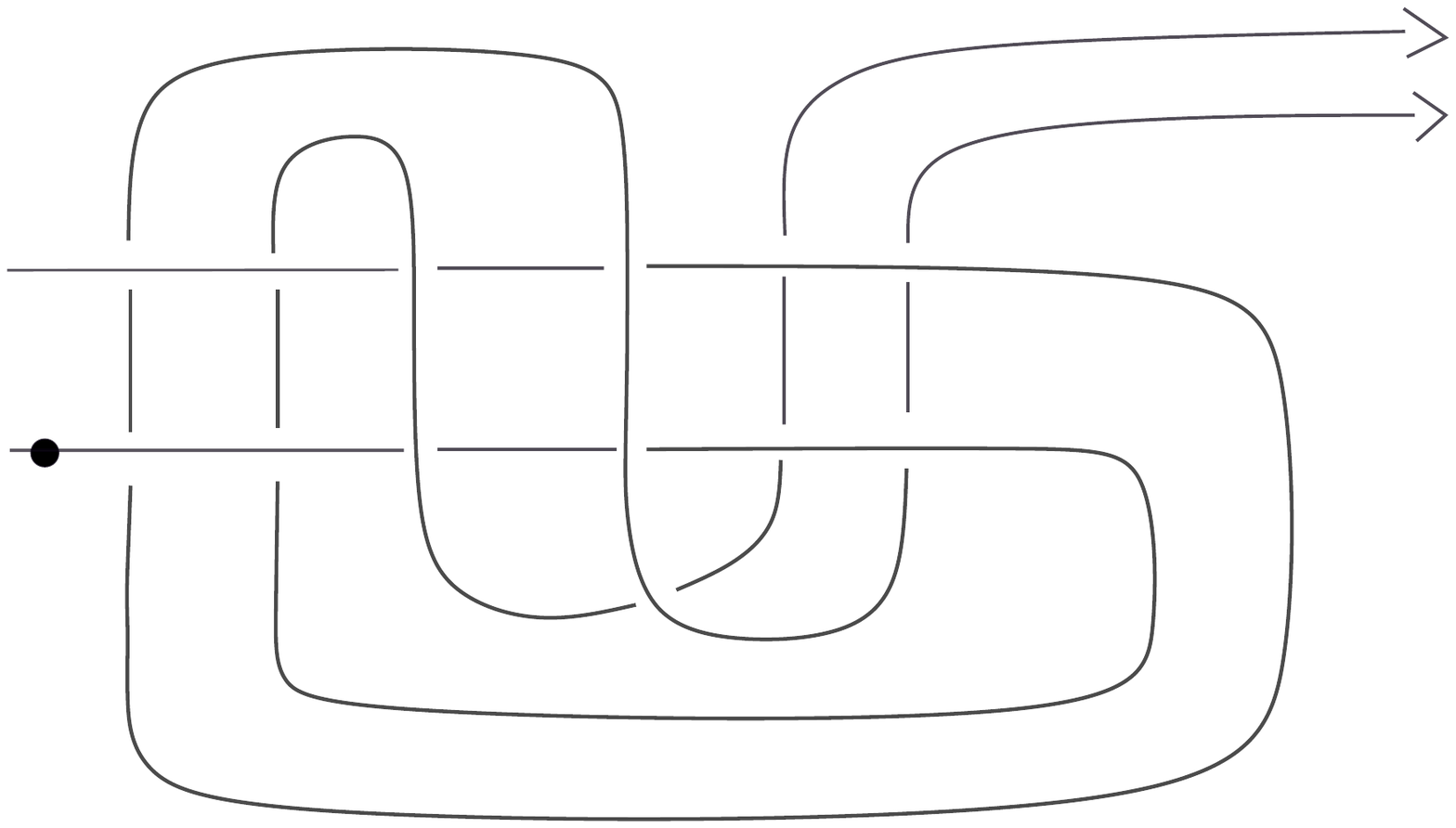}};
    \draw(-8.4,-4.3) node {(13)};
    \draw[->,very thick] (-6.8,-3) -- (-6.1,-3) node[right] {};
    
    \node[inner sep=0pt] at (-3.6,-3) {\includegraphics[width = 0.35 \textwidth]{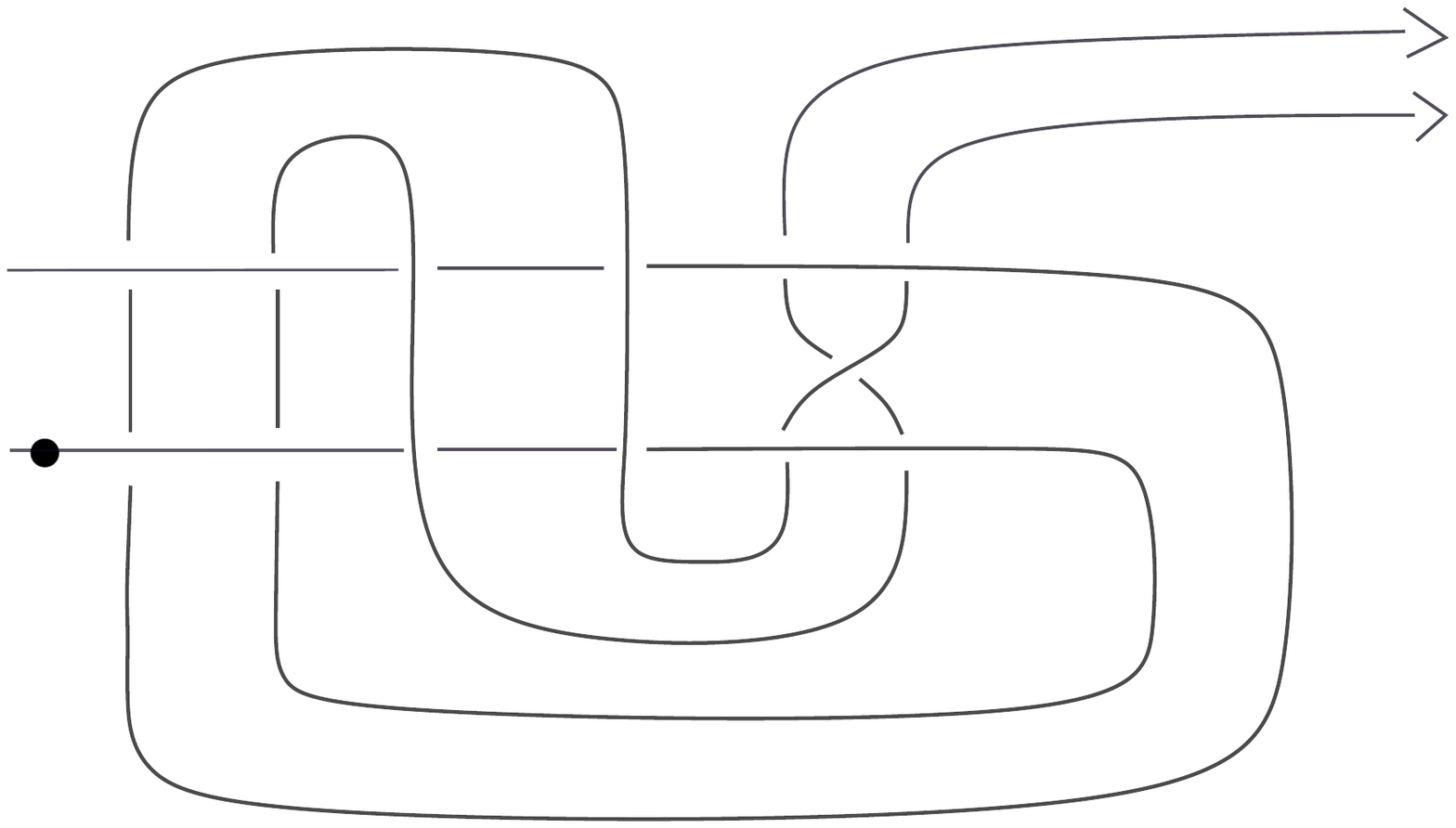}};
    \draw(-3.9,-4.3) node {(14)};
    \draw[->,very thick] (-2.3,-3) -- (-1.6,-3) node[right] {};
    
    \node[inner sep=0pt] at (1,-3) {\includegraphics[width = 0.35 \textwidth]{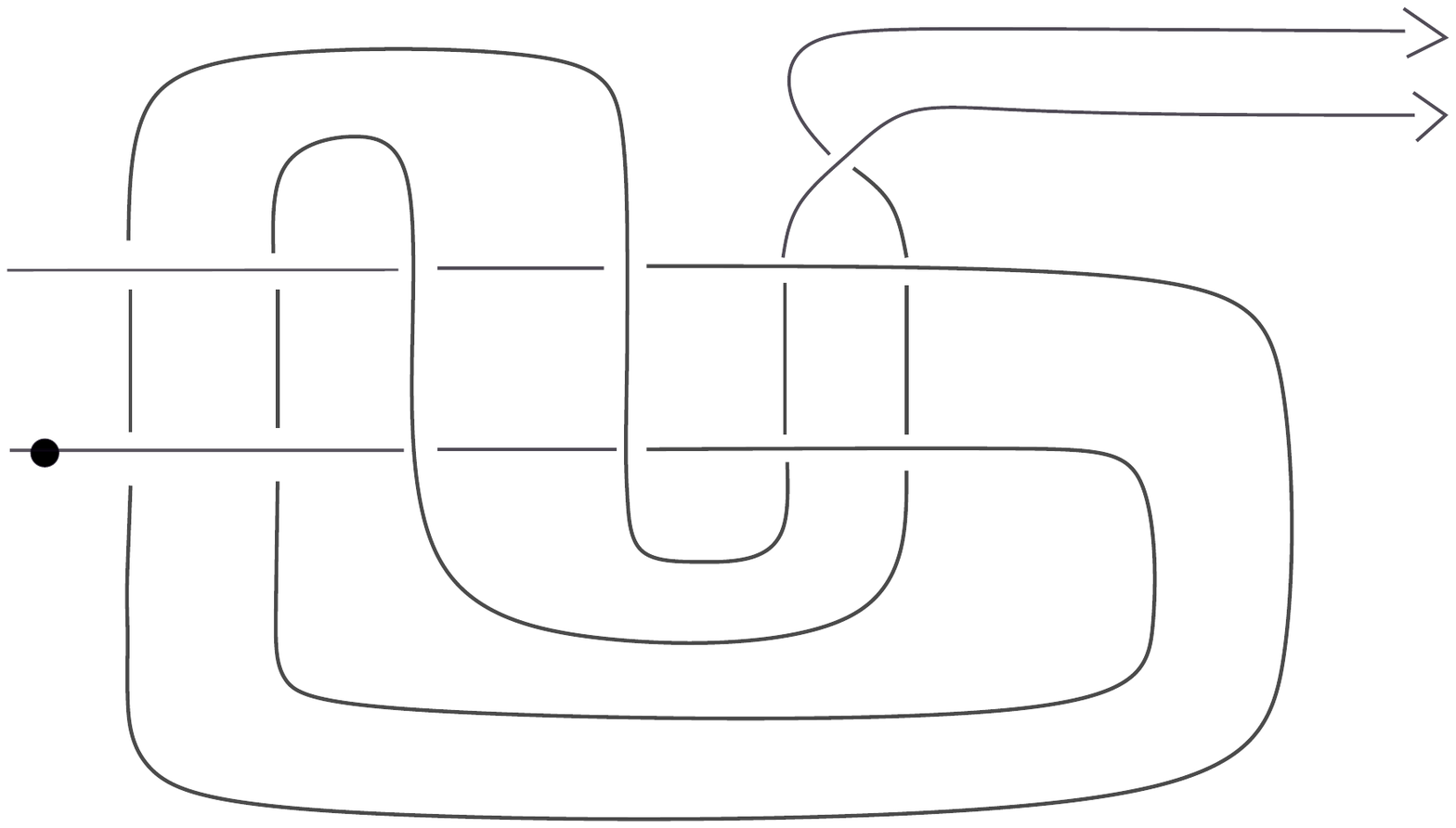}};
    \draw(0.6,-4.3) node {(15)};
    \draw[->,very thick] (2.3,-3) -- (3.0,-3) node[right] {};
    
    \node[inner sep=0pt] at (5,-3) {\includegraphics[width = 0.28 \textwidth]{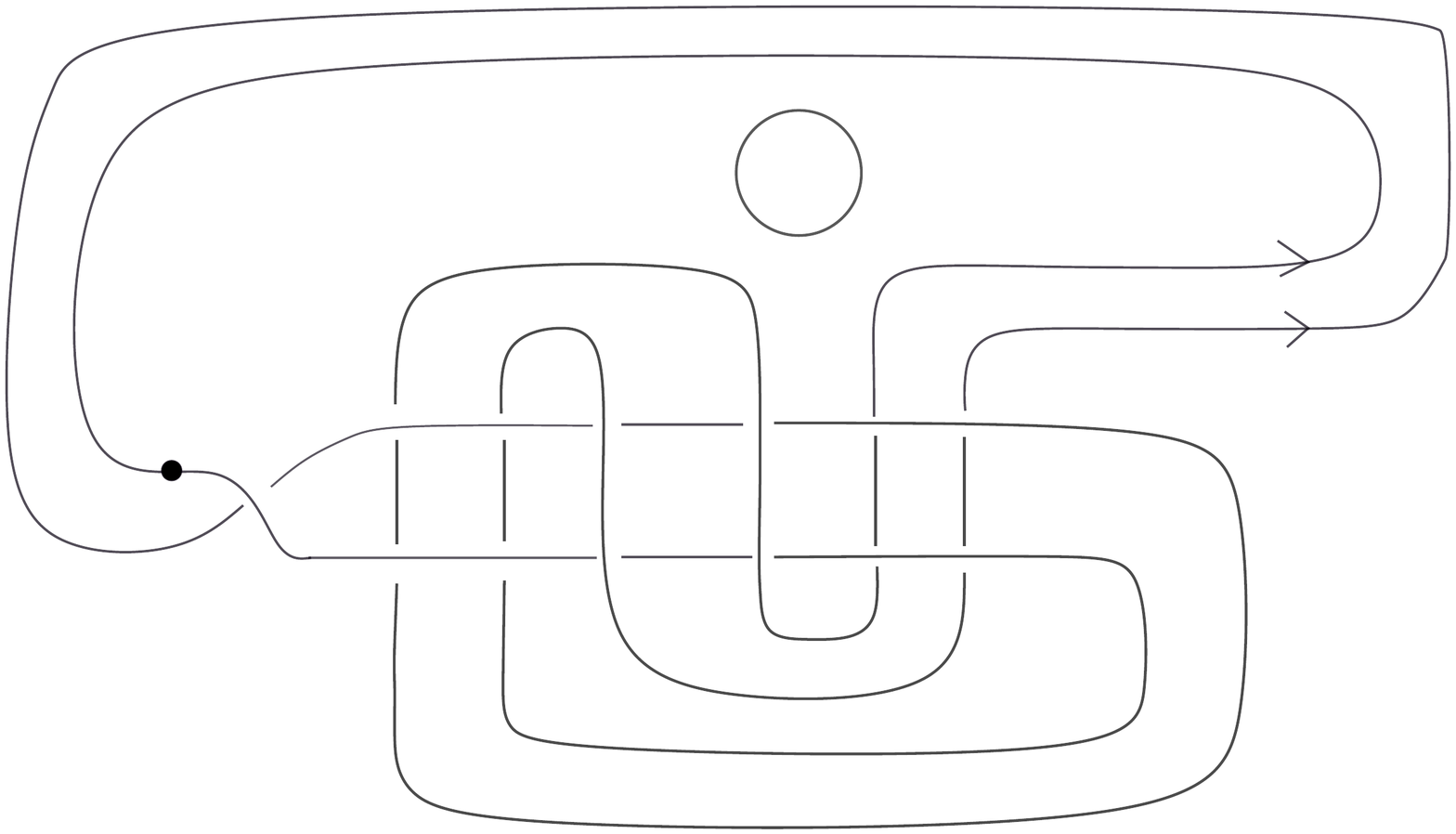}};
    \draw(5.5,-4.9) node {(16)};
    
    \draw[->,very thick] (-6.8,-6) -- (-6.1,-6) node[right] {};
    \draw(-5,-6) node {...};
    
    \draw[->,very thick] (-3.8,-6) -- (-3.1,-6) node[right] {};
    \node[inner sep=0pt] at (0,-6) {\includegraphics[width = 0.35 \textwidth]{2tre_1.pdf}};
    \draw(-0.3,-7.1) node {(1)};
    
  \end{tikzpicture}
  \caption{push loop}
  \label{Fig.push}
\end{figure}
\\
\\

\eject \pdfpagewidth=14in \pdfpageheight=7in
\clearpage

\end{document}